\documentclass[iccmp]{ipbook}

\RequirePackage{amsmath,amssymb,amsthm}
\RequirePackage{mathrsfs}
\RequirePackage{amscd}
\RequirePackage{xcolor}
\RequirePackage{hyperref}
\hypersetup{
  colorlinks,
  citecolor=red,
  filecolor=black,
  linktoc=all,
  linkcolor=blue,
  urlcolor=blue
}
\startlocaldefs
\theoremstyle{plain}
\newtheorem{theorem}{Theorem}[section]
\newtheorem{lemma}[theorem]{Lemma}
\newtheorem{proposition}[theorem]{Proposition}
\newtheorem{corollary}[theorem]{Corollary}

\newtheorem{problem}[theorem]{Problem}
\theoremstyle{definition}
\newtheorem{definition}[theorem]{Definition}

\theoremstyle{remark}
\newtheorem*{remark}{Remark}

\newcommand{\Aut}{\operatorname{Aut}}
\newcommand{\id}{\operatorname{id}}
\newcommand{\im}{\operatorname{Im}}
\newcommand{\Reg}{\operatorname{Reg}}
\newcommand{\hra}{\hookrightarrow}

\endlocaldefs

\firstpage{1}
\lastpage{54}

\title[Rigidity of $\Gamma$-equivariant holomorphic maps]{Rigidity of bounded $\Gamma$-equivariant holomorphic maps for $\pi_1$ of irreducible Shimura varieties of rank $\ge 2$ via K\"ahler geometry, harmonic analysis and ergodic theory}

\author{Ngaiming Mok}
\address{Department of Mathematics, The University of Hong Kong, Pokfulam, Hong Kong}
\email{nmok@hku.hk}
\thanks{The research of the first author is partially supported by a General Research Grant 17304321 of the HKRGC.}

\author{Kwok-Kin Wong}
\address{Shenzhen University, Guangdong 518060, P.~R.~China}
\email{kkwong@szu.edu.cn}

\begin{document}

\begin{abstract}
Let $\Omega$ be a bounded symmetric domain of rank $\ge 2$
and $\Gamma \subset \Aut(\Omega)$ be a torsion-free irreducible lattice,
so that $X_\Gamma :=\Omega/\Gamma$ carries a canonical structure
as a quasi-projective manifold.
Let $D \Subset \mathbb C^N$ be a bounded domain,
$\Gamma' \subset \Aut(D)$ be a torsion-free discrete subgroup such that
$Y_{\Gamma'} := D/\Gamma'$ is biholomorphic to a quasi-projective manifold.
Let $f: X_{\Gamma} \to Y_{\Gamma'}$ be a holomorphic map
which induces an isomorphism
$\Phi: \Gamma \overset{\cong}\longrightarrow \Gamma' \subset \Aut(D)$,
i.e., the lifting $F: \Omega \to D$ of $f: X_\Gamma \to Y_{\Gamma'}$
is a $\Phi$-equivariant holomorphic map,
where $\Phi = f_*: \pi_1(X_\Gamma)
= \Gamma \overset{\cong}\longrightarrow\Gamma' \subset \Aut(D)$ is the induced map, assumed to be an isomorphism, $\Phi$ satisfying $F(\gamma z) = \Phi(\gamma)(F(z))$ for any $\gamma \in \Gamma$.  In a recent article of the authors, we proved a result called the Isomorphism Theorem, which says that $F: \Omega \to D$ is necessarily a biholomorphism. The proof of the Isomorphism Theorem uses in essential ways K\"ahler geometry, function theory of several complex variables, harmonic analysis and ergodic theory.  Here we will focus on a slight variation of the Isomorphism Theorem in which $D$ is replaced by a simply connected complete K\"ahler-Einstein manifold $(M,h_M)$ which is moreover assumed to be Carath\'eodory hyperbolic (i.e., the infinitesimal complex Finsler pseudometric $\kappa_M$ induced from the space of bounded holomorphic maps into the Poincar\'e disk is a complex Finsler metric) and $\Gamma' \subset \Aut(M)$ is a torsion-free discrete subgroup such that the quotient manifold $Y_{\Gamma'} := M/\Gamma'$ is
of finite volume with respect to the quotient K\"ahler-Einstein metric.
In this setting, we will explain the essential roles played by K\"ahler geometry, harmonic analysis and ergodic theory in the proof of the Isomorphism Theorem.

\end{abstract}

\maketitle

\vskip 0.2cm
\section{Introduction}
\label{intro}

\vskip 0.2cm
\subsection{The Borel embedding and the Harish-Chandra realization}
\label{borel-hc}
\mbox{}

\vskip 0.5cm
\subsubsection{Bounded symmetric domains as Hermitian symmetric spaces of the noncompact type}
\label{herm-sym-noncompact}
Let $\Omega$ be an irreducible bounded symmetric domain. Writing $G_0 =\Aut_0(\Omega)$ for the identity component of the automorphism group $\Aut(\Omega)$, and $K \subset G_0$ for the isotropy subgroup at a reference point $o \in \Omega$, we have $\Omega = G_0/K=: X_0$ as a homogeneous manifold.  Here $K \subset G_0$ is a maximal compact subgroup and the irreducibility of $\Omega$ is equivalent to the statement that $G_0$ is (centerless and) simple.   With further details to be described below,
$X_0$ is an irreducible Riemannian symmetric space of the noncompact and semisimple type defined by a group automorphism $s$ of $G_0$ of order $2$ (given by $s(g) = s_1gs_1^{-1}$ for some $s_1 \in G_0$ of order 2) whose fixed point subgroup is $K$.  Denote by Gothic letters the Lie algebras corresponding to given Lie groups, so that $\mathfrak {g}_0$ is the Lie algebra of $G_0$, etc. (and writing also $\mathfrak g_0 = {\rm Lie}(G_0)$ to signify this, etc.). Since $s$ is a group isomorphism on $G_0$, its differential $ds =: \theta$ is a Lie algebra isomorphism of order 2 on the Lie algebra $\mathfrak g_0$ of left invariant smooth vector fields on $G_0$.
We have a Cartan decomposition $\mathfrak{g}_0 = \mathfrak{k} \oplus \mathfrak{m}$, where $\mathfrak{k} = {\rm Lie}(K)$ and at the same time the eigenspace of the Cartan involution $\theta$ corresponding to the eigenvalue $1$, while $\mathfrak m \subset \mathfrak{g}_0$ is the eigenspace of $\theta$ corresponding to the eigenvalue $-1$. By considering the action of the Cartan involution $\theta$ on $\mathfrak g_0$ and interpreting $\mathfrak k$ resp.\,$\mathfrak m$ as eigenspaces of $\theta$ corresponding to eigenvalues $1$ resp.\,$-1$, we have immediately the inclusions $[\mathfrak k, \mathfrak k] \subset \mathfrak k$, $[\mathfrak k,\mathfrak m] \subset \mathfrak m$ and $[\mathfrak m,\mathfrak m] \subset \mathfrak k$. The involution $s$ on $G_0$ with $K \subset G_0$ as fixed point subgroup induces a diffeomorphism $\sigma$ of $X_0 = G_0/K$ onto itself with an isolated fixed point at $o = eK \in G_0/K = X_0$.

\vskip 0.2cm
The homogeneous manifold $X_0$ is equipped with a $G_0$-invariant Riemannian metric
$g_0$, which is unique up to a multiplicative scalar
since $K$ acts irreducibly on the real tangent space $T^{\mathbb R}_o(X_0)$ of $X_0$
at $o = eK$ with the identification
$T^{\mathbb R}_o(X_0) = \mathfrak{g}_0/\mathfrak{k} \cong \mathfrak{m}$.
The involution $\sigma$ is a holomorphic isometry on the Riemannian manifold
$(X_0,g_0)$ fixing $o \in X_0$ as an isolated fixed point.
Since $G_0$ acts transitively on $X_0$, $\sigma$ defines on $(X_0,g_0)$
the structure of an irreducible Riemannian symmetric manifold of
the noncompact and semisimple type, hence of
nonpositive Riemannian sectional curvature and negative Ricci curvature.
The subgroup $K \subset G_0$ is reductive with a 1-dimensional center
$Z \cong \mathbb S^1$.
There is an element $z \in \mathfrak z$ which defines on $X_0$
an almost complex structure $J$. More precisely, at $o = eK$,
given $v \in  T^{\mathbb R}_o(X_0)$ and
any smooth vector field $\widetilde{v}$ on $X_0$
satisfying $\widetilde{v}(o) = v$, we check that
$[z, \widetilde{v}](o) \in T^{\mathbb R}_o(X_0)$ is independent of
the choice of $\widetilde{v}$ and define $J_ov := [z,\widetilde{v}](o)$.
From $z \in \mathfrak z$ it follows readily that $J_o$ is $K$-invariant,
and, for a suitable choice of $z\in\mathfrak z$ (so that ${\rm exp}(z)$
corresponds to multiplication by $\pm\sqrt{-1}$
on $\mathbb S^1 \subset \mathbb C^*$), we have
$J_o^2 = -\id_{T_o^{\mathbb R}(X_0)}$ and it defines the $G_0$-invariant almost complex structure $J$ on $X_0 = G_0/K$ via left translation by arbitrary elements $g\in G_0$.  The integrability of $J$ is checked from the invariance of the vanishing of the Nijenhuis tensor $N_J$ of $J$ given explicitly by
$$
N_{J}(\xi,\eta) = [J\xi, J\eta] - J[J\xi,\eta] - J[\xi,J\eta] - [\xi,\eta]
$$
for smooth vector fields $\xi$ and $\eta$ on $X_0$.
The vanishing of $N_J$ then follows from $d\sigma(o) = -\id_{T^{\mathbb R}_o(X_0)}$, yielding $N_J(\xi,\eta)(o) = -N_J(-\xi,-\eta)(o) = -N_J(\xi,\eta)(o)$, thus $N_J(\xi,\eta)(o) = 0$, and from the $G_0$-invariance of $J$ and hence of $N_J$.

\vskip 0.2cm
We note that at $o = eK \in X_0$, the Riemannian inner product on $T_o(X_0)$ given by $g_0(o)$ is $J_o$-invariant.  To see that, since $g_0(o)$ is $K$-invariant and $Z\subset K$, for $u, v \in T_o^{\mathbb R}(X_0)$, we have $g_0({\rm exp}(tz)u,{\rm exp}(tz)v) = g_0(u,v)$ for all $t\in\mathbb R$.  Differentiating against $t$ at $t = 0$, from $J_o(v) = [z,\widetilde{v}](o)$ we deduce that $g_0(Ju,v)+g_0(u,Jv) = 0$ (so that in particular $g_0(u,Ju) = 0$). It follows that $g_0(Ju,Jv) = -g_0(u,J^2v) = -g_0(u,-v) = g_0(u,v)$, showing that $g_0(Ju,Jv) = g_0(u,v)$, i.e., $g_0(o)$ is $J_o$-invariant. By homogeneity of $(X_0,g_0)$ under $G_0$, the Riemannian metric $g_0$ on $X_0$ is $J$-invariant, hence $(X_0,g_0)$ is a Hermitian manifold.  Since the linear action of $K$ on $T_o^{\mathbb R}(X_0)$ commutes with the action of the circle group $Z \subset K$, $K$ preserves $J_o$ and acts as a group of complex linear maps on $(T_o^{\mathbb R}(X_0),J_o)$, when $T_o^{\mathbb R}(X_0)$ is regarded as a complex vector space with the linear complex structure given by $J_o$.  If we define the Hermitian bilinear pairing $h_o$ on $(T_o^{\mathbb R}(X_0),J_o)$ by $h_o(u,v) = g_0(u,\overline{v})$, where complex conjugation $\overline{\cdot}$ on $T_o^{\mathbb R}(X_0)$ is defined by $J_o$, then $K$ acts as a group of complex linear isometries on $(T_o^{\mathbb R}(X_0),J_o;h_o)$. From the $G_0$-invariance of both $(X_0,J)$ and $(X_0,g_0)$, $G_0$ acts as a transitive group of holomorphic isometries on $(X_0,g_0)$. Hence, $(X_0,g_0)$ is an irreducible Hermitian symmetric manifold, which is in fact K\"ahler, since any odd-degree $G_0$-invariant differential form on $X_0$ vanishes as can be seen from the action of $d\sigma(o)$ on $T^{\mathbb R}_o(X_0)$.

\vskip 0.2cm
\subsubsection{A dual pair of Hermitian symmetric manifolds \texorpdfstring{$(X_0, X_c)$}{(X0, Xc)}}
\label{dual-pair}
Denote by $\mathfrak g = \mathfrak{g}_0\otimes_{\mathbb R}\mathbb C$ the complexification of $\mathfrak{g}_0$, and by $G$ the centerless simple complex Lie group such that $\mathfrak g = {\rm Lie}(G)$ and $G$ contains $G_0$ as a noncompact real form.  (Given a centerless real algebraic group $H_0$, the natural homomorphism from $H_0$ to its complexification $H$ as a complex algebraic group $H$ is necessarily injective, here and henceforth we will identify $H_0$ as a real algebraic subgroup of $H$.)
We have a compact real form $G_c \subset G$, where $G_c$ is a compact Lie subgroup with Lie subalgebra $\mathfrak{g}_c = \mathfrak{k} \oplus \sqrt{-1}\mathfrak{m} \subset \mathfrak{g}$. Writing $X_c := G_c/K$, there is a $G_c$-invariant Riemannian metric $g_c$ on $X_c$ unique up to a positive multiplicative scalar as $K$ acts irreducibly on $T^{\mathbb R}_o(X_c) = \sqrt{-1}\mathfrak{m}$.
Equivalently, $g_c$ is the same as the Riemannian metric on $X_c$ induced by $-\lambda B$, where $B(\cdot,\cdot)$ stands for the (negative definite) Killing form on the compact Lie algebra $\mathfrak{g}_c$, and $\lambda$ is some positive constant.
The pair $(X_0,g_0)$ and $(X_c,g_c)$ of Riemannian manifolds form a dual pair of irreducible Riemannian symmetric spaces of the semisimple type when we normalize the choices of $g_0$ and $g_c$ so that the mapping $T^{\mathbb R}_o(X_0) \cong \mathfrak{m} \mapsto \sqrt{-1}\mathfrak{m} \cong T^{\mathbb R}_o(X_c)$ defined by $v \mapsto \sqrt{-1}v$ is an isometry when the inner products on $\mathfrak{m}$ resp.~on $\sqrt{-1}\mathfrak{m}$ are given by $g_0$ resp.\,$g_c$. (Here and in the above we have an abuse of notation in that $o$ is used to denote the base point $eK$ of both $X_0 = G_0/K$ and $X_c = G_c/K$ which are {\it a priori} two distinct manifolds.)
Again, the 1-dimensional center $Z \cong \mathbb S^1$ of $K$ induces on $X_c$ the structure of an integrable almost complex structure, and the pair $(X_0,g_0)$ and $(X_c,g_c)$ of Hermitian manifolds is in fact a dual pair of irreducible Hermitian symmetric spaces of the noncompact resp.\,the compact type, so that $(X_0,g_0)$ and $(X_c,g_c)$ are K\"ahler manifolds.  They are in fact K\"ahler-Einstein by the irreducibility of the action of $K$ on $\mathfrak{m}$ resp.\,$\sqrt{-1}\mathfrak{m}$.

\vskip 0.2cm
\subsubsection{Embedding \texorpdfstring{$X_0 = G_0/K$}{X0 = G0/K} holomorphically as a domain on its compact dual \texorpdfstring{$X_c$}{Xc}}
\label{borel-embed}
Let $(X_0,g_0)$ be an irreducible Hermitian symmetric space of the noncompact type and write $(X_c,g_c)$ for its dual Hermitian symmetric space of the compact type.  $G_c$ acts as a transitive group of holomorphic isometries on $(X_c,g_c)$.  The simple complex Lie group $G$, which is also a complexification of $G_c$, acts as a group of biholomorphisms on $X_c$.
Denote by $P \subset G$ the (maximal) parabolic subgroup consisting of biholomorphisms fixing $o = eK$, so that $X_c = G/P$ as a complex homogeneous manifold. In what follows we will write $X := G/P$ instead of $X_c$ when we forget about the underlying K\"ahler-Einstein metric $g_c$.  Since $(X_c,g_c)$ is a compact K\"ahler manifold of positive Ricci curvature, the anti-canonical line bundle $K_X^{-1}$ is ample, and $X$ is a projective manifold by the Kodaira embedding theorem.  In fact, we have ${\rm Pic}(X) \cong \mathbb Z$, and,
denoting by $\mathcal O(1)$ its positive generator,
it defines the minimal projective embedding $\nu: X \hra \mathbb P(\Gamma(X,\mathcal O(1))^*)$.  For instance, for the Grassmann manifold $X := {\rm Gr}(k,\mathbb C^{\ell})$ of $k$-planes on an $\ell$-dimensional complex vector space, $0< k <\ell$, this gives the Pl\"ucker embedding of $X$ into some projective space $\mathbb P^N$.

\vskip 0.2cm
Since $G$ acts as a complex Lie group acting transitively on $X$, we may interpret $\mathfrak{g} =  {\rm Lie}(G)$ as the simple complex Lie algebra consisting of global holomorphic vector fields on the projective manifold $X$.  Hence $z \in \mathfrak{z} \subset \mathfrak{k} \subset \mathfrak{g}$ as a holomorphic vector field on $X$.  As such $z$ vanishes at $o = eP$.  For example, in the case where $X = \mathbb P^n$, $z$ can be chosen as a multiple of the Euler vector field $\sum_{i=1}^n z_i\frac{\partial}{\partial z_i}$ in the inhomogeneous coordinates $(z_1,\cdots,z_n)$ on the open Schubert cell $\mathbb C^n \subset \mathbb P^n$ when $o = eP$ is identified as the origin of $\mathbb C^n$.

\vskip 0.2cm
We have $\mathfrak{g_0} = \mathfrak{k}\oplus\mathfrak{m}$.  Writing the complexification of a real vector space $V$ as $V^\mathbb C := V\otimes_{\mathbb R}\mathbb C$, and noting that $\mathfrak{g} = \mathfrak{g}_0^{\mathbb C}$, we have
$\mathfrak{g} = \mathfrak k^{\mathbb C}\oplus\mathfrak{m}^{\mathbb C}$.
From $J_o^2 = -\id_{T_o^{\mathbb R}(X_0)}$ and identifying $T_o^{\mathbb R}(X_0)$ as $\mathfrak{m}$ and hence $T_o^{\mathbb C}(X_0)$ as $\mathfrak{m}^{\mathbb C}$, we have $\left({\rm ad}(z)|_\mathfrak{m}\right)^2 \equiv -\id_{\mathfrak{m}^{\mathbb C}}$ and thus the eigenspace decomposition $\mathfrak{m}^{\mathbb C} = \mathfrak{m}^{+} \oplus\mathfrak{m}^{-}$ corresponding to the eigenvalues $\sqrt{-1}$ resp.\,$-\sqrt{-1}$. Here $\mathfrak m^{\mathbb C}\subset \mathfrak g$ is the complexification of $\mathfrak m \subset \mathfrak g_0$, and with respect to conjugation $\tau$ on $\mathfrak g$ treating $\mathfrak g_0$ as the real part, $\mathfrak m^+$ and $\mathfrak m^{-}$ are conjugate to each other, i.e., $\mathfrak m^{-} = \tau(\mathfrak m^{+})$, so that $\dim_{\mathbb C}\mathfrak m^{+} = \dim_{\mathbb C}\mathfrak m^{-} = \frac{1}{2}\dim_{\mathbb R}\mathfrak m$.  Regarding ${\rm ad}(z)$ as a linear operator on $\mathfrak{g}$, we have an eigenspace decomposition
$$
\mathfrak{g} = \mathfrak m^{+} \oplus \mathfrak k^{\mathbb C} \oplus \mathfrak m^{-}
$$
of $\mathfrak{g}$ into a direct sum of eigenspaces of ${\rm ad}(z)$ corresponding to the eigenvalues $\sqrt{-1}$ resp.\,0 resp.\,$-\sqrt{-1}$.
Here $\mathfrak{m}^{+},\mathfrak{m}^{-} \subset \mathfrak{g}$  are abelian subalgebras. For $\mathfrak{p} = {\rm Lie}(P)$, we have  $\mathfrak{p} = \mathfrak k^{\mathbb C}\oplus\mathfrak m^{-}$, where $\mathfrak m^{-} \subset \mathfrak{p}$ is the nilpotent radical.  Correspondingly we have the Levi decomposition $P = K^{\mathbb C}M^{-}$, where $K^{\mathbb C} \subset G$ is a complexification of $K \subset G_0\subset G$ in $G$ and $M^{-} \subset P$ is the unipotent radical. Interpreting $\mathfrak g$ as the complex Lie algebra of global holomorphic vector fields on $X$, $\mathfrak{p} \subset \mathfrak g$ consists of all holomorphic vector fields vanishing at $o \in X$, and $\mathfrak{m}^{-} \subset \mathfrak p$ is the subspace of holomorphic vector fields vanishing at $o$ to the order $\ge 2$. (Actually for $m^{-} \in \mathfrak m^{-}$, either $m^{-}$ vanishes at $o$ precisely to the order $2$ or $m^- = 0$.)

\vskip 0.2cm
Relating $X_0$ and $X$ we have the Borel embedding theorem, as follows, cf. \cite{Wo1972}.

\begin{theorem} \
Consider the smooth mapping $\xi: X_0 = G_0/K \longrightarrow G/P = X$, which is defined because $K \subset P$. Then, $\xi$ embeds $X_0$ biholomorphically onto an open subset of $X$.
\end{theorem}

\begin{proof}
At $o = eK$ we have $T_o^{\mathbb R}(X_0) \cong \mathfrak g_0/\mathfrak k\cong \mathfrak m$, while $T^{\mathbb R}_o(G/P) \cong \mathfrak g/\mathfrak p \cong \mathfrak m^{+}$ as real vector spaces. We have $\dim_{\mathbb R}\mathfrak m^+ = 2\dim_{\mathbb C}\mathfrak m^{+} = \dim_{\mathbb R}\mathfrak m$.  On the other hand, at the infinitesimal level the given identification $d\xi(o): \mathfrak m \to \mathfrak m^{+}$ is induced from the composition of linear maps $\mathfrak m \hra \mathfrak g_0 \hra \mathfrak g = \mathfrak m^{+}\oplus\mathfrak k^{\mathbb C}\oplus\mathfrak m^{-} \to \mathfrak m^+$.  Thus, the kernel of $d\xi(o)$ is given by $\mathfrak m\cap (\mathfrak k^{\mathbb C}\oplus\mathfrak m^{-})$. Taking imaginary parts with respect to the real structure given by $\tau$, for $v \in \mathfrak m\cap (\mathfrak k^{\mathbb C}\oplus\mathfrak m^{-}$) we have $v = (k_1 + ik_2) + m^{-}$, where $k_1, k_2 \in \mathfrak k$ and $m^{-} \in \mathfrak m^{-}$.
Since $v \in \mathfrak m$ we have $0 =\im(v) = k_2 + \im(m^{-})$ and $v = k_1 + {\rm Re}(m^{-})$.
From $\mathfrak k \cap \mathfrak m = 0$ it follows that $k_2 = \im(m^{-}) = 0$, so that $m^{-}$ is a real point of $\mathfrak m$.
Since $\overline{\mathfrak{m}^{-}}=\mathfrak{m}^+$ and $\mathfrak{m}^-\cap \mathfrak{m}^+=0$,
the only real point on $\mathfrak{m}^-$ is $0$.
Hence $m^{-} = 0$ and $v = k_1 \in\mathfrak m\cap\mathfrak k = 0$, proving that ${\rm Ker}(d\xi(o)) = 0$, implying that $\xi: X_0 = G_0/K \to G/P$ is a local diffeomorphism at $o = eK$ and hence a local diffeomorphism everywhere on $X_0$ by the $G_0$-equivariance of $\xi$.

\vskip 0.2cm
We assert that $\xi: X_0 = G_0/K \to G/P = X$ is a holomorphic map.  It suffices to check this at $o = eK$, and in what follows $J$ will stand for $J_o$.  For $v \in T_o^{\mathbb R}(X_0) \equiv \mathfrak m$, $d\xi(v) = v^{+} \in \mathfrak m^{+}$
is determined by the decomposition $v = v^+ + v^-$ in accordance with $v \in \mathfrak m \subset \mathfrak m^{\mathbb C} = \mathfrak m^+ \oplus\mathfrak m^-$. Now, we have $v = \frac{1}{\phantom{,}2\phantom{,}}(v-\sqrt{-1}Jv)+\frac{1}{\phantom{,}2\phantom{,}}(v+\sqrt{-1}Jv)$. Then, $d\xi(v) = v^+ = \frac{1}{\phantom{,}2\phantom{,}}(v-\sqrt{-1}Jv)$, while
$d\xi(Jv) = \frac{1}{\phantom{,}2\phantom{,}}(Jv-\sqrt{-1}J^2v) = \frac{1}{\phantom{,}2\phantom{,}}(Jv+\sqrt{-1}v)$, since $J^2v = -v$. Hence, $d\xi(Jv) = \frac{\sqrt{-1}}{2}(v-\sqrt{-1}Jv) = \sqrt{-1}v^+ = \sqrt{-1}d\xi(v)$. Consider the underlying smooth real manifold $G/P$, we have $T^{\mathbb R}_o(G/P) \cong\mathfrak m^+$, where $\mathfrak m^+$ is regarded as a real $2n$-dimensional vector space, $n = \dim_{\mathbb C}(G/P)$. Since the linear almost complex structure $J'$ on $\mathfrak m^+$ as a real vector space agrees with multiplication by $\sqrt{-1}$ on $\mathfrak m^+$ regarded as a complex vector space, $d\xi(Jv) = J'(d\xi(v))$, hence $\xi: X_0 \to X$ is locally biholomorphic at $o$ and hence everywhere on $X_0 = G_0/K$. (We note that the identification of $T_o^{\mathbb R}(X)$ with $T^{1,0}_o(X) \cong \mathfrak m^+$ is given by $2{\rm Re}(\eta) \leftrightarrow \eta$ for $\eta \in \mathfrak m^+$, in the same way that we have the identification $\frac{\partial}{\partial x} \leftrightarrow \frac{\partial}{\partial z}$ on the complex plane.)

\vskip 0.2cm
To complete the proof of the Borel embedding theorem, it remains to show that $\xi: X_0 = G_0/K \to G/P = X$ is injective, i.e., to prove that $G_0\cap P = K$. We have $\mathfrak g_0 \cap \mathfrak p = (\mathfrak k + \mathfrak m)\cap (\mathfrak k^{\mathbb C}+\mathfrak m^-) = \mathfrak k + (\mathfrak m\cap\mathfrak m^-) = \mathfrak k$.  Hence, writing $K' :=G_0\cap P$, we have $K'\supset K$ and $\dim_{\mathbb R}K' = \dim_{\mathbb R}K$. We proceed to prove $K' = K$ by contradiction.  Suppose there exists $g \in K'-K$.  Then, $gKg^{-1} \subset K'$ is the isotropy subgroup at $g(o) \in X_0$, $g(o) \neq o$.  For each $k \in K$, $kgKg^{-1}k^{-1} \subset K'$ is the stabilizer of $k(g(o)) \in X_0$.  Since $o \in X_0$ is the unique point fixed by $K$, $K'$ contains a distinct positive-dimensional family of maximal compacts $K_{k(g(o))}$ fixing $k(g(o))$ as $k(g(o))$ varies over the $K$-orbit of $g(o)$, and it follows that $\dim_{\mathbb R}K' > \dim_{\mathbb R}K$, a plain contradiction.  We conclude that in fact $K' = G_0\cap P = K$, proving that $\xi: X_0 = G_0/K \longrightarrow G/P = X$ is injective, completing the proof of the Borel embedding theorem.
\end{proof}

\vskip 0.2cm
\subsubsection{The Harish-Chandra realization of \texorpdfstring{$X_0 = G_0/K$}{X0 = G0/K} as a bounded domain}
An irreducible Hermitian symmetric space of the noncompact type $X_0 = G_0/K$ in the notation in previous paragraphs admits a holomorphic embedding in a canonical way into its dual compact Hermitian symmetric space of the compact type $X = G/P$ by the Borel embedding.  By a theorem of Harish-Chandra, writing $n := \dim_{\mathbb C}X_0$, we have in fact $X_0 \cong \Omega \Subset \mathbb C^n \subset \widehat{\Omega} := X_c = G/P$, and we call $\Omega$ the Harish-Chandra realization of $X_0$.

\vskip 0.3cm
Essential for the understanding of $(X_0,g_0)$, $X_0 = G_0/K$ is the Polydisk theorem.
A crucial method for the study of $X_0 = G_0/K$ is the root space decomposition of the simple Lie algebra $\mathfrak g$ with respect to a
Cartan subalgebra $\mathfrak{h}$. Write $n$ for the complex dimension of $X_0$ and let $r = r(X_0)$ be the rank of $(X_0,g_0)$ as a Riemannian symmetric space of the noncompact type.  $(X_0,g_0)$ is a simply connected complete Riemannian manifold of nonpositive sectional curvature and hence a Cartan-Hadamard manifold.  Hence, writing ${\rm exp}: T^{\mathbb R}_o(X_0) \to X_0$ for the exponential map in the sense of Riemannian geometry (i.e., endowing $T^{\mathbb R}_o(X_0) \cong \mathbb R^{2n}$ with
``{\it polar coordinates\/}" so that lines passing through
the origin $T^{\mathbb R}_o(X_0)$ correspond to geodesics emanating from $o$),
${\rm exp}: T^{\mathbb R}_o(X_0) \to X_0$ is a diffeomorphism.
Identifying $T^{\mathbb R}_o(X_0) = T^{\mathbb R}(G_0/K) \cong \mathfrak m \subset \mathfrak g_0$, there exists an $r$-dimensional vector subspace $\mathfrak a \subset \mathfrak m$ which is a maximal abelian subalgebra, and ${\rm exp}|_{\mathfrak a}:\mathfrak a \to X_0$ is an isometry mapping $\mathfrak a$ onto a totally geodesic flat subspace $A \subset X_0$ of maximal dimension.

\vskip 0.2cm
So far we have not made use of the complex structure of $(X_0,g_0)$.  Using root space decomposition of $\mathfrak g$ with respect to $\mathfrak h$, one can prove that $A$ can be complexified to a totally geodesic $r$-dimensional complex submanifold $\Pi \subset X_0$ so that $(\Pi,g_0|_{\Pi})$ is biholomorphically isometric to a Cartesian product of $r$ identical copies of the Poincar\'e disk $(\Delta,ds_{\Delta}^2)$, i.e., the Gaussian curvature of each Cartesian factor is the same negative constant. Thus, the automorphism group of $\Pi$ preserves $g_0|_{\Pi}$ even when factors are permuted, and it is a nontrivial fact that all such automorphisms of $\Pi$ extend to automorphisms of $X_0$ (which are holomorphic isometries with respect to $g_0$).
The Polydisk theorem says that for any $r$-dimensional totally geodesic complex submanifold $\Pi \subset X_0$ passing through $o \in X_0$, $X_0$ is the union of translates $k\Pi \subset X_0$ as $k$ ranges over the isotropy subgroup $K \subset G_0$ of $(X_0,g_0)$ at $o \in X_0$.

\vskip 0.2cm
We proceed now to formulate and prove the Harish-Chandra embedding theorem.
Recall that $X = G/P$ and $X$ embeds as a projective submanifold of
$\mathbb P(\Gamma(X,\mathcal O(1))^*) =: \mathbb P^N$.
Then, $G$ acts as a group of projective linear transformations
on $\mathbb P^N$ preserving $X$.
Hence, $G \subset \mathbb P{\rm GL}(N+1,\mathbb C)$ is a complex algebraic group.
$P \subset G$ is the isotropy subgroup at $o$, so that $g\in P$ if and only if
$g\cdot o=o$. Hence $P$ is an algebraic subgroup.

\vskip 0.2cm
Recall the Harish-Chandra decomposition
$\mathfrak g = \mathfrak m^+ \oplus \mathfrak k^{\mathbb C} \oplus \mathfrak m^-$,
$\mathfrak p = \mathfrak k^{\mathbb C} \oplus \mathfrak m^- \subset \mathfrak g$
being a maximal parabolic subalgebra and
$\mathfrak m^- \subset\mathfrak p$ being the nilpotent radical.
Correspondingly we have the Levi decomposition $P=K^\mathbb{C}M^-$,
where $M^-=\exp(\mathfrak{m}^-)\subset P$ is the unipotent radical.
To see this, we note that, choosing a basis of $\mathfrak{g}$
consisting of any basis of the Cartan subalgebra $\mathfrak{h}$
and root vectors of $\mathfrak{g}$ with respect to $\mathfrak{h}$,
for any $x\in \mathfrak{m}^-$, $\exp({\rm ad}(x))$ is expressed as
a unipotent matrix in terms of the chosen basis,
and $\exp\vert_{\mathfrak{m}^-}: \mathfrak{m}^-\rightarrow G$
is a biregular homomorphism onto an algebraic subgroup $M^-=\exp(\mathfrak{m}^-)$.
Regarding $\mathfrak g = \mathfrak g_0\otimes_{\mathbb R}\mathbb C$
as the complexification of $\mathfrak g_0$,
we had the conjugation $\tau = \overline{\cdot}$ on $\mathfrak g$
with respect to which $\mathfrak g_0$ is the real part.
Taking conjugation with respect to $\tau$,
the Harish-Chandra decomposition transforms into
$\mathfrak g = \mathfrak m^- \oplus \mathfrak k^{\mathbb C} \oplus \mathfrak m^+$.
Thus $\mathfrak p':= \mathfrak k^{\mathbb C} \oplus \mathfrak m^+ \subset \mathfrak g$
is also a maximal parabolic subalgebra with $\mathfrak m^+ \subset \mathfrak p'$
being the nilpotent radical, and we have analogously the algebraic subgroup $M^+ = {\rm exp}(\mathfrak m^+)$.  We have now the Harish-Chandra embedding theorem, as follows.

\begin{theorem} \
Consider the mapping $\mu: M^+\times K^{\mathbb C}\times M^- \to G$ defined by $\mu(m^+,k,m^-) = m^+km^-$.  Then, $\mu$ maps $M^+\times K^{\mathbb C}\times M^-$ biregularly onto a Zariski open subset of the complex algebraic group $G$. Moreover, defining $\nu: M^+ \to G/P = X$ by $\nu(m^+) = \mu(m^+,e,e)\ {\rm mod}(P) \in X$, there is a relatively compact connected open subset $\Omega \subset M^+ \cong \mathbb C^n$ such that $\nu: \Omega \overset{\cong}\longrightarrow X_0 \subset G/P$ is a biholomorphism onto $X_0$.
\label{harish}
\end{theorem}

\begin{proof}
We check first of all that $\mu$ is injective.  Suppose for $i = 1,2$ and for
$(m_i^+,k_i,m_i^-) \in M^+\times K^{\mathbb C}\times M^-$
we have $m_1^+ k_1 m_1^- = m_2^+ k_2 m_2^-$.  Then $(k_1 m_1^-)(k_2 m_2^-)^{-1} = (m_1^+)^{-1}m_2^+$.  Thus $(m_1^{+})^{-1}m_2^{+} \in M^+ \cap P$. Since $M^+, P \subset G$ are algebraic subgroups, $M^+\cap P \subset G$ must be an algebraic subgroup of $M^+ \cong \mathbb C^n$. Now $T_e(M^+)\cap T_e(P)$ can be identified with $\mathfrak m^+\cap\mathfrak p = \mathfrak m^+\cap(\mathfrak k^{\mathbb C}\oplus\mathfrak m^{-}) = 0$, so $M^+\cap P\subset M^+$ is an algebraic subgroup of dimension 0, thus a finite subgroup, hence trivial as $M^+\cong \mathbb C^n$ has no non-trivial finite subgroups.

\vskip 0.2cm
At the identity $(e,e,e)$ of $M^+\times K^{\mathbb C}\times M^-$, identifying $T_e(M^+)$ with $\mathfrak m^+$, etc., $d\mu(e,e,e)$ is the identity map, hence $\mu$ is a local biholomorphism at $(e,e,e)$.  In general, viewing Lie subalgebras as left-invariant holomorphic vector fields, at the point $(m^+,k,m^-)$ we have
$$
\begin{gathered}
d\mu(\mathfrak m^+ + \mathfrak k^{\mathbb C} + \mathfrak m^-) = {\rm Ad}\left((m^-)^{-1}k^{-1}\right)(\mathfrak m^+) + {\rm Ad}\left((m^-)^{-1}\right)(\mathfrak k^{\mathbb C}) + \mathfrak m^- \\
= {\rm Ad}\left((m^-)^{-1}\right)(\mathfrak m^+) + \mathfrak k^{\mathbb C} + \mathfrak m^- = \mathfrak g\, .
\end{gathered}
$$
Here we made use of ${\rm Ad}\big((m^-)^{-1}\big)(\mathfrak m^-) = \mathfrak m^-$ and ${\rm Ad}\big((m^-)^{-1}\big)(\mathfrak p) = \mathfrak p$ (since $M^-, P \subset G$ are Lie subgroups) to deduce ${\rm Ad}\left((m^-)^{-1}\right)(\mathfrak k^{\mathbb C})  \ {\rm mod} \ \mathfrak m^- = \mathfrak k^{\mathbb C}$.
Likewise we made use of ${\rm Ad}\big((m^-)^{-1}\big)(\mathfrak p) = \mathfrak p$ and ${\rm Ad}\big((m^-)^{-1}\big)(\mathfrak g) = \mathfrak g$ to deduce ${\rm Ad}\big((m^-)^{-1}\big)(\mathfrak m^+) \ {\rm mod} \ \mathfrak p = \mathfrak m^+$.

\vskip 0.3cm
For the argument in the above in the case where $(X_0,g_0)$ is the Poincar\'e disk one checks in Wolf (cf.\,\cite{Wo1972}) that identifying $M^+$ with $\mathbb C$, with respect to the Borel embedding $X_0 \subset X = \mathbb P^1$, $X_0$ is the image of $\nu(\Delta)$ for the holomorphic map $\nu: M^+ \to G/P = X$ given by $\nu(m^+) = \mu(x,e,e) \ {\rm mod}\, P$, for the unit disk $\Delta \subset \mathbb C \cong M^+$.
The preceding arguments apply also for reducible Hermitian symmetric spaces of the noncompact type,
and, in the case of a Cartesian product of $r$ identical copies of the Poincar\'e disk, the Harish-Chandra realization then gives $\Delta^r \Subset \mathbb C^r \subset (\mathbb P^1)^r$.
Finally, from the Polydisk theorem it follows that, with respect to the embedding $\nu: M^+ \to X = G/P$ for a dual pair of irreducible Hermitian symmetric spaces $(X_0, X_c)$ of the noncompact resp.~compact type, $X_0 \subset X$
is the union of $k$-translates of a maximal polydisk $\Pi \subset X_0$.  In terms of the Euclidean coordinates on $M^+ \cong \mathbb C^n$, $X_0 = \nu(\Omega)$ where $\Omega$ is the union of $k(\Delta^r)$, where $k$ runs over the isotropy subgroup $K$ acting on $M^+$. Since $K$ is compact, it follows that $\overline{\Omega} = \bigcup_{k\in K}\big (k \overline{\Delta^r}\big) \subset \mathbb C^n$ is compact, hence $\Omega$ is relatively compact in $\mathbb C^n$, written $\Omega\Subset\mathbb C^n$ for the Harish-Chandra realization, and the proof of the theorem is complete.
\end{proof}

\begin{remark} \
\rm
\begin{enumerate}
\item{}
In terms of the Harish-Chandra coordinates, $K$  acts on $\mathbb C^n \cong M^+$ as a group of unitary transformations, hence its complexification $K^{\mathbb C}$ acts on $\mathbb C^n$ as a group of complex linear transformations.  The center $Z \cong \mathbb S^1$ acts as multiplication by $\{e^{i\theta}: \theta\in\mathbb R\}$, and as such $\Omega \Subset \mathbb C^n$ is a complete circular domain.
\item{}
By the Hermann convexity theorem (cf.\,Wolf \cite{Wo1972}), the Harish-Chandra realization $\Omega\Subset\mathbb C^n$ is a convex domain.  By Mok-Tsai \cite{MT1992}, any convex realization is up to a complex affine-linear transformation equivalent to the Harish-Chandra realization.
\item{}
As remarked, the unipotent radical $M^-$ of $P$ acts on $X$ as a group of biholomorphisms $g$ fixing $o \in X$ (which corresponds to $0 \in \mathbb C^n$) such that $dg(0) = \id_{\mathbb C^n}$.
\end{enumerate}
\end{remark}

\vskip 0.2cm
\subsubsection{Classification of bounded symmetric domains}
\label{classify-bsd}
Any bounded symmetric domain $\Omega \Subset \mathbb C^n$ is equipped with the Bergman metric $ds_{\Omega}^2$, which is complete, and $(\Omega,ds_{\Omega}^2)$ is a Hermitian symmetric space of the noncompact and semisimple type.  As such it decomposes into the Cartesian product of irreducible Hermitian symmetric spaces of the noncompact and semisimple type.
Thus any bounded symmetric domain $\Omega$ can be represented as $\Omega = \Omega_1 \times \cdots \times \Omega_s$, where for $1\le i \le s$, $\dim_{\mathbb C}\Omega_{i} = n_i$, and $\Omega_i \Subset \mathbb C^{n_i}$ is an irreducible bounded symmetric domain in its Harish-Chandra realization.
The classification of bounded symmetric domains up to biholomorphism amounts to classifying irreducible bounded symmetric domains.  We have the following listing of irreducible bounded symmetric domains in their Harish-Chandra realizations broken up into 4 classical series and 2 exceptional domains. For the classical series we also present the Harish-Chandra realizations of the bounded symmetric domains in the standard forms. We write an irreducible bounded symmetric domain as $\Omega \cong X_0 = G_0/K$, where $G_0$ is a simple centerless real Lie group which is the identity component $\Aut_0(\Omega)$ of the automorphism group $\Aut(\Omega)$.  In what follows, we present in the classical case $G_0$ as $G_1/\Psi$, where $G_1$ is a simple classical real Lie group in its standard form, and $\Psi \subset G_1$ is the finite center of $G_1$.  For instance in the case of the $n$-dimensional complex unit ball $\mathbb B^n$ we have
$\mathbb B^n = {\rm SU}(1,n)/{\mu_{n+1}\cdot I_{n+1}}$, where ${\rm SU}(p,q)$ stands for the special unitary group of $(p+q)$-by-$(p+q)$ matrices with complex coefficients which preserve the Hermitian bilinear form $H_{p,q}(\cdot,\overline{\cdot})$ of signature $(p,q)$
defined by $H_{p,q}\left(w;\overline{w'}\right) = \left(w_1\overline{w'_1} + \cdots + w_p\overline{w'_p}\right) - \left(w_{p+1}\overline{w'_{p+1}} + \cdots + w_{p+q}\overline{w'_{p+q}}\right)$, and for $k\ge 2$, $\mu_k$ stands for the multiplicative group of $k$-th roots of unity. In what follows we write $M(p,q;\mathbb C)$ for the complex vector space of $p$-by-$q$ matrices with complex coefficients, and $Z^T \in M(q,p;\mathbb C)$ for the transpose matrix of the matrix $Z \in M(p,q;\mathbb C)$.

\vskip 0.2cm
\noindent
\underline{List of irreducible bounded symmetric domains $\Omega \Subset\mathbb C^n$}
\vskip 0.2cm
\noindent We have the following complete list of irreducible bounded symmetric domains $\Omega$ in their Harish-Chandra realizations. In the list we write $r(\Omega)$ for
the rank of $\Omega$ as a Riemannian symmetric space.

\begin{enumerate}
\item{}
$D^I(p,q)= \{Z\in M(p,q;\mathbb C): I-\overline Z^T Z > 0\}, \ p,q\ge 1$, $G_1 = {\rm SU}(p,q)$, $r(\Omega) = {\rm min}(p,q)$;
\item{}
$D^{II}(n,n) = \{Z\in D^{I}(n,n): Z^T = -Z\}, \ n\ge 2$, $G_1 = {\rm SO}^{*}(2n)$, $r(\Omega) = \left[ \frac{n}{\phantom{.}2\phantom{.}}\right]$;
\item{}
$D^{III}(n,n) = \{Z\in D^{I}(n,n): Z^T = Z\}, \ n\ge 2$, $G_1 = {\rm Sp}(2n;\mathbb R)$, $r(\Omega) = n$;
\item{}
$D^{IV}_n:=\bigl\{z=(z_1,\dots,z_n)\in\mathbb{C}^n: \|z\|^2 < 2;$\\
$\|z\|^2 < 1 + \bigl|\tfrac12(z_1^2+\cdots+z_n^2)\bigr|^2\bigr\}$,
$n\ge 3$, $G_1 = {\rm SO}(n,2)$, $r(\Omega) = 2$;
\item{}
$D^{V}$ is an exceptional domain of complex dimension 16, where $G_1$ is an exceptional noncompact real Lie group of type $E_6$, $r(\Omega) = 2$;
\item{}
$D^{VI}$ is an exceptional domain of complex dimension 27, where $G_1$ is an exceptional noncompact real Lie group of type $E_7$, $r(\Omega) = 3$.
\end{enumerate}
\noindent Note that we also write $D^I_{p,q}:=D^I(p,q), D^{II}_n:=D^{II}(n,n)$ and $D^{III}_n:=D^{III}(n,n)$.

\begin{remark} \
\rm
\begin{enumerate}
\rm \item{} Here, in (2) for type-II domains, ${\rm SO}^{*}(2n)$ is a noncompact real form of ${\rm SO}(2n;\mathbb C)$ accompanying the compact real form ${\rm SO}(2n)$.

\item{}
In (3) for type-III domains, $D^{III}_n$ is biholomorphic to the Siegel upper half-plane $\mathcal H_n = \{\tau\in M(n,n;\mathbb C): \tau^T = \tau, \im(\tau) > 0\}$ via the inverse Cayley transformation given by
$Z = (\tau -iI_n)(\tau+iI_n)^{-1}$. The Siegel upper half-plane is the natural parameter space for principally polarized abelian varieties of dimension $n$ in view of the Riemann bilinear relations.
\item{}
In the case where $\Omega = D^{I}_{n,n}$, a maximal polydisk $\Pi \subset \Omega$ is given by all diagonal matrices
 ${\rm diag}(z_{11},\cdots,z_{nn})$ belonging to $\Omega$, i.e., those satisfying $|z_{11}| < 1$, $\cdots$, $|z_{nn}| < 1$. It is straightforward that for any
$n$-by-$n$ matrix $Z$ belonging to $\Omega$, there exist unitary matrices $U, V \in {\rm U}(n)$ such that $UZV^T = {\rm diag}(z_{11},\cdots, z_{nn}) \in \Pi$, which serves as an example illustrating the Polydisk theorem, viz., that $\Omega$ is the union of translates of $\Pi$ under $K$, which is the image of ${\rm U}(n)\times {\rm U}(n)$ in ${\rm GL}(M(n,n;\mathbb C))$ under the homomorphism $\Phi(U,V)(X) = UXV^T$ for $X \in M(n,n;\mathbb C) \cong T_0(\Omega)$. Note that here $Z$ is not assumed to be symmetric (nor Hermitian): the normal form for type-I domains with $p=q=n$ uses two unitary matrices $U,V$ via $UZV^T$, rather than a single unitary matrix as in the complex-symmetric type-III case.
\end{enumerate}
\end{remark}

\vskip 0.2cm
\subsubsection{Decomposition of \texorpdfstring{$\partial\Omega$}{partial Omega} into \texorpdfstring{$G_0$}{G0}-orbits}
\label{bdy-orbits}
For an irreducible bounded symmetric domain $\Omega \Subset \mathbb C^n$ in its Harish-Chandra realization of rank $r = r(\Omega)$,
there is a decomposition of its boundary $\partial\Omega$ in $\mathbb{C}^n$, which is semi-algebraic, into the disjoint union of $r$ $G_0$-orbits $E_0, E_1, \ldots, E_{r-1} = \Reg(\partial\Omega)$, where $E_k \subset \overline{E}_{k+1}$ for $0\le k < r-1$, so that for $0\le k\le r-1$, $\overline{E_k}$ is the disjoint union $E_0\cup E_1 \cup \cdots\cup E_k$. The deepest stratum $E_0$ of $\partial\Omega$ is the Shilov boundary ${\rm Sh}(\Omega)$ of $\Omega$,
i.e., the smallest closed subset of $\overline{\Omega}$ where all complex continuous functions $f$ on $\overline{\Omega}$ holomorphic on $\Omega$ attain the maxima of their absolute values (i.e., ${\rm sup}\{|f(z)|: z \in\overline{\Omega}\}$).  For $0\le k\le r-1$, $E_k$ decomposes into the disjoint union of complex submanifolds $\Phi \subset E_k$ which are biholomorphic to an irreducible bounded symmetric domain of rank $k$.
Furthermore $\Phi$ is a connected open subset of a complex affine subspace $A$ whose topological closure $\overline{A}$ is biholomorphic to the Hermitian symmetric space of the compact type dual to $\Phi$, and we will write $\widehat{\Phi} :=\overline{A}$.
Moreover, $\Phi\Subset A\subset\widehat{\Phi}$ is isomorphic to $\Omega_k\Subset \mathbb C^{n_k}\subset\widehat{\Omega}_k$ in the sense that there is a biholomorphism $\nu: \widehat{\Phi} \overset{\cong}\longrightarrow \widehat{\Omega}_k$ such that $\nu(\Phi) = \Omega_k$ and $\nu(A) = \mathbb C^{n_k}$, where $\Omega_k \Subset \mathbb C^{n_k}$ is an irreducible bounded symmetric domain of rank $k$ and of dimension $n_k\ge 0$ in its Harish-Chandra realization. In this article we will also refer to boundary components on $\partial\Omega$ as faces. Thus, $\Reg(\partial\Omega)$ decomposes into the disjoint union of maximal faces.

\vskip 0.2cm
\subsection{Quotients of a bounded symmetric domain by an irreducible lattice of automorphisms}
\label{quotients-lattice}
Let now $\Omega$ be a bounded symmetric domain of rank $\ge 2$ and $\Gamma \subset \Aut(\Omega)$ be a torsion-free irreducible lattice, so that $X_\Gamma :=\Omega/\Gamma$ carries a canonical structure as a quasi-projective manifold.
Let $D \Subset \mathbb C^N$ be a bounded domain, $\Gamma' \subset \Aut(D)$ be a torsion-free discrete subgroup such that $Y_{\Gamma'} = D/\Gamma'$ is biholomorphic to a quasi-projective variety.  Let $f: X_{\Gamma} \to Y_{\Gamma'}$ be a holomorphic map which induces an isomorphism $\Phi:= f_*: \Gamma = \pi_1(X_\Gamma) \overset{\cong}\longrightarrow \Gamma' \subset \Aut(D)$, i.e., the lifting $F: \Omega \to D$ is a $\Phi$-equivariant holomorphic map for $\Phi$ defined by $F(\gamma z) = \Phi(\gamma)(F(z))$ for any $\gamma \in \Gamma$, and $\Phi: \Gamma \to \Aut(D)$ maps $\Gamma$ bijectively onto $\Gamma'$.
In a recent article of the authors, we proved a result called the Isomorphism Theorem, which implies that $F: \Omega \to D$ is necessarily a biholomorphism. The study of rigidity results on lattices of a Riemannian symmetric space of negative Ricci curvature started with Mostow's rigidity theorem \cite{Mos1973} and, in the case of irreducible lattices of rank $\ge 2$, Margulis' arithmeticity theorem says that the lattice must be arithmetic \cite{Mar1984}, while Margulis' superrigidity theorem says that, denoting by $G_0 = \Aut_0(X_0)$ resp.\,$G_0' = \Aut_0(X_0')$ the identity components of the automorphism groups of Riemannian symmetric spaces of the semisimple and noncompact type $(X_0,g_0)$ resp.\,$(X_0',g_0')$, any homomorphism $\Phi: \Gamma \to G_0'$ of the discrete subgroup $\Gamma \subset G_0$ with infinite image in a real semisimple Lie group $G_0'$ of the noncompact type must lift to a group homomorphism $\widetilde{\Phi} :G_0 \to G_0'$ (cf. Margulis \cite{Mar1991}).

\vskip 0.2cm
In the study of higher-dimensional analogues of the uniformization theorem generalizing the hyperbolic case of  algebraic curves of genus $\geq 2$, the first result concerning compact K\"ahler manifolds by imposing topological assumptions was the following strong rigidity theorem of Siu \cite{Siu1980} (1980).

\begin{theorem} {\rm (Siu \cite{Siu1980, Siu1981})} \
Let $\Omega$ be an irreducible bounded symmetric domain of complex dimension $\ge 2$ and $\Gamma \subset \Aut(\Omega)$ be a torsion-free cocompact lattice.  Writing $X_\Gamma := \Omega/\Gamma$, let $Z$ be a compact K\"ahler manifold homeomorphic to $X_\Gamma$.  Then, $Z$ is either biholomorphic or conjugate biholomorphic to $X_{\Gamma}$.
\label{strong-rigid}
\end{theorem}

For irreducible bounded symmetric domains $\Omega$ of rank $\ge 2$ and torsion-free lattices $\Gamma$ of automorphisms on $\Omega$ we have the following Hermitian metric rigidity theorem of Mok \cite{Mok1987} (1987) and To \cite{To1989} (1989).

\begin{theorem} {\rm (Mok \cite{Mok1987}, To \cite{To1989})} \
Let $\Omega$ be an irreducible bounded symmetric domain of ${\rm rank}\ge 2$ and $\Gamma \subset \Aut(\Omega)$ be a torsion-free lattice.
Write $n=\dim_\mathbb{C}\Omega$ and denote by $g_{\Omega}$ the canonical complete K\"ahler-Einstein metric on $\Omega$ of
constant Ricci curvature $-(n+1)$ and by $g$ its quotient K\"ahler-Einstein metric on $X_\Gamma$. Let $h$ be a Hermitian metric on $X_\Gamma$ of nonpositive curvature in the sense of Griffiths. Then, there exists a constant $\lambda > 0$ such that $h \equiv\lambda g$ on $X_\Gamma$.
\label{hermite-rigid}
\end{theorem}

Theorem \ref{hermite-rigid} was proven by Mok \cite{Mok1987, Mok1989} for $X_\Gamma$ compact and for $X_\Gamma$ noncompact (and of finite volume) under the assumption that either $h \le Cg$ for some constant $C > 0$ or $h$ is K\"ahler. To \cite{To1989} proved the general case for $X_\Gamma$ noncompact by making use of the Satake-Baily-Borel compactification of Satake \cite{Sa1960} (1960) and Baily-Borel \cite{BB1966} (1966).

\vskip 0.2cm
\section{A uniformization theorem under a topological and analytic hypothesis for complete K\"ahler-Einstein manifolds of finite-volume}
\vskip 0.2cm
\subsection{Statement of a rigidity result concerning irreducible lattices \texorpdfstring{$\Gamma \subset\Aut(\Omega)$}{Gamma subset Aut(Omega)} of rank \texorpdfstring{$\ge 2$}{>= 2}, K\"ahler-Einstein metrics and bounded holomorphic functions}
\label{rigidity-statement}
We start with the statement of a uniformization theorem characterizing $X_\Gamma$ biholomorphically among complete K\"ahler-Einstein manifolds of finite volume. The following result is not stated explicitly in Mok-Wong \cite{MW2025} but it follows from the proof of the {\it Isomorphism Theorem\/} of the article.

\begin{theorem} \
Let $\Omega$ be a bounded symmetric domain of rank $\ge 2$,
$\Gamma \subset \Aut(\Omega)$ be a torsion-free irreducible lattice,
 and write $X_\Gamma := \Omega/\Gamma$.
 Let $(M,h_M)$ be a simply connected complete K\"ahler-Einstein manifold of
 negative Ricci curvature, and $\Gamma' \subset \Aut(M)$ be a torsion-free
 discrete subgroup such that, writing $Y_{\Gamma'} := M/\Gamma'$ and denoting
 by $h_{Y_{\Gamma'}}$ the K\"ahler-Einstein metric inherited from $h_M$,
 we have ${\rm Volume}\left(Y_{\Gamma'},h_{Y_{\Gamma'}}\right) < \infty$.
 Let $f: X_\Gamma \to Y_{\Gamma'}$ be a holomorphic map such that $f_*: \Gamma = \pi_1(X_\Gamma) \overset{\cong}\longrightarrow \pi_1(Y_{\Gamma'}) =: \Gamma'$ is a group isomorphism.  Assume furthermore that $M$ is Carath\'eodory hyperbolic.  Then, $f: X_\Gamma \overset{\cong}\longrightarrow Y_{\Gamma'}$ is a biholomorphism, hence its lifting $F: \Omega  \overset{\cong}\longrightarrow M$ is also a biholomorphism.
\label{isom}
\end{theorem}

\begin{remark} \
\begin{enumerate}
\rm
\item{} Without requiring that $M$ is simply connected, Theorem \ref{isom} can be slightly rephrased as follows.
Let $\Gamma' \subset \Aut(M)$ be a torsion-free discrete subgroup of automorphisms and write $Y_{\Gamma'} := M/\Gamma'$.
Let $f: X_\Gamma \to Y_{\Gamma'}$ be a holomorphic map which lifts to a $\Phi$-equivariant holomorphic map $F: \Omega \to M$,
where the group isomorphism $\Phi: \Gamma \overset{\cong}\longrightarrow \Gamma'$ is defined by $F(\gamma z) = \Phi(\gamma)F(z)$,
and where $F$ may {\it a priori\/} be a lifting of $f$ to an intermediate covering of $Y_{\Gamma'}$.
In this more general setting the analogue of Theorem \ref{isom} remains valid: $F:\Omega \overset{\cong}\longrightarrow M$ is a biholomorphism,
so that $M$ is {\it a posteriori\/} simply connected.

\item{}
In \cite{MW2025} the target is a quotient $Y_{\Gamma'} = D/\Gamma'$ of a bounded domain $D$
(more generally, a bounded domain on a Stein manifold),
and the finite-volume hypothesis is imposed with respect to
the Kobayashi-Eisenmann volume form $d\mu$ of $Y_{\Gamma'}$,
which is valid whenever $Y_{\Gamma'}$ is biholomorphic to a quasi-projective manifold
by the Schwarz Lemma for volume forms.
The formulation of Theorem \ref{isom} in terms of complete K\"ahler-Einstein manifolds of finite volume
is a variant following from the same arguments.

\item{}
When $M$ is biholomorphic to a bounded domain of holomorphy $D$,
by Cheng-Yau \cite{CY1980} and Mok-Yau \cite{MY1983} $D \Subset \mathbb C^N$
admits a unique complete K\"ahler-Einstein metric of
constant Ricci curvature $-(N+1), N=\dim_\mathbb{C}D$.
 The Isomorphism Theorem in Mok-Wong \cite{MW2025} is
 formulated for such bounded domains $D \Subset \mathbb C^N$
 (and more generally for bounded domains on a Stein manifold)
 without requiring that $D$ is a domain of holomorphy.
 Without the condition that $D$ is a domain of holomorphy,
 we embed $D$ into its hull of holomorphy $\pi: \widehat D \to \mathbb C^N$,
 which is {\it a priori\/} a Riemann domain of holomorphy.
 Using the hypothesis that $\left(Y_{\Gamma'},d\mu\right)$ is of finite volume, by Cauchy estimates we proved nonetheless that $\widehat{D} - D$ is of zero Lebesgue measure, so that in fact $\widehat{D}$ is a schlicht (univalent) bounded domain on $\mathbb C^N$.  The conclusion of the Isomorphism Theorem of Mok-Wong \cite{MW2025} says that $F:\Omega \overset{\cong}\longrightarrow \widehat{D}$ is already a biholomorphism, so that $D = \widehat D$ is itself a bounded domain of holomorphy.

\end{enumerate}
\end{remark}

\vskip 0.2cm
\subsection{Previous methods in complex geometry on rigidity concerning irreducible lattices \texorpdfstring{$\Gamma \subset \Aut(\Omega)$}{Gamma subset Aut(Omega)}}
\label{premethod}
\mbox{}

\vskip 0.5cm
\subsubsection{The method of harmonic maps in K\"ahler geometry for the proof of strong rigidity (Theorem \ref{strong-rigid})}
\label{harm-maps-method}
The starting point of the proof of Theorem \ref{strong-rigid} is
Siu's $\partial\overline{\partial}$-Bochner-Kodaira formula in \cite{Siu1980}.
Let $(Z,h)$ be a K\"ahler manifold and denote by $\omega$ its K\"ahler form.
For a smooth map $f: Z \to X$ from $Z$ to a complex manifold $X$,
$df$ may be regarded as an $f^*T_X^{\mathbb R}$-valued smooth 1-form on $Z$.
Decomposing $T_Z^{\mathbb C}
= T_Z^{\mathbb R}\otimes_{\mathbb R}\mathbb C = T^{1,0}_Z \oplus T^{0,1}_Z$
and similarly for $T_X^{\mathbb C}$ we have a decomposition
$df = \partial f + \partial \overline{f}
+ \overline{\partial}f + \overline{\partial}\overline{f}$,
where $\partial f$ is an $f^*T_X^{1,0}$-valued $(1,0)$-form on $Z$,
$\overline{\partial}f$ is an $f^*T_X^{1,0}$-valued $(0,1)$-form on $Z$, etc.
Siu's $\partial\overline{\partial}$-Bochner-Kodaira formula gives
\footnote{For the curvature tensor $R = R^N$ of
a K\"ahler manifold $(N,s)$ we adopt the sign convention
such that for the Poincar\'e disk $\left(\Delta,ds_{\Delta}^2\right)$
of constant Gaussian curvature $-2$ we have
$R_{1\overline{1}1\overline{1}} = -2$,
as opposed to $+2$ in
\cite{Siu1980}.}

$$
(\sharp) \quad \partial\overline{\partial}\langle g,\partial f\wedge\overline{\partial}f\rangle =
\left\langle -R, \partial f\wedge\overline{\partial}f\wedge\partial\overline{f}\wedge\overline{\partial f}\right\rangle - \left\langle g, D\overline{\partial}f\wedge \overline{D}\partial\overline f\right\rangle
$$
Let $f_0: (Z,h) \to (X,g)$ be an arbitrary continuous map between compact K\"ahler manifolds and assume that $(X,g)$ is of nonpositive Riemannian sectional curvature.  Then, replacing $f_0$ by a homotopic smooth map if necessary, the method of the heat flow deforms $f_0$ to a harmonic map $f: (Z,h) \to (X,g)$. Let $\omega$ be a K\"ahler form on $Z$, and write $n :=\dim_{\mathbb C}Z$.  Then, starting with the formula $(\sharp)$, taking wedge of both sides of the equation with $\omega^{n-2}$ and integrating over $Z$, by Stokes' Theorem and the harmonicity of $f$ we get the identity
$$
(\sharp\sharp) \quad
0 = \int_Z \left\langle -R,
\partial f\wedge\overline{\partial}f\wedge\partial\overline{f}\wedge\overline{\partial f}\right\rangle \wedge \omega^{n-2}
+ \int_Z  \left\langle g,D\overline{\partial}f\wedge \overline{D}\partial\overline{f}\wedge \omega^{n-2}\right\rangle
$$
Here the integrand of the last term  is a gradient term on $\overline{\partial}f$ regarded as an $f^*T_X^{1,0}-$valued $(1,0)$-form on $Z$.
From linear algebra we obtain the following identity
$$
-\left\langle g, D\overline{\partial}f\wedge \overline{D}\partial\overline f\right\rangle \wedge \omega^{n-2} = \|D\overline{\partial}f\|^2\frac{\omega^n}{n(n-1)}
$$
under the assumption that $f$ is harmonic, i.e., the trace of $D\overline{\partial}f$ with respect to the K\"ahler form $\omega$ is 0. When the curvature tensor $R$ of $(X,g)$ is nonpositive in the dual Nakano sense, i.e.,
$\sum_{i,j,k,\ell} R_{i\overline{j}k\overline{\ell}} A^{i\overline{j}}\overline{A^{\ell\overline{k}}} \le 0$ for any matrix $A^{i\overline{j}}$ at every point (which in particular implies that $(X,g)$ is of nonpositive Riemannian sectional curvature), it was deduced that the curvature term in the integrand of the right hand-side of $(\sharp\sharp)$ is nonnegative, and one deduces readily that both integrands on the right-hand side vanish identically.
Under the assumptions $\dim_{\mathbb C}\Omega\ge 2$ and $df: T_Z^{\mathbb R} \to T_X^{\mathbb R}$ is of rank $> s(\Omega)$ for an integer $s(\Omega)$ determined explicitly,
Siu \cite{Siu1980} and Zhong \cite{Zh1984} proved that $f$ must be either holomorphic or anti-holomorphic.  In particular this is always the case when $f_0: X \to Z$ is a homeomorphism (Siu \cite{Siu1980, Siu1981}).
In this case, assuming without loss of generality that $f$ is holomorphic, the map is necessarily finite since no irreducible subvariety of dimension $k\ge 1$, whose homology class in $H_{2k}(X,\mathbb R)$ is nontrivial, can be blown down to a point, and it must be of degree 1 since $f_0$ is a homeomorphism, proving Theorem \ref{strong-rigid}.

\vskip 0.2cm
\subsubsection{Varieties of minimal rational tangents and a Gauss-Bonnet formula for the proof of Hermitian metric rigidity (Theorem \ref{hermite-rigid})}
\label{vmrt-gauss}
Let $\Omega$ be an irreducible bounded symmetric domain of rank $r\ge 1$ and let $\Omega \subset \widehat{\Omega}$ be the Borel embedding.
For each point $x \in \widehat{\Omega}$ denote by $\mathscr C_x(\widehat{\Omega})$ the set of all $[\alpha] \in \mathbb PT_x(\widehat{\Omega})$ such that $\alpha$ is tangent at $x$ to a minimal rational curve $\ell$, i.e., a projective line lying on $\widehat{\Omega}$ when $\widehat{\Omega}$ is identified with its image under the minimal projective embedding (such as the Pl\"ucker embedding in the case of a Grassmannian) into a projective space. Here $\mathscr C_x(\widehat{\Omega})$ is called the variety of minimal rational tangents (VMRT) at $x$ and the collection of all VMRTs as $x$ varies over $\widehat{\Omega}$ defines the VMRT structure $\pi: \mathscr C(\widehat{\Omega}) \to \widehat{\Omega}$, where $\pi$ is the natural projection.
For $x\in \Omega$ and $[\alpha] \in \mathscr C_x(\Omega)$ we define $\mathcal N_\alpha = \big\{\zeta \in T_x(\Omega): R_{\alpha\overline{\alpha}\zeta\overline{\zeta}} = 0\big\}$.
When $r \ge 2$ we have $\dim_\mathbb{C}\mathcal N_\alpha = q(\Omega) \ge 1$ for some $q := q(\Omega)$ independent of $[\alpha]\in \mathscr C_x(\Omega)$. We define
$\mathscr C(\Omega) := \pi^{-1}(\Omega)$.  $\mathscr C(\Omega)$ is invariant under the action of $\Aut(\Omega)$ and in particular under the torsion-free discrete subgroup $\Gamma \subset \Aut(\Omega)$.  We write $\mathscr C(X_\Gamma) := \mathscr C(\Omega)/\Gamma$, which has naturally the structure of a quasi-projective manifold equipped with the canonical projection $\mu_\Gamma: \mathscr C(X_\Gamma) \to X_\Gamma$.

\vskip 0.2cm
Recall that for any Hermitian holomorphic vector bundle $(V,s)$ over a complex manifold $M$ we have correspondingly the tautological Hermitian holomorphic line bundle $\left(L,\widehat{s}\right)$ over $\mathbb PV$ and that $(V,s)$ is of (semi-)negative curvature in the sense of Griffiths if and only if the Hermitian holomorphic line bundle $(L,\widehat{s})$ is of (semi-)negative curvature. The Gauss-Bonnet type formula arises first of all from the integral formula over $\mathscr C(X_\Gamma)$ given by
$$
(\dagger) \ \int_{\mathscr C(X_\Gamma)}  \left(-c_1(L,\widehat{g})\right)^{2n-2q}\wedge (\mu_{\Gamma}^*\omega)^{q-1} = 0\, ,
$$
where $g$ denotes the quotient K\"ahler metric on $X_\Gamma$ descending from the K\"ahler-Einstein metric $g_\Omega$ on $\Omega$, $\omega$ denotes the K\"ahler form of $(X_\Gamma,g)$. The formula ($\dagger$) is checked by verifying that the kernel ${\bf N}_{[\alpha]}$ of the $(1,1)$-form $c_1(X_\Gamma,g)([\alpha])$ is tangent to $\mathscr C(X_\Gamma)$. One way to check that is to observe that, given $[\alpha] \in \mathscr C_x(\Omega)$ and $\zeta \in \mathcal N_\alpha$, there is a holomorphic {\it totally geodesic\/} isometric embedding of $\Delta\times \Omega'$ into $\Omega$, $\Omega'$ being an irreducible bounded symmetric domain of rank $r - 1\ge 1$ (and of dimension $q = q(\Omega)$) such that, identifying $\Delta\times \Omega'$ with its image $S\subset\Omega$, we have $x\in S$ while $\alpha$ is tangent to the disk factor $\Delta$ and $\zeta$ is tangent to the $\Omega'$ factor. If now $X_\Gamma$ is compact and $h$ is a Hermitian metric on $X_\Gamma$ of nonpositive curvature in the sense of Griffiths, then $(X_\Gamma,g+h)$ has the same curvature property. The first Chern form $c_1(L,\widehat{g}+\widehat{h})$ is nonpositive, i.e., $c_1(L,\widehat{g}+\widehat{h}) \le 0$, and it is cohomologous to $c_1(L,\widehat{g})$, leading by Stokes' theorem to the Gauss-Bonnet type formula
$$
\begin{gathered}
(\dagger\dagger) \ \int_{\mathscr C(X_\Gamma)} \left(-c_1(L,\widehat{g}+\widehat{h})\right)
\wedge\left(-c_1(L,\widehat{g})\right)^{2n-2q-1}\wedge (\mu_{\Gamma}^*\omega)^{q-1} \\
= \int_{\mathscr C(X_\Gamma)} \left(-c_1(L,\widehat{g})\right)^{2n-2q}\wedge (\mu_{\Gamma}^*\omega)^{q-1} = 0\, ,
\end{gathered}
$$
from which one deduces the vanishing of the integrand on the left-hand side of $(\dagger\dagger)$.

\vskip 0.2cm
\noindent
Proof of Theorem \ref{hermite-rigid} in the case $X_\Gamma$ is compact. \ Theorem \ref{hermite-rigid} for $X_\Gamma$ compact follows from the Gauss equation for Hermitian holomorphic vector subbundles. More precisely, from the Gauss-Bonnet formula $(\dagger\dagger)$ we deduce the vanishing of $R_{\alpha\overline{\alpha}\zeta\overline{\zeta}}(g+h) = 0$ whenever $[\alpha] \in \mathscr C(X_\Gamma)$ and $\zeta \in \mathcal N_{\alpha}$. Write $x$ for the base point of $\alpha$. Comparing the curvature term $R_{\alpha\overline{\alpha}\zeta\overline{\zeta}}(g+h) = 0$ with respect to $g+h$ to the same curvature term with respect to $g$ resp.~$h$ yields by
the Gauss equation the partial vanishing $(\sharp) \nabla_{\zeta} h_{\alpha\overline{\eta}}$ for all $\eta \in T_x(X_\Gamma)$ and for $\nabla$ denoting covariant differentiation on $(X_\Gamma,g)$, and expanding $(\sharp)$,
where $\zeta \in \mathcal N_\alpha$ varies anti-holomorphically as $\alpha$ varies holomorphically, yields the identity $\nabla h \equiv 0$ on $X_\Gamma$, implying Theorem \ref{hermite-rigid} for cocompact lattices $\Gamma \subset \Aut(\Omega)$ by the irreducibility of $\Omega$.

\vskip 0.2cm
From Theorem \ref{hermite-rigid} and its proof one can deduce the following rigidity result on holomorphic mappings.

\begin{corollary} {\rm (Mok \cite{Mok1987})} \
Let $(Y,h)$ be a K\"ahler manifold of nonpositive holomorphic bisectional curvature.  In the notation as in Theorem \ref{hermite-rigid}, let $f:X_\Gamma\to Y$ be a nonconstant holomorphic map.  Then, $f: (X_\Gamma,g)\to (Y,h)$ must necessarily be a totally geodesic immersion and a holomorphic isometry up to a scaling constant.
\label{rigid-map}
\end{corollary}

\begin{proof} Theorem \ref{hermite-rigid} (Hermitian metric rigidity theorem) implies $g+f^*h = \lambda g$ for some $\lambda > 1$, implying that $f:X_\Gamma \to Y$ is a holomorphic isometry up to a scaling constant, in particular a holomorphic immersion.  Identifying $X_\Gamma$ locally as a complex submanifold $S$ of $Y$, the equalities $R^S_{\alpha\overline{\alpha}\zeta\overline{\zeta}}(h) = 0 = R_{\alpha\overline{\alpha}\zeta\overline{\zeta}}(h) = 0$ whenever $[\alpha] \in \mathscr C(X_\Gamma)$ and $\zeta \in \mathcal N_{\alpha}$ with base point $x$ lying on $S$ yield by the Gauss equation for K\"ahler submanifolds the partial vanishing $(\sharp) \ \sigma_{\alpha\zeta} = 0$ of the second fundamental form of $\left(S, h|_S\right) \hra (Y,h)$ whenever $[\alpha]\in \mathscr C(X_\Gamma)$ and $\zeta \in \mathcal N_\alpha$.  Expanding $(\sharp)$ as $\alpha$ varies holomorphically and $\zeta$ varies anti-holomorphically, and polarizing, we obtain $\sigma \equiv 0$.
\end{proof}

Let $\Omega'$ be any bounded symmetric domain and $\Gamma' \subset \Aut(\Omega')$ be a torsion-free lattice.
Taking $Y :=\Omega'/\Gamma'$ and letting $h$ be the canonical K\"ahler-Einstein metric on $Y$ of constant negative Ricci curvature, then the Margulis superrigidity theorem (cf. Margulis \cite{Mar1991}) applies together with uniqueness results on harmonic maps at least in the case of cocompact lattices to show that the holomorphic map $f: (X_\Gamma,g)\to (Y,h)$ must necessarily be a totally geodesic immersion and a holomorphic isometry up to a scaling constant. Corollary \ref{rigid-map} yields in particular the same result when $\Gamma' \subset \Aut(\Omega')$ is only assumed to be an arbitrary torsion-free subgroup.

\vskip 0.2cm
\subsubsection{Finsler metric rigidity for the proof of an embedding theorem}
\label{finsler-method}
We will be studying a new rigidity phenomenon for torsion-free irreducible lattices $\Gamma$ of rank $\ge 2$ on a bounded symmetric domain $\Omega$.  The rigidity will concern holomorphic maps from $\Omega$ to complex manifolds equivariant with respect to some discrete representations of $\Gamma$, and will be formulated in terms of bounded holomorphic functions on the target complex manifold.  As a precursor to such a phenomenon, by an extension of Hermitian metric rigidity to Finsler metric rigidity, Mok \cite{Mok2004} gave an embedding theorem for such maps.

\vskip 0.2cm
For a possibly reducible bounded symmetric domain $\Omega$ in its Harish-Chandra realization,
 we write $\Omega = \Omega_1\times\cdots\times\Omega_s$
 for its decomposition into irreducible factors.
By a factor submanifold we mean a complex submanifold
$S_k := ({\bf a}_1,\cdots,{\bf a}_{k-1})\times \Omega_k \times ({\bf a}_{k+1},\cdots,{\bf a}_s)$
for some point ${\bf a} = ({\bf a}_1,\cdots,{\bf a}_{k-1}; \\{\bf a}_{k+1},\cdots,{\bf a}_{s}) \in (\Omega_1\times\cdots\times\Omega_{k-1})\times (\Omega_{k+1}\times\cdots\times\Omega_{s})$.

\vskip 0.2cm
We state now a version of the Finsler metric rigidity theorem sufficient for applications to rigidity problems. In what follows let $\Omega \Subset \mathbb C^n$ be a bounded symmetric domain of rank $\ge 2$, $\Omega = \Omega_1\times\cdots\times \Omega_s$ be its decomposition into irreducible factors, $T_\Omega = T_1\oplus\cdots\oplus T_s$ be the corresponding direct sum decomposition of its holomorphic tangent bundle. We have

\begin{theorem} \
Let $\Gamma \subset \Aut(\Omega)$ be a torsion-free irreducible cocompact lattice, and write $X_\Gamma := \Omega/\Gamma$. Let $g$ be the canonical K\"ahler-Einstein metric on $X_\Gamma$, and $h$ be a continuous complex Finsler metric on $X_\Gamma$ of nonpositive curvature. Denote by $\|\cdot\|_g$ resp.\,$\|\cdot\|_h$ lengths of vectors measured with respect to $g$ resp.\,$h$. Then, there exist positive constants $c_1,\cdots, c_s$ such that for any $\eta \in T_{X_\Gamma}$ that can be lifted to a minimal rational tangent vector belonging to $T_k$; $1 \le k \le s$; we have $\|\eta\|_h = c_k\|\eta\|_g$. Moreover, the same holds true in the case of nonuniform torsion-free irreducible lattices $\Gamma\subset\Aut(\Omega)$, provided that there exists a positive constant $C$ such that $h \le Cg$ and $h$ is uniformly Lipschitz with respect to $g$.
\label{finsler}
\end{theorem}

For the explanation of the proof of Finsler metric rigidity, we will restrict to the case where $\Omega$ is irreducible and of rank $r \ge 2$, which is the essential case, and assume that $X_\Gamma$ is compact. In this case, we can interpret the Gauss-Bonnet formula $(\dagger\dagger)$ in the following way.
The closed nonnegative $(1,1)$-form $-c_1(L,\widehat{g})$ has a complex $q$-dimensional kernel at every point, the distribution $\mathscr D$ defined by $[\alpha] \mapsto {\rm Ker}(c_1(L,\widehat{g}))$ is necessarily integrable by the closedness of the first Chern form, and its leaves on $\mathscr C(\Omega)$ at every point $[\alpha]$ are given as follows.
Let $\alpha \neq 0$ be a nonzero minimal rational tangent vector at $x\in X_\Gamma$. Then, there exists a $q$-dimensional irreducible bounded symmetric domain $\Omega^{\flat}$ of rank $r - 1$ (biholomorphic to some, hence any, maximal face on $\Reg(\partial\Omega)$).
Equipping $\Delta$ with the Poincar\'e metric, there is a totally geodesic holomorphic isometric embedding of $\nu: \Delta\times\Omega^\flat \to \Omega$ such that $\alpha$ is $d\nu(\alpha_0)$ for a nonzero vector $\alpha_0$ tangent to the disk $\Delta$ factor at $(0;0) \in \Omega$.  Then, there is a constant section $\sigma$ of $\mu: \mathscr C_{\Gamma}(\Omega) \to \Omega$ over $\mathfrak N_{[\alpha]}:= \nu(\{0\}\times\Omega^{\flat})$ in terms of Harish-Chandra coordinates, whose image $\sigma(\mathfrak N_\alpha)$ is an integral manifold of ${\rm Ker}(c_1(L,\widehat{g}))$. The decomposition of $\mathscr C(\Omega)$ into such maximal integral complex manifolds $\mathfrak N_{[\alpha]}$ of the integrable distribution $\mathscr D$ gives a real-analytic foliation $\mathscr F$ with closed leaves which are $q$-dimensional complex manifolds.  The leaf space $\mathcal M$ is a $(4p+2)$-dimensional real-analytic manifold equipped with a symplectic form $\xi$, such that, writing $\chi: \mathscr C(\Omega) \to \mathcal M$ for the universal family for the foliation $\mathscr F$, we have $-c_1(L,\widehat{g}) = \chi^*\xi$.

\vskip 0.2cm
Theorem \ref{finsler} in the case where $X_\Gamma$ is compact and locally irreducible
can be understood as an application of the integral formula $(\dagger\dagger)$
interpreted using the foliation $\mathscr{F}_\Gamma$ on $\mathscr C(X_\Gamma)$,
which is obtained by descent from $\mathscr{F}$.
Lifted locally to $\mathscr C(\Omega)$, the left-hand side of the integrand in $(\dagger\dagger)$ can be interpreted as integrating first of all the $(q,q)$-form $-c_1(L,\widehat{g}+\widehat{h})\wedge(\mu^*\omega)^{q-1}$ over each $q$-dimensional leaf $\mathfrak N_{[\alpha]}$ of $\xi: \mathscr C(\Omega) \to \mathcal M$ and then integrating against the transverse measure which is the volume form $\xi^{2p+1}$, given $-c_1(L,\widehat{g}) = \chi^*\xi$. The vanishing of the integral on the left-hand side can be used to deduce that, writing $\widehat{g}+\widehat{h} =: e^{u}\widehat{g}$, $u$ is constant on each leaf $\mathcal L$ of $\mathscr F_{\Gamma}$ on $\mathscr C(X_\Gamma)$, and the desired conclusion $u \equiv \lambda$ for some positive constant $\lambda$ then follows from the existence of a single dense leaf $\mathcal L$ of $\mathscr F_\Gamma$, which follows from Moore's ergodicity theorem on semisimple real Lie groups, an essential tool used in the proof of Theorem \ref{isom} to be expounded in \S\ref{ergodic}. Finally $(\dagger\dagger)$ holds true without starting with a Hermitian metric $h$.  It suffices to replace $\widehat{h}$ by any smooth complex Finsler metric $s$ defined on the VMRT structure $\mu_{\Gamma}: \mathscr C_{\Gamma}(\Omega) \to X_{\Gamma}$. The case where $s$ is only assumed continuous and of nonpositive curvature in the sense of currents requires a justification given in Mok \cite{Mok2004}.

\begin{theorem} \
Let $\Omega \Subset \mathbb C^n$ be a possibly reducible bounded symmetric domain in its Harish-Chandra realization, $\Omega = \Omega_1\times\cdots \times\Omega_s$ be its decomposition into irreducible factors. Let $\Gamma\subset \Aut(\Omega)$ be a torsion-free irreducible lattice and write $X_\Gamma := \Omega/\Gamma$. Let $N$ be a complex manifold and denote by $\widetilde{N}$ its universal cover. Let $f : X_\Gamma \to N$ be a holomorphic mapping and $F : \Omega \to \widetilde{N}$ be its lifting to universal covers. Assume that, for each $k$, $1 \le k\le s$, there exists a bounded holomorphic function $h_k$ on $\widetilde N$ and a factor submanifold $S_k \subset\Omega$  such that $F^*h_k$ is nonconstant on $S_k$. Then, $F : \Omega \to \widetilde{N}$ is a holomorphic embedding.
\label{cara-embed}
\end{theorem}
Equip the unit disk $\Delta$ with the Poincar\'e metric $g_\Delta$ of constant Gaussian curvature $-2$.  We write $\|\cdot\|_{\rm Poin}$ for $\|\cdot\|_{g_\Delta}$ to emphasize the metric under consideration is the Poincar\'e metric.

\begin{definition}[Carath\'eodory pseudometric and hyperbolicity] \
Let $n$ be a positive integer and $M$ be an $n$-dimensional complex manifold.  By the $($intrinsic$)$ Carath\'eodory pseudometric, we mean a pseudonorm $\kappa_M$ (or denoted by $\|\cdot \|_{\kappa(M)}$) on the holomorphic tangent $T_M$ bundle defined by
\begin{equation}
\|\eta\|_{\kappa_M} = {\rm sup}\left\{\|ds(\eta)\|_{\rm Poin}: s\in H^\infty(M) \ and \ s(M)\subset \Delta\right\}\, ,\quad \forall \eta \in T_M,
\label{cara}
\end{equation}
where $H^\infty(M)$ is the complex vector space of all bounded holomorphic functions on $M$. When $\kappa_M$ is nondegenerate on $M$ $($i.e. it is a norm for every $x\in M)$, we call $\kappa_M$ the Carath\'eodory metric and we say that $M$ is Carath\'eodory hyperbolic.
\end{definition}

Clearly $\kappa_M$ is invariant under $\Aut(M)$. For a quotient $M'$ of $M$ by a torsion-free discrete group of automorphisms, $\kappa_M$ induces on $M'$ a quotient Carath\'eodory Finsler pseudometric.
The set ${\rm Hol}(M,\Delta)$ of holomorphic mappings $h:M\rightarrow \Delta$ may consist only of constant maps, in which case $\kappa_M \equiv 0$.

When $M=D\Subset \mathbb{C}^n$ is a bounded Euclidean domain, it is well-known that $\kappa_D$  is nondegenerate, and $\kappa_D$ is a continuous complex Finsler metric. Denoting by $\Theta(D,\kappa_D)$ the curvature current of the Carath\'eodory metric on $D$, we have $\Theta(D,\kappa_D)\le 0$ in the sense of currents. In other words, writing $\widehat{\kappa}_D$ for the continuous Hermitian metric on the tautological line bundle $L$ over $\mathbb PT_D$, and defining $\widehat{\kappa}_D(e;\overline{e}) = s$ for a nowhere vanishing local holomorphic section $e$, we have $-i\partial\overline{\partial}\log s \ge 0$ in the sense of currents, i.e., $\log s$ is a plurisubharmonic function.

\vskip 0.2cm

For instance, given a holomorphic map $f: X_\Gamma \to Y_{\Gamma'} = D/{\Gamma'}$ to a quotient of a bounded domain $D$ by a torsion-free discrete group of automorphisms, denoting by $\kappa_{Y_{\Gamma'}}$ the induced Carath\'eodory metric, and by $g$ the canonical complete K\"ahler-Einstein metric on $X_\Gamma$, then $g+f^*\kappa_{Y_{\Gamma}'}$ is a continuous complex Finsler metric of nonpositive curvature in the sense of currents.  The same applies when $Y_{\Gamma'}$ is replaced by a complex manifold $N$ and $D$ is replaced by the universal covering space $\widetilde{N}$ for the lifting $F: \Omega \to \widetilde{N}$ to covering spaces, since $g+f^*\kappa_{N}$ is a nondegenerate continuous complex Finsler metric in any event even though $\kappa_N$ is only a pseudometric.

\vskip 0.2cm
Now assuming $f: X_\Gamma \to N$ to be such that $F^*H^{\infty}(\widetilde N)$ satisfies the hypothesis of Theorem \ref{cara-embed}, which then applies to show that $dF(\eta) \neq 0$ for any $\eta \in T_{\Omega}$ which is a nonzero minimal rational tangent vector of one of the irreducible factors of $\Omega$. In Mok \cite{Mok2004} the first author made use of extremal bounded holomorphic functions on $\Omega$ to prove that $dF(\eta) \neq 0$ for any nonzero tangent vector $\eta \in T_\Omega$. To prove that $F$ separates points, an averaging technique on continuous complex Finsler metrics of nonpositive curvature was introduced in \cite{Mok2004} to define from $F^*H^{\infty}(\widetilde{N})$
a new complex Finsler metric on the tautological line bundle over $\mathbb{P}T_\Omega$ restricted to $\mathscr C(\Omega)$ when $\Omega$ is irreducible (with an obvious generalization to the case where $\Omega$ is reducible). In the case where $\Gamma \subset \Aut(\Omega)$ is a nonuniform lattice, the above arguments still apply since $g+F^*\kappa_{\widetilde N}$ is uniformly Lipschitz with respect to $g$ by Cauchy estimates.

\vskip 0.2cm
Theorem \ref{cara-embed} prompted the following question which was at the source of Theorem \ref{isom}.

\begin{problem}[The Extension Problem] \
\label{ext-problem}
In the setting of Theorem \ref{cara-embed} $($the Embedding Theorem$)$,
let $F:\Omega\to \widetilde{N}$ be the given $\Gamma$-equivariant holomorphic embedding.
Does there exist a bounded holomorphic map $R:\widetilde N \to \mathbb C^n$ such that $R\circ F \equiv \id_\Omega$?
\end{problem}

\vskip 0.2cm
\subsection{Methodological comparison with previous rigidity results concerning irreducible lattices \texorpdfstring{$\Gamma \subset \Aut(\Omega)$}{Gamma subset Aut(Omega)}}
\label{method-compare}
\mbox{}

\vskip 0.5cm
\subsubsection{The method of harmonic maps into \texorpdfstring{$X_\Gamma$}{X-Gamma} for cocompact lattices}
\label{harm-maps-cocompact}
Theorem \ref{strong-rigid} of Siu \cite{Siu1980} \cite{Siu1981} on strong rigidity is a characterization of compact quotient manifolds of irreducible bounded symmetric domains of dimension $\ge 2$ among compact K\"ahler manifolds by a topological condition, viz.,
by the condition that the given compact K\"ahler manifold $X$ is homeomorphic to $X_\Gamma := \Omega/\Gamma$ for an irreducible bounded symmetric domain $\Omega$ of rank $\ge 2$ and a torsion-free cocompact lattice $\Gamma \subset \Aut(\Omega)$.  In case of complex dimension $2$ even the K\"ahler condition can be dropped since every compact complex manifold with even first Betti number admits a K\"ahler structure.  The proof started with a smooth map $f_0:X \to X_\Gamma$ which is a deformation of a homeomorphism and proceeds with the method of harmonic maps which allows $f_0$ to be deformed to a harmonic map, which can then be proven to be a holomorphic or anti-holomorphic diffeomorphism by means of the $\partial\overline{\partial}$-Bochner-Kodaira formula. In the event that $\Omega$ is of rank $\ge 2$, combining Theorem \ref{strong-rigid} with Theorem \ref{hermite-rigid} in the compact case of Mok \cite{Mok1987}, we can conclude that, given $(X,h)$ compact K\"ahler of nonpositive holomorphic bisectional curvature with $X$ homeomorphic to $X_\Gamma$, $(X,h)$ is up to complex conjugation biholomorphically isometric to $(X_{\Gamma},g)$, thus yielding a rigidity statement both on the complex structure and on the K\"ahler metric of $X_\Gamma$. The starting point of Theorem \ref{strong-rigid} applies also to any smooth map $f_0: X \to X_\Gamma$, for which Siu's $\partial\overline{\partial}$-Bochner-Kodaira formula applies to prove that the harmonic map $f: X \to X_\Gamma$ is pluriharmonic, and can be shown to be holomorphic or anti-holomorphic by a study of the curvature tensor of $(X_\Gamma,g)$ under topological conditions guaranteeing the real rank of $f$ to be sufficiently large.

\vskip 0.2cm
In place of studying maps into $X_\Gamma$ for torsion-free cocompact lattices $\Gamma \subset \Aut(\Omega)$, one may in the opposite direction study maps defined from $X_\Gamma$ into a compact K\"ahler manifold $X$.
For the existing methods concerning harmonic maps it is nonetheless necessary to have target compact K\"ahler manifolds $(N,h)$ of nonpositive sectional curvature for harmonic maps to exist, and to assume a stronger curvature condition, viz., that $(N,h)$ is of nonpositive curvature in the dual sense of Nakano, in order to be able to deduce from Siu's $\partial\overline{\partial}$-Bochner-Kodaira formula that the harmonic map $f: X_\Gamma \to N$ is pluriharmonic.
The curvature assumption is very strong, and there are very few examples $(N,h)$ for which the theory of harmonic maps $f: X_\Gamma \to N$ applies other than the case where $N$ is itself Hermitian locally symmetric of the noncompact type.  Extending the scope to the case where the target $N$ is a Riemannian symmetric space of the compact type and of nonpositive sectional curvature in the complexified sense, it is possible to prove pluriharmonicity of a harmonic map $f:X_\Gamma \to N$ and deduce thereby that $f$ is a totally geodesic isometric embedding, as done in Mok \cite{Mok1991}, which is a special case of geometric rigidity of Mok-Siu-Yeung \cite{MSY1993}. One motivation for the authors to prove Theorem \ref{isom} from \cite{MW2025} is to enlarge the scope of target manifolds $N$ for which rigidity phenomena for holomorphic maps $f: X_\Gamma \to N$ can be examined, and that is where we started to weaken the assumption of nonpositivity in the dual Nakano sense to a nonpositivity assumption on the curvature of a complex Finsler metric, and where we started to bring in bounded holomorphic functions, since the infinitesimal Carath\'eodory metric on a bounded domain is a continuous complex Finsler metric of nonpositive curvature.

\vskip 0.2cm
\subsubsection{The entry of K\"ahler geometry, harmonic analysis and ergodic theory in the proof of the Isomorphism theorem}
\label{entry-methods}
With the very weak assumption of nonpositivity of curvature on complex Finsler metrics, we do not have at our disposal the method of harmonic maps.  We are therefore obliged to impose the assumption of holomorphicity on the maps $f:X_\Gamma \to N$ being studied. On the other hand, there are two aspects in which the context goes beyond that of the study of harmonic maps into $X_\Gamma$.  First of all, our method applies to the case of nonuniform lattices $\Gamma \subset \Aut(\Omega)$. Secondly, we do not assume $N$ to be K\"ahler in Theorem \ref{cara-embed} in general.
In the special case where $N = Y_{\Gamma'} = D/\Gamma'$ is one in which $D$ is a bounded domain of holomorphy, it is known from Cheng-Yau \cite{CY1980} and Mok-Yau \cite{MY1983} that $N$ carries a canonical complete K\"ahler-Einstein metric of negative Ricci curvature.
In this case Theorem \ref{isom} applies and,
assuming $D$ to be simply connected for convenience, we proved that a holomorphic map $f: X_\Gamma\to Y_{\Gamma'}$ inducing an isomorphism $f_*:\Gamma =\pi_1(X_\Gamma) \overset{\cong}\longrightarrow \pi_1(Y_{\Gamma'}) = \Gamma'$ is automatically a biholomorphism.  Thus, invoking the existence of bounded holomorphic functions on $D$, we were able to prove a rigidity theorem with a topological assumption solely on the {\it fundamental group\/}.

\vskip 0.2cm
While our context is to study holomorphic maps $f: X_\Gamma \to Y_{\Gamma'}$
the road leading to this problem was actually from Theorem \ref{finsler}
on Finsler metric rigidity and Theorem \ref{cara-embed} yielding
the Embedding Theorem under a weak hypothesis.
The latter results had their origin in Theorem \ref{hermite-rigid}
on Hermitian metric rigidity, hence the entirety of rigidity results
explained in this section integrate with the Isomorphism Theorem (Theorem \ref{isom}) as a whole.
The proof of Theorem \ref{finsler} brings in
Moore ergodicity theorem, and that of Theorem \ref{cara-embed}
brings in bounded holomorphic functions.
As will be seen, the proof of the Isomorphism Theorem
will rely to a larger extent on Moore's ergodicity theorem and on harmonic analysis in the study of boundary values of bounded holomorphic functions on the complex unit ball $\mathbb B^n$.
The fact that in place of harmonic analysis concerning boundary values on $\partial\Omega$ we can resort to the much better known generalized Fatou's theorem for $\mathbb B^n$
relies on the identification of the fibers of the limiting maps of hyperbolic flows onto maximal faces as images of holomorphic isometric embeddings of some complex unit ball $\mathbb B^{p(\Omega)+1}$, a result that originated from the study of varieties of minimal rational tangents of Mok \cite{Mok2016}.  In what follows we will trace some basics of these essential elements in the proof of the Isomorphism Theorem, viz., on holomorphic isometries in K\"ahler geometry in \S\ref{holiso}, on the study of boundary values of bounded holomorphic functions in \S\ref{harmenter}, and on an in-depth application of Moore's ergodicity theorem on semisimple Lie groups in \S\ref{ergodic}, before concluding with an overview of the proof of the Isomorphism Theorem in the concluding \S\ref{isomproof}.  As we hope to familiarize the reader with the subject of bounded symmetric domains and to facilitate the reader's understanding of the path leading to the formulation of the Isomorphism Theorem,
the background materials may go beyond what is absolutely necessary for reading the proof of the theorem in Mok-Wong \cite{MW2025}. We will have succeeded if this article allows the reader to understand our overall strategy and better appreciate the roles played by K\"ahler geometry, harmonic analysis and ergodic theory in the proof of the Isomorphism Theorem.

\vskip 0.2cm
\section[Holomorphic isometries with respect to the Bergman metric]{Holomorphic isometries between bounded domains\\ with respect to the Bergman metric}
\label{holiso}
\vskip 0.2cm
\subsection{Analytic continuation of germs of holomorphic isometries up to scaling constants}
\label{holiso-continuation}
\mbox{}

\vskip 0.5cm
\subsubsection{Genesis of the problem}
\label{genesis}
The systematic study of holomorphic isometries from a real-analytic K\"ahler manifold $(M,g)$ into complex space forms such as the projective space $\mathbb P^N$ or $\mathbb P^{\infty}$ equipped with the Fubini-Study metric originated from the seminal work of Calabi \cite{Ca1953}, in which he introduced the notion of the {\it diastasis\/}, a normalized potential function, and proved results of analytic continuation for germs of holomorphic isometries into such space forms.

\vskip 0.2cm
In the more recent past, motivated by questions of Clozel-Ullmo \cite{CU2003} in 2003 arising from
arithmetic dynamics concerning commutants of Hecke correspondences, it was desirable to prove extension theorems for germs of holomorphic isometries up to normalizing constants $f: (\Omega,\lambda\, ds_{\Omega}^2;x_0) \to (\Omega',ds_{\Omega'}^2;y_0)$ between bounded symmetric domains.
In \cite{CU2003} it was observed that, assuming $\Omega$ to be irreducible, all such germs of maps are necessarily totally geodesic whenever ${\rm rank}(\Omega)\ge 2$ as a consequence of the proof of Hermitian metric rigidity in Mok \cite{Mok1987}.  In Mok \cite{Mok2012} the first author considered the analogous problem of analytic continuation in the more general situation of bounded domains equipped with the Bergman metric.

\vskip 0.2cm
For a bounded domain $U$, we denote by $K_U(z,w)$ the Bergman kernel of $U$ defined by the Hilbert space $H^2(U)$ of square-integrable holomorphic functions. Then $\varphi_U(z) = \log K_U(z,z)$ is globally defined on $U$, and $i\partial\overline{\partial}\varphi_U =:\omega_U$ defines the K\"ahler form of the Bergman metric $ds_U^2$. Let $f: (D,ds_D^2;x_o) \to (D',ds_{D'}^2;x'_o)$ be a germ of a holomorphic map.  Because of the existence of global potential functions $\varphi_D$ and $\varphi_{D'}$,
\cite{Ca1953} applies to give interior extension, as follows. Choosing any orthonormal basis $\left(h_i\right)$ of $H^2(D')$, we define a holomorphic map $\Phi_{D'}: D' \to \mathbb P^\infty \cong  \mathbb P(H^2(D')^\star)$ by $\Phi_{D'}(\zeta) = \left[h_0(\zeta),\cdots,h_i(\zeta),\cdots\right]$.
The mapping $\Phi_{D'}\circ f: (D,ds_D^2;x_0) \to \big(\mathbb P(H^2(D')^\star), \frac 1{\phantom{,}\lambda\phantom{,}} ds_{FS}^2;\Phi_{D'}(x_0')\big)$ is a holomorphic isometric embedding into a projective space of countably infinite dimension equipped with the Fubini-Study metric of constant holomorphic sectional curvature equal to $2$.
Let $\mathbb P(\Lambda) \subset \mathbb P(H^2(D')^\star)$ be the topological projective-linear span of the image of $\Phi_{D'}\circ f$.  Then, results in \cite{Ca1953} on analytic continuation apply to
show that $\Phi_{D'}\circ f: (D,ds_D^2;x_o) \to \mathbb{P}(\Lambda)$ extends to give a holomorphic isometry from $D$ into $\mathbb P(\Lambda)$, and that in turn implies that $f: (D,ds_D^2;x_o)\to (D',ds_{D'}^2;x_o)$ extends to a holomorphic isometry $F: (D,ds_D^2) \to (D',ds_{D'}^2)$ provided that the K\"ahler metric $ds_{D'}^2$ is complete.

\vskip 0.2cm
The difficulty of the problem of analytic continuation for bounded domains $D \Subset \mathbb C^n$ and $D'\Subset \mathbb C^n$ therefore lies in the further analytic continuation of the map beyond the boundaries. In order to study this problem one has to specify realizations of $D$ and $D'$ as bounded domains, e.g., it is natural to use the Harish-Chandra realization in the case where $D = \Omega$ and $D' = \Omega'$ are bounded symmetric domains.

\vskip 0.2cm
\subsubsection{Algebraic extension of holomorphic isometries between bounded domains with rational Bergman kernels}
\label{algebraic-ext}
The first author considered in \cite{Mok2012} the problem of analytic continuation of germs of holomorphic isometries between bounded domains $U$ having the property that the Bergman kernel $K_U(z,w)$ is a rational function in $(z,\overline{w})$. Here is the main result in this context.

\begin{theorem}[{\rm Mok \cite{Mok2012}}] \
Let $D \Subset \mathbb C^n$ resp.\,$D' \Subset \mathbb C^{n'},$ be bounded domains.
Let $x_o \in D, \lambda \in \mathbb R, \lambda > 0$,
and $f:(D,\lambda ds_D^2;x_o) \to (D',ds_{D'}^2;x'_o)$, $x_o':= f(x_o)$, be a germ of holomorphic isometry.
Suppose the Bergman kernel $K_D(z,w)$ extends as a rational function in $(z,\overline{w})$
and likewise the Bergman kernel $K_{D'}(z',w')$ extends as a rational function in $(z',\overline{w'})$. Then, the germ of $\text{\rm Graph}(f) \subset D \times D'$ at $(x_o,x'_o)$
extends to an irreducible affine-algebraic subvariety $S^{\sharp} \subset \mathbb C^n\times \mathbb C^{n'}$. If $(D',ds_{D'}^2)$ is complete as a K\"ahler manifold, then $S := S^{\sharp} \cap (D\times D')$ is the graph of a holomorphic isometric embedding $F:(D,\lambda \, ds_D^2) \to (D',ds_{D'}^2)$. If furthermore $(D,ds_D^2)$ is complete, then $F:D\to D'$ is a proper map.
\label{hol-isom}
\end{theorem}
As in Clozel-Ullmo \cite{CU2003}, in the proof of Theorem \ref{hol-isom} the first author made use of the diastasis, except that in place of the real-analytic functional equations arising from equating diastases, he made use of holomorphic identities obtained from polarizing such functional identities.
Obviously the identities are related to the Bergman kernels $K_D$ and $K_{D'}$.
In the event that the variety $V$ of common solutions of the infinite family of holomorphic identities thus obtained,
which contains the germ of ${\rm Graph}(f) \subset D\times D'$ as a subset, has exactly the same dimension $n$ as that of ${\rm Graph}(f)$, one obtains immediately the desired analytic continuation.
The difficulty therefore arises when $V$ has excessive dimension, in which case it was shown in \cite{Mok2012} that $V$ can be cut down to a variety $W \subset \mathbb C^{n}\times\mathbb C^{n'}$
by requiring further that $(z,z') \in W$ satisfy $h_{\alpha}(z') = 0$ for an infinite family of extremal functions $h_\alpha \in H^2(D')$ which are necessarily restrictions of rational functions on $\mathbb C^{n'}$ under the rationality assumption on the Bergman kernel $K_{D'}(z',w')$ in the statement of the theorem.

\vskip 0.2cm
We note that the Bergman kernels $K_\Omega(z,w)$ of bounded symmetric domains $\Omega \Subset \mathbb C^n$ in their Harish-Chandra realizations are necessarily rational in $(z,\overline{w})$.
For instance, in the case of type-I domains $D^{I}_{p,q}$, we have
$K_{\Omega}(Z,W) = \frac{C_{p,q}}{\left(\det\left(I_q-\overline{W}^TZ\right)\right)^{p+q}}$ for some positive constants $C_{p,q}$,
which reduces in the case of $\mathbb B^n \equiv D^{I}_{n,1}$ to $K_{\mathbb B^n}(z,w) = \frac{C_n}
{\left(1-\left(z_1\overline{w_1}+\cdots+ z_n\overline{w_n}\right)\right)^{n+1}}$ for some positive constant $C_n = C_{n,1}$.

\begin{remark} \
\rm
The proof of Theorem \ref{hol-isom} given in \cite{Mok2012} does not make use of analytic continuation of holomorphic isometries into complex space forms, and as such it gives a self-contained proof that a K\"ahler submanifold of constant holomorphic sectional curvature of an arbitrary bounded symmetric domain $\Omega$ is necessarily biholomorphically isometric to the complex unit ball $(\mathbb B^n,ds^2_{\mathbb B^n})$, and not just locally holomorphically isometric to it. This fact will be used in the study of the examples of holomorphic isometries in \S\ref{ball-embeddings}.
\end{remark}

\vskip 0.2cm
\subsection{Holomorphic isometric embeddings of the complex unit ball of maximal dimension into an irreducible bounded symmetric domain: cones of minimal disks}
\label{ball-embeddings}
\mbox{}

\vskip 0.5cm
\subsubsection{The examples}
\label{examples}
It was for some time unknown whether there exist nonstandard $($i.e., not totally geodesic$)$ holomorphic isometric embeddings of a higher-dimensional complex unit ball $\mathbb B^m$, $m\ge 2$, into bounded symmetric domains.  The first author raised this question in \cite{Mok2011}, and found a series of examples in Mok \cite{Mok2016} which turn out to be important for rigidity problems, as follows.

\begin{theorem} {\rm (Mok \cite{Mok2016})} \
Let $\Omega \Subset \mathbb  C^N$ be an irreducible bounded symmetric domain of rank $\ge 2$ and
denote by $\widehat{\Omega}$ the irreducible Hermitian symmetric manifold of the compact type
dual to $\Omega$.  Denoting by $\delta \in H^2(\widehat{\Omega},\mathbb Z) \cong \mathbb  Z$
the positive generator of the second integral cohomology group of $\widehat{\Omega}$,
we write $c_1(\widehat{\Omega}) = (p+2)\delta$.  Then, there exists a nonstandard proper holomorphic isometric embedding $F: \left(\mathbb B^{p+1}, ds_{\mathbb B^{p+1}}^2\right) \hra \left(\Omega,ds_\Omega^2\right)$. More precisely, letting $\Omega \subset \widehat{\Omega}$ be the Borel embedding and denoting by $\mathcal V_x$ the union of all minimal rational curves on $\widehat{\Omega}$ passing through
a point $x \in \widehat{\Omega}$, for any smooth boundary point $a \in \Reg(\partial\Omega)$, the intersection $V_a := \mathcal V_a \cap \Omega$ is the image of a holomorphic isometric embedding $F_a: \left(\mathbb B^{p+1}, ds_{\mathbb B^{p+1}}^2\right) \hra \left(\Omega,ds_\Omega^2\right)$.
\label{holoisom-ball}
\end{theorem}
The Bergman metric on a bounded homogeneous domain is K\"ahler-Einstein and has Ricci curvature equal to $-1$.
In this article, we will mainly use another normalization of the K\"ahler-Einstein metric
on an irreducible bounded symmetric domain $\Omega$ in place of the Bergman metric, viz.,
we denote by $g_\Omega$ the canonical K\"ahler-Einstein metric on $\Omega$
such that minimal disks are of constant Gaussian curvature $-2$.
It turns out that Theorem \ref{holoisom-ball} remains valid when the Bergman metrics
are replaced by the normalized K\"ahler-Einstein metrics $g_{\mathbb B^{p+1}}$ and $g_\Omega$,
and the two statements (with Bergman metrics resp. normalized K\"ahler-Einstein metrics) are equivalent. For simplicity we also write $g_n$ for $g_{\mathbb B^{n}}$.

\vskip 0.2cm
The proof of Theorem \ref{holoisom-ball} is geometric.
On the one hand, $V_a$ is strictly pseudoconvex at smooth boundary points,
so that $V_a$ is asymptotically of constant holomorphic sectional curvature $-2$ at a smooth boundary point.
On the other hand, given any point $x\in V_a$ there is a
1-parameter group $H = \{\varphi_t: -\infty < t < +\infty\} \subset \Aut(\Omega)$ of automorphisms
preserving $V_a$ such that $\varphi_t(x)$ converges to some smooth point of $\partial V_a$.
These two properties of $V_a$ imply that $V_a$ is of constant holomorphic sectional curvature $-2$, hence there is a germ of holomorphic isometry $f: \left(\mathbb B^{p+1},g_{p+1};0\right)
\to \left(V_a,g_\Omega|_{V_a};x_0\right)$. By Calabi \cite{Ca1953} or by Theorem \ref{hol-isom}, $f$ admits an extension to a proper holomorphic isometric embedding $F: \left(\mathbb B^{p+1},g_{p+1}\right) \to \left(V_a,g_\Omega|_{V_a}\right)$, so that $\left(V_a,g_\Omega|_{V_a}\right)$ is globally and not just locally holomorphically isometric to $\mathbb B^{p+1}$ equipped with the normalized canonical K\"ahler-Einstein metric.

\begin{remark} \
\rm
It was also proven in \cite{Mok2016} that $p+1$, $p = p(\Omega)$,
is the maximal dimension for which there exist holomorphic isometric embeddings
from $\mathbb B^n$ into $\Omega$ with respect to the Bergman metric or with respect to the normalized K\"ahler-Einstein metric.
\end{remark}

\vskip 0.2cm
\subsubsection{Privileged Harish-Chandra coordinates}
\label{privileged-coords}
It is convenient to choose Harish-Chandra coordinates in a specific form by making some unitary transformation (which is just a permutation of the coordinates if we start with Harish-Chandra coordinates such that the basis vectors at $0$ are unit root vectors). Let $\Omega \Subset \mathbb C^n$ be the Harish-Chandra realization of $\Omega$. Let $\ell$ be a minimal rational curve passing through $0\in\Omega$ and modify the Harish-Chandra coordinates $(z_1,\cdots,z_n)$ by a unitary transformation so that $T_0(\ell) =\mathbb C\frac{\partial}{\partial z_1}$.  We will also write $\alpha := \frac{\partial}{\partial z_1}$. Writing $p = p(\Omega)$ for the complex dimension of $\mathscr C_0(\widehat{\Omega})$,
choose the coordinates $(z_2,\cdots,z_{p+1})$ so that, with respect to the Grothendieck decomposition $T_{\widehat{\Omega}}|_\ell = \mathcal O(2)\oplus\mathcal O(1)^p\oplus \mathcal O^q$, writing $P_\ell$ for $\mathcal O(2)\oplus\mathcal O(1)^p$, we have $P_0 = \mathbb C\frac{\partial}{\partial z_1}\oplus\left(\mathbb C\frac{\partial}{\partial z_2}\oplus\cdots\oplus \mathbb C\frac{\partial}{\partial z_{p+1}}\right)$. Then, we have Harish-Chandra coordinates $z = (z_1;z';z'')$ where $z' = (z_2,\cdots,z_{p+1})$ and $z'' = (z_{p+2},\cdots,z_n)$.  In terms of the curvature tensor $R_{i\overline{j}k\overline{\ell}}(0)$ of $(\Omega,g_0)$ normalized so that minimal disks are
of constant Gaussian curvature $-2$, $\left(R_{1\overline{1}k\overline{\ell}}\right)$
defines a Hermitian bilinear form $H_\alpha$ on $T_{\Omega,0}$ with eigenvalues $-2$, $-1$ and $0$ and corresponding eigenspaces $\mathbb C\frac{\partial}{\partial z_1}$
resp.~$\mathcal H_\alpha = {\rm Span}_{\mathbb C}\left\{\frac{\partial}{\partial z_2}\cdots\frac{\partial}{\partial z_{p+1}}\right\}$
 resp.~$\mathcal N_\alpha = {\rm Span}_{\mathbb C}\left\{\frac{\partial}{\partial z_{p+2}}\cdots\frac{\partial}{\partial z_{n}}\right\}$.
 Note that $\mathcal N_\alpha \subset T_{\Omega,0}$
is the null space of the Hermitian form $H_\alpha$. In the ``{\it
privileged coordinates\/}'' $z = (z_1,z',z'')$, there is a totally geodesic complex submanifold $\Delta\times\{0\}\times \Omega^{\flat} \subset \Omega$, where $\Omega^{\flat}$ is an irreducible bounded symmetric domain of rank $r-1$ biholomorphic to a maximal face $\Phi \subset \Reg(\partial\Omega)$. We have $\mathcal N_\alpha = T_0(\{(0;0)\}\times\Omega^{\flat})$.

\vskip 0.2cm
\subsubsection{Cones of minimal disks as fibers of limits of hyperbolic flows}
\label{conevmrt}
There is a well-known geometric interpretation of the inverse Cayley transformation $\alpha: \Delta \to \mathcal H$ by a 1-parameter hyperbolic flow along the geodesic $\Lambda=(-1,1)\subset\Delta$.  Let $\psi_t(z)=\frac{z+t}{1+tz}$, $t\in(-1,1)$, be the hyperbolic flow fixing the end points $\pm 1$.  As $t\to 1$, $\psi_t(z)\to 1$ for every $z\in\Delta$.  Composing on the left with affine maps $A_t$ so that $A_t\circ\psi_t$ is given by $z\mapsto i-\frac{2iz}{1+tz}$, one obtains as $t\to 1$
\[
\alpha(z)
=
i\frac{1-z}{1+z},
\]
which maps $\Delta$ onto $\mathcal H$ and sends $z=1,0,-1$ to $0,i,\infty$ respectively.\footnote{Composing further with the automorphism $w\mapsto -\frac{1}{w}$ of $\mathcal H$ $($equivalently, replacing $z$ by $-z$ on $\Delta$$)$ yields the other inverse Cayley transformation $i\frac{1+z}{1-z}$, which sends $z=-1,0,1$ to $0,i,\infty$.}

In \cite{MW2025} Mok-Wong related $\mathcal V_a$ for $a\in \Reg(\partial\Omega)$ with fibers of the limiting maps of hyperbolic flows onto maximal faces.  For the unit disk $\Delta$ there is a group homomorphism $\Theta: {\rm SU}(1,1) \to \Aut_0(\Omega) =: G_0$ defined using the root system such that, writing $L_0 = \Theta({\rm SU}(1,1))$, the $L_0$-orbit of $0\in\Omega$ is a minimal disk $D$, and embedding $L_0$ into its complexification $L$ the $L$-orbit of $0$ is a minimal rational curve $\ell \subset \widehat{\Omega}$, $\ell\cap\Omega = D$.  We write $\theta_t:=\Theta(\psi_t)$ and $\{\theta_t: -1 < t < 1\} \subset G_0$ for the resulting 1-parameter subgroup of transvections, which will also be referred to as a hyperbolic flow.

\begin{proposition}[{Mok-Wong \cite[{Proposition 3.4}]{MW2025} }] \
Let $\psi_t\in {\rm SU}(1,1)$ be given by $\psi_t(z):=\frac{z+t}{1+tz},t\in (-1,1)$, and write $\theta_{t,\Sigma}:=\Theta(\psi_t)$. Then, $\rho_{\Sigma}:=\lim\limits_{t\rightarrow 1} \theta_{t,\Sigma}$ exists, and $\rho_{\Sigma}: \Omega\rightarrow \Sigma$
is a holomorphic submersion onto the maximal face $\Sigma = \{(1;0)\} \times \Omega^\flat \subset \partial \Omega$. If $b\in \Sigma$, then $\rho_{\Sigma}(x)=b$ for any $x\in V_{b'}=\mathcal{V}_{b'}\cap \Omega$, where $b'$ is the opposite boundary point on $\Sigma'$ of $b \in \Sigma$
with respect to $0$, $\Sigma'$ being the opposite boundary of $\Sigma$ with respect to $0$.
Moreover, for each $x \in \Omega$, there exists a unique $b \in \Sigma$ such that $x \in V_{b'}$, so that $\rho_\Sigma^{-1}(b) = V_{b'}$ and the level sets $V_{b'} = \rho_{\Sigma}^{-1}(b), b \in \Sigma$, give a decomposition of $\Omega$ into a disjoint union $\Omega = \coprod\{V_{b'}: b' \in \Sigma'\} =  \coprod\{V_{b'}: b \in \Sigma\}$.
\label{cayproj}
\end{proposition}

\vskip 0.2cm
\begin{definition} \
In the notation of Proposition \ref{cayproj} and written in terms of the privileged Harish-Chandra coordinates, the pair
\[
\Sigma = \{(1;0)\}\times\Omega^\flat\, ,\qquad
\Sigma' = \{(-1;0)\}\times\Omega^\flat
\]
is called the \emph{standard (reference) pair} of maximal faces of $\Omega$ $($equivalently, of maximal boundary components of $\Omega)$.
\label{stdpair}
\end{definition}

\vskip 0.2cm
Let $\Phi \subset \Reg(\partial\Omega)$ be a maximal face. Then, there exists $k \in K$ such that $\Phi = k\Sigma$. If we define $\theta_{t,\Phi} := k\theta_{t,\Sigma}k^{-1}$, then $\rho_\Phi := \lim_{t \to 1} \theta_{t,\Phi} = k\rho_\Sigma k^{-1}$ is also a holomorphic submersion $\rho_\Phi : \Omega \to \Phi$. Proposition \ref{cayproj} remains valid with $\Sigma$ being replaced by $\Phi$, $\rho_\Sigma$ being replaced by $\rho_\Phi$ and $b' \in \Phi'$ meaning the opposite point of $b \in \Phi$ on the opposite face with respect to $0$.

\begin{definition}[{cf. Mok-Wong \cite[{Definition 3.9}]{MW2025} }] \
The holomorphic submersion $\rho = \rho_\Sigma : \Omega \to \Sigma$ in Proposition \ref{cayproj} onto $\Sigma$, or more generally $\rho = \rho_\Phi = k\rho_\Sigma k^{-1}$, $\rho : \Omega \to \Phi$ as in the last paragraph, is called a standard Cayley projection.
\label{standcay}
\end{definition}

More general Cayley projections $\rho_{\Phi,\Psi}$ exist for certain pairs $(\Phi,\Psi)$ in $\Reg(\partial\Omega)\times \Reg(\partial\Omega)$, as limits of 1-parameter subgroups $\{\theta_{t,\Phi,\Psi}\}_{-1<t<1}\subset \Aut(\Omega)$ constructed similarly to the standard pair $(\Sigma, \Sigma')$ of Definition \ref{stdpair},
cf. Mok-Wong \cite[Lemma 3.11, Proposition 3.12]{MW2025}.
The standard Cayley projections $\rho_\Phi: \Omega\rightarrow \Phi$ are exactly $\rho_{\Phi,\Phi'}: \Omega\rightarrow \Phi$,
where $\Phi'$ is the opposite face of $\Phi$ with respect to $0$.

\vskip 0.2cm
\subsubsection{The first partial inverse Cayley transformation in the type-III model}
\label{typeIII-cayley}
For an irreducible bounded symmetric domain of rank $r$ there is a tower of $r$ partial inverse Cayley transformations.  In the rank-one case of the ball $\mathbb B^n$, the $($full$)$ inverse Cayley transformation realizes $\mathbb B^n$ as the Siegel domain of the second kind
\[
\mathcal S_n
:=
\bigl\{(w;z_1,\cdots,z_{n-1})\in\mathbb C\times\mathbb C^{n-1}:
\im(w)>|z_1|^2+\cdots+|z_{n-1}|^2\bigr\}.
\]
Let $M_s(n;\mathbb C)$ denote the space of complex symmetric $n$-by-$n$ matrices and set
\[
D^{III}_n
:=
\bigl\{Z\in M_s(n;\mathbb C): I_n-Z\overline{Z}>0\bigr\},
\]
so that $r(D^{III}_n)=n$ and $\dim_{\mathbb C}D^{III}_n=\frac{n(n+1)}{2}$.  $($The \emph{full} inverse Cayley transformation $Z\mapsto i(I_n+Z)(I_n-Z)^{-1}$ realizes $D^{III}_n$ as the Siegel upper half-plane $\{\,Z\in M_s(n;\mathbb C):\im(Z)>0\,\}$; the first partial inverse Cayley transformation retains a bounded base factor.$)$\footnote{In dimension $1$ this is $i\frac{1+z}{1-z}$. Composing further with the automorphism $\tau\mapsto -\tau^{-1}$ of the Siegel upper half-plane $($equivalently, replacing $Z$ by $-Z$$)$ yields $i(I_n-Z)(I_n+Z)^{-1}$, which in dimension $1$ is the inverse Cayley transformation $i\frac{1-z}{1+z}$ of \S\ref{conevmrt}.}

\vskip 0.2cm
Decompose $Z\in D^{III}_n$ relative to $\mathbb C\oplus\mathbb C^{n-1}$ as
\[
Z
=
\begin{pmatrix}
a & b^T \\
b & C
\end{pmatrix},
\qquad
a\in\mathbb C,\quad
b\in\mathbb C^{n-1},\quad
C\in M_s(n-1;\mathbb C).
\]
We use these matrix blocks $(a;b;C)$ throughout the present subsection.  The corresponding privileged Harish-Chandra coordinates for the minimal rational direction $\alpha=\frac{\partial}{\partial a}$ are $\bigl(a;\sqrt{2}\,b;C\bigr)$ $($the factor $\sqrt{2}$ arising from root-length normalizations$)$.  The special product subspace $\Delta\times D^{III}_{n-1}\hra D^{III}_n$ is the locus $b=0$, and the standard opposite maximal faces of Definition~\ref{stdpair} are $\Sigma=\{(1;0)\}\times D^{III}_{n-1}$ and $\Sigma'=\{(-1;0)\}\times D^{III}_{n-1}$.

\vskip 0.2cm
Let $\psi_t(z)=\frac{z+t}{1+tz}$, $t\in(-1,1)$, be as in \S\ref{conevmrt}.  The embedding of $\psi_t$ into $\Aut(D^{III}_n)$ along the first long root is given by the linear fractional transformation $Z\mapsto(PZ+Q)(RZ+S)^{-1}$ on symmetric matrices, where $P^TS-R^TQ=I_n$, $P^TR=R^TP$ and $Q^TS=S^TQ$.  The element $\bigl(\begin{smallmatrix}\alpha&\beta\\ \beta&\alpha\end{smallmatrix}\bigr)\in{\rm SU}(1,1)$ with $\alpha=\frac{1}{\sqrt{1-t^2}}$ and $\beta=\frac{t}{\sqrt{1-t^2}}$ $($so $\alpha^2-\beta^2=1$$)$ embeds as $P=S=\operatorname{diag}(\alpha,I_{n-1})$ and $Q=R=\operatorname{diag}(\beta,0)$.  These blocks satisfy the three relations, and one obtains the automorphism
\begin{equation}
\label{eq:phi-t}
\varphi_t
\begin{pmatrix} a & b^T \\ b & C \end{pmatrix}
=
\begin{pmatrix}
\dfrac{a+t}{1+ta}
&
\dfrac{\sqrt{1-t^2}\,b^T}{1+ta}
\\[0.7em]
\dfrac{\sqrt{1-t^2}\,b}{1+ta}
&
C-\dfrac{t}{1+ta}\,bb^T
\end{pmatrix}.
\end{equation}
This is the transvection $\theta_{t,\Sigma}$ of Proposition~\ref{cayproj}.  As $t\to 1$ one has $a\mapsto 1$, $b\mapsto 0$, and $C\mapsto C_{\infty}(a;b;C):=C-\frac{bb^T}{1+a}$.  The denominator is nonzero because the $(1,1)$-entry of $I_n-Z\overline{Z}$ forces $|a|<1$.  The limit $\varphi_t(Z)\to\bigl(1;0;C_{\infty}(a;b;C)\bigr)$ lies on the face $\Sigma$ if and only if $C_{\infty}(a;b;C)\in D^{III}_{n-1}$.  To see the latter, take any $\xi\in\mathbb C^{n-1}\setminus\{0\}$ and set $x=-\frac{b^T\xi}{1+a}$ together with
\[
v
=
\begin{pmatrix} x \\ \xi \end{pmatrix},
\qquad
Zv
=
\begin{pmatrix} -x \\ C_{\infty}(a;b;C)\,\xi \end{pmatrix}.
\]
Then $v^*\bigl(I_n-\overline{Z}Z\bigr)v=\xi^*\bigl(I_{n-1}-\overline{C_{\infty}(a;b;C)}\,C_{\infty}(a;b;C)\bigr)\xi$.  The left-hand side is positive, since $I_n-Z\overline{Z}>0$ if and only if $I_n-\overline{Z}Z>0$, so $I_{n-1}-C_{\infty}(a;b;C)\,\overline{C_{\infty}(a;b;C)}>0$.  Thus $C_{\infty}(a;b;C)$ is a coordinate on $\Sigma$, and $\lim_{t\to 1}\varphi_t$ is the standard Cayley projection $\rho_\Sigma$ of Definition~\ref{standcay}.

\vskip 0.2cm
Composing $\varphi_t$ on the left with affine maps which act on the first coordinate as in the case of the unit disk in \S\ref{conevmrt} and which dilate the $b$-block by $\frac{1}{\sqrt{1-t^2}}$, one obtains
\[
\Psi_t(a;b;C)
=
\Bigl(
i-\frac{2ia}{1+ta};\,
\frac{b}{1+ta};\,
C-\frac{t}{1+ta}\,bb^T
\Bigr).
\]
As $t\to 1$,
\begin{equation}
\label{eq:alpha1-typeIII}
\alpha_1(a;b;C)
:=
\lim_{t\to 1}\Psi_t(a;b;C)
=
\Bigl(i\frac{1-a}{1+a};\,
\frac{b}{1+a};\,
C_{\infty}(a;b;C)\Bigr).
\end{equation}
Write $w:=i\frac{1-a}{1+a}$ and $u:=\frac{b}{1+a}$, so $\alpha_1=(w;u;C_{\infty}(a;b;C))$.  The first coordinate maps $\Delta$ onto the upper half-plane $\mathcal H$: for $|a|<1$ one has $\im w=\frac{1-|a|^2}{|1+a|^2}>0$, and it sends $a=1,0,-1$ to $w=0,i,\infty$.  On $b=0$ one has $C_{\infty}(a;0;C)=C$, hence $\alpha_1(a;0;C)=\bigl(i\frac{1-a}{1+a};0;C\bigr)$.

\vskip 0.2cm
Let $\mu_1\bigl(w;u;C_{\infty}(a;b;C)\bigr)=C_{\infty}(a;b;C)$ be the Euclidean projection of $\alpha_1(D^{III}_n)$ onto the third factor $\Omega^\flat=D^{III}_{n-1}$.  Writing $\iota:\Omega^\flat\to\Sigma$ for the identification $W\mapsto(1;0;W)$, one has
\[
\rho_\Sigma
=
\iota\circ\mu_1\circ\alpha_1.
\]
For each $W\in D^{III}_{n-1}$,
\begin{align*}
(\mu_1\circ\alpha_1)^{-1}(W)
&=
\rho_\Sigma^{-1}\bigl(1;0;W\bigr)
=
\bigl\{
Z\in D^{III}_n:
C_{\infty}(a;b;C)=W
\bigr\}
\\
&=
\bigl\{
Z\in D^{III}_n:
C=W+\frac{bb^T}{1+a}
\bigr\}.
\end{align*}
On $b=0$ this graph reduces to $Z=\operatorname{diag}(a,W)$, which lies in $D^{III}_n$ if and only if $|a|<1$, and $\alpha_1$ maps it onto $\bigl\{(w;0;W):\im w>0\bigr\}\cong\Delta$.

\vskip 0.2cm
Inverting the first two coordinates of $\alpha_1$ gives $a=\frac{i-w}{i+w}$ and $b=(1+a)u=\frac{2iu}{i+w}$.  Write $G=(I_{n-1}-W\overline{W})^{-1}$, which is positive definite Hermitian.  Substituting the graph into $I_n-Z\overline{Z}>0$, the image of $\rho_\Sigma^{-1}(1;0;W)$ under $\alpha_1$ is the fiber
\[
\mu_1^{-1}(W)
=
\bigl\{
(w;u):
\im w
>
2\,u^*Gu
+
2\operatorname{Re}\bigl(u^T\overline{W}\,Gu\bigr)
\bigr\}
\]
as a domain in the $(w;u)$-space.  $($For $n=2$, after returning to $(a;b)$, the same inequality is $\det\bigl(I_2-Z\overline{Z}\bigr)=(1-|a|^2)(1-|W|^2)-2|b|^2-2\operatorname{Re}\bigl(\frac{\overline{W}b^2(1+\overline{a})}{1+a}\bigr)>0$.$)$  On the fiber $W$ is constant.  Setting
\[
w'
=
w-2i\,u^T\overline{W}\,Gu,
\qquad
\eta
=
\sqrt{2}\,G^{1/2}u
\]
one has $\im w'=\im w-2\operatorname{Re}(u^T\overline{W}\,Gu)$ and $2u^*Gu=\|\eta\|^2$, hence the fiber is the standard Siegel domain $\{\im w'>\|\eta\|^2\}$.  The inverse Cayley transformation $z_1=\frac{i-w'}{i+w'}$, $z'=(1+z_1)\eta$ maps this domain onto the Euclidean ball $\bigl\{|z_1|^2+\|z'\|^2<1\bigr\}\cong\mathbb B^n$.  When $W=0$ the inequality reduces to $\im w>2\|u\|^2$, or equivalently $|a|^2+2\|b\|^2<1$, which is the Euclidean ball in the privileged Harish-Chandra coordinates $\bigl(a;\sqrt{2}\,b;C\bigr)$.  In all cases $\dim\mu_1^{-1}(W)=n=\dim D^{III}_n-\dim D^{III}_{n-1}$, in agreement with $p(D^{III}_n)=n-1$ and Theorem~\ref{holoisom-ball}.

\vskip 0.2cm
Thus each fiber of the standard Cayley projection is a ball.  The same pattern holds for a general irreducible bounded symmetric domain of rank $\ge 2$: the first partial inverse Cayley transformation realizes $\Omega$ as an unbounded domain $\mathscr D_1$ fibered over a rank-$(r-1)$ factor $\Omega^\flat$ by Siegel domains of the second kind, each biholomorphic to $\mathbb B^{p(\Omega)+1}$.  The geometric construction of Theorem~\ref{holoisom-ball} in \cite{Mok2016} establishes this identification uniformly for all types without relying on matrix formulae.

\vskip 0.2cm
\subsubsection{Boundary behavior of the biholomorphism \texorpdfstring{$V_a\cong \mathbb B^{p+1}$}{Va cong ball B(p+1)}}
\label{bdy-bihol}
For the $m$-dimensional complex unit ball $\mathbb B^m$, recall that $g_m$ stands for the normalized K\"ahler-Einstein metric $g_{\mathbb B^m}$. For a point $a \in \Reg(\partial\Omega)$, we know from Theorem \ref{holoisom-ball} that $\left(V_a,g_{\Omega}\right)$ is biholomorphically isometric to $\left(\mathbb B^{p+1},g_{p+1}\right)$.
We now examine the boundary behavior of the biholomorphic isometries $\theta: V_a \overset{\cong}\longrightarrow \mathbb B^{p+1}$ for $a \in \Reg(\partial\Omega)$.
Write $\mathcal K$ for the moduli space of minimal rational curves $\ell$ on the Borel dual $\widehat{\Omega}$ of $\Omega$ and denote by $\mathscr D \subset\mathcal K$ the open subset defined by $\ell\cap\Omega\neq\emptyset$. For $a\in\Reg(\partial\Omega)$, write $\mathscr D_a:=\{[\ell]\in\mathscr D: a\in\ell\}$. For $[\ell] \in \mathscr D$, $D_\ell := \ell \cap \Omega \subset \Omega$ is a totally geodesic holomorphic isometric image of the Poincar\'e disk.  Moreover, $\left(D_\ell,g_\Omega|_{D_\ell}\right) \hra (V_a, g_\Omega|_{V_a})$ is also totally geodesic.
Hence, writing $\Delta_\ell := \theta(D_\ell)$, we have $\left(\Delta_\ell,g_{p+1}|_{\Delta_\ell}\right) \hra \left(\mathbb B^{p+1},g_{p+1}\right)$. Writing $\theta_\ell := \theta|_{D_\ell}$, we have $\theta_\ell: D_\ell \overset{\cong}\longrightarrow \Delta_\ell$, and $\Delta_\ell \subset \mathbb B^{p+1}$ is also totally geodesic. Regarding the boundary behavior of the holomorphic isometry $\theta: \left(V_a, g_\Omega|_{V_a}\right) \overset{\cong}\longrightarrow \left(\mathbb B^{p+1},g_{p+1}\right)$ near the vertex $a$ of $\mathcal V_a$ we have,

\begin{proposition}[{Mok-Wong \cite[{Proposition 3.7}]{MW2025} }] \
For every $[\ell] \in \mathscr D_a$, $\theta_\ell$ extends to a biholomorphism $\theta_\ell^\sharp: \ell \overset{\cong}\longrightarrow \Lambda$, where $\Lambda \subset \mathbb P^{p+1}$
is a projective line and $\Delta_\ell \subset \Lambda$ is the Borel embedding. There exists $u \in \partial \mathbb B^{p+1}$ such that $\theta_\ell^\sharp(a) = u$ for every $[\ell] \in \mathscr D_a$.
Moreover, defining $\partial^\flat V_a = \coprod\left\{\partial D_\ell - \{a\}: [\ell] \in \mathscr D_a\right\}$, $\theta: V_a \overset{\cong} \longrightarrow \mathbb B^{p+1}$ extends holomorphically to a biholomorphism
$\theta^\sharp: W \overset{\cong}\longrightarrow U$ where $W$ is a
neighborhood of $V_a \coprod \partial^\flat V_a$ in $\widehat{\Omega}$ and $U$ is a neighborhood of $\mathbb B^{p+1} -\{u\}$ in $\mathbb P^{p+1}$. Furthermore, $\theta^\sharp\big|_{V_a \coprod \partial^\flat V_a}$ extends continuously to $\theta^\dagger: \overline{V_a} \to \overline{\mathbb B^{p+1}}$
$($such that $\theta^\dagger|_{V_a\coprod \partial^\flat V_a}: V_a\coprod \partial^\flat V_a \overset{\cong}\longrightarrow \overline{\mathbb B^{p+1}} - \{u\}$
is a diffeomorphism between smooth manifolds with boundary$)$.
\label{ballext}
\end{proposition}

\begin{remark} \
\rm
For $\Omega$ irreducible and of rank $\ge 2$, the closed set $E :=\partial V_a - \partial^\flat V_a \subset\partial\Omega$ is not a point. The proposition says in particular that the map $\theta^\dagger$, already defined up to $\partial^\flat V_a$ by the earlier part of the proposition, extends continuously to $\overline{V_a}$ when we define $\theta^\dagger(e) = u$ for any point $e\in E$.
\label{cont-ext}
\end{remark}

\vskip 0.2cm
\section{Harmonic analysis enters: admissible boundary limits of bounded holomorphic functions}
\label{harmenter}
\vskip 0.2cm
\subsection{Generalized Fatou's theorem on the complex unit ball}
\label{fatou-ball}
\mbox{}

\vskip 0.5cm
\subsubsection{Boundary values on the complex unit ball}
\label{bdy-values-ball}
\begin{definition} {\rm (admissible boundary values)} \
For $\alpha>1$ and $\xi\in \partial \mathbb{B}^n$, define $D_\alpha(\xi) := \left\{z\in \mathbb{B}^n: \frac{|1-\langle z,\xi\rangle|}{1-|z|^2}<\frac{\alpha}{2}\right\}$, $\langle \cdot,\cdot\rangle$ being the Euclidean Hermitian inner product.
Let $f$ be a complex-valued function on $\mathbb{B}^n$. We say that $f$ converges admissibly to $\lambda\in \mathbb{C}\cup \{\infty\}$ at $\xi$ if for any $\alpha>1$ and any sequence $\{p_n\}\subset D_\alpha(\xi)$ such that $p_n\rightarrow \xi$ as $n\rightarrow \infty$, we have $\lim_{p_n\rightarrow \xi} f(p_n)=\lambda$.
\label{admissible}
\end{definition}

The above definition is that of Kor\'anyi \cite{Kor1972}. From the above and Furstenberg \cite{Fur1963} we have

\begin{theorem} \
{\rm (generalized Fatou's theorem)}
Let $h \in H^\infty(\mathbb{B}^n)$. Then, $h$ converges admissibly to some $g\in L^\infty(\partial \mathbb{B}^n)$ almost everywhere on $\partial \mathbb{B}^n$.
\label{fatou}
\end{theorem}

\vskip 0.2cm
\subsubsection{The problem of finding limits of bounded holomorphic functions on \texorpdfstring{$\Omega$}{Omega} under Cayley projections}
\label{limits-cayley}
Suppose $\{\eta_t\}\subset \Aut(\mathbb{B}^n)$ is a 1-parameter subgroup of transvections, i.e., a hyperbolic flow. By Cartan's theorem there is a subsequence of $\{\eta_t\}$ converging uniformly on compact subsets to a constant function $\eta:\mathbb{B}^n\rightarrow \{\xi\}\subset \partial \mathbb{B}^n$. Let $p\in \mathbb{B}^n$. We observe that the trajectory $\{\eta_t(p)\}$ lies inside an admissible domain.  Noting that $D_\alpha(\xi)$ is invariant under unitary transformations, we only need to consider the boundary point $\xi=e_1:=(1,0,\dots, 0)\in \partial \mathbb{B}^n$ and the flow
\[
\eta_t(z)
:=
\Bigl(
\frac{z_1+t}{1+tz_1},\,
\frac{\sqrt{1-t^2}}{1+tz_1}z_2,\,
\dots,\,
\frac{\sqrt{1-t^2}}{1+tz_1}z_n
\Bigr),
\qquad
0\le t<1,
\]
for $z=(z_1,\dots,z_n)\in \mathbb{B}^n$. Write $w=(w_1,\dots,w_n)=\eta_t(z)$.
We have
\[
\frac{|1-\langle w,e_1\rangle|}{1-|w|^2}=\frac{|1-w_1|}{1-|w|^2}
=\frac{|1+tz_1|}{1+t}\frac{|1-z_1|}{1-|z|^2}
:=\beta(t,z_1)\frac{|1-z_1|}{1-|z|^2},
\]
where $\beta$ is a positive continuous function on $[0,1]\times \Delta$.
For each fixed $p=(p_1,\dots,p_n)\in \mathbb{B}^n$, there exists $M_{p_1}>0$ such that $\beta(t,p_1)<M_{p_1}<\infty$.
So
\[
\{\eta_t(p) \mid 0\leq t<1\}\subset D_{\beta_p}(e_1),
\quad \text{where} \,\, \beta_p=2M_{p_1}\frac{|1-p_1|}{1-|p|^2},
\]
as desired.  Now for each $\xi \in \partial \mathbb{B}^n$, one obtains a corresponding 1-parameter subgroup $\{\eta_{t,\xi}\}\subset \Aut(\mathbb{B}^n)$ a subsequence of which converges to the constant map $\eta_\xi:\mathbb{B}^n\rightarrow \{\xi\}$. By the generalized Fatou theorem on $\mathbb B^n$, $\lim\limits_{t\rightarrow 1} \eta_{t,\xi}^*h (z)$ exists for almost all $\xi\in \partial \mathbb{B}^n$.

\vskip 0.2cm
Recall the standard reference pair $(\Sigma,\Sigma')$ of maximal faces in Definition \ref{stdpair}.
From Proposition \ref{cayproj} there is a Cayley projection $\rho_{\Sigma}:\Omega \to \Sigma \cong \Omega^\flat$ whose fibers $V_{b'} = \rho_{\Sigma}^{-1}(b)$ for $b\in\Sigma$ (with $b'$ opposite to $b$) are biholomorphic to the complex unit ball $\mathbb B^{p+1}$, $p = p(\Omega)$ being the fiber dimension of the VMRT structure $\pi:\mathscr C(\widehat{\Omega}) \to \widehat{\Omega}$. Let now $h \in H^\infty(\Omega)$ be a given nonconstant bounded holomorphic function.
Suppose for every point $b\in\Sigma$, the function $h|_{V_{b'}}$ has an admissible boundary value $s_\Sigma(b)$ at $b\in\Sigma$, then, for the 1-parameter group of transvections $H = \{\eta_t: -1 < t < 1\}$ satisfying $\eta_t(z;0;w) = \left(\frac{z+t}{1+tz};0,w\right)$ in privileged Harish-Chandra coordinates,
the limit $\lim_{t\to 1} h(\eta_t(z;z';w))$ is well-defined, and we will have obtained a bounded holomorphic function $\hat{h} = \rho_\Sigma^*s_\Sigma$ on $\Omega$ which depends on fewer variables. By the proof of Montel's theorem on normal families, it is sufficient to assume that admissible boundary values at $b\in \Sigma$ of $h|_{V_{b'}}$ exist for almost every point $b\in\Sigma$.

\vskip 0.2cm
To answer Problem \ref{ext-problem} (the Extension Problem) in the affirmative, i.e., finding a bounded holomorphic map $R: \widetilde{N} \to \mathbb C^n$ such that $R\circ F \equiv \id_{\Omega}$, it is equivalent to find $\mu_1,\cdots, \mu_n\in H^\infty(\widetilde{N})$ such that $F^*\mu_i = z_i$ on $\Omega$ for $1\le i\le n$. We define $\mathscr F$ to be the algebra of bounded holomorphic functions on $\Omega$ obtained by pulling back bounded holomorphic functions on $\widetilde{N}$ by $F: \Omega \to \widetilde{N}$ and examine its properties. The first problem is to show that, composing by some automorphism of $\Omega$ if necessary, there exists $h \in \mathscr F$ such that $h|_{V_{b'}}$ has admissible boundary values almost everywhere on $\Sigma$.

\vskip 0.2cm
\subsection{The algebra \texorpdfstring{$\mathscr F = F^*H^{\infty}(\widetilde{N})$}{script-F = pullback of bounded H-infinity} of bounded holomorphic functions}
\label{algF}
We gather first of all elementary properties on the algebra $\mathscr F = F^*H^{\infty}(\widetilde{N})$ of bounded holomorphic functions which follow from its very definition and from Montel's theorem on normal families of holomorphic functions.

\begin{lemma} \
For the algebra $\mathscr F = F^*H^\infty(\widetilde{N})$ the following holds:
\begin{enumerate}
\item[{\rm (a)}]
For any $h \in \mathscr F$ and for any group element $\gamma \in \Gamma$, we have $\gamma^*h \in \mathscr F$.
\item[{\rm (b)}]
Suppose $h_i = F^*u_i\in \mathscr F$ for $0 \le i < \infty$ such that $\{u_i\}$ is uniformly bounded in $H^\infty(\widetilde{N})$ and such that $h_i$ converges uniformly on compact subsets to $h \in H^\infty(\Omega)$. Then, $h \in \mathscr F$.
\item[{\rm (c)}]
Let $Q$ be a manifold equipped with a smooth volume form $d\mu$ such that ${\rm Volume}(Q,d\mu) < \infty.$  Suppose $H: \Omega \times Q \to \mathbb C$ is a bounded measurable function on $\Omega\times Q$ such that for a.e. parameter $t \in Q$, the function $h_t: \Omega \to \mathbb C$ is a bounded holomorphic function belonging to $\mathscr F$. Then, $f: \Omega \to \mathbb C$ defined by $f(z) = \int_Q H(z,t) d\mu(t)$ also belongs to $\mathscr F$.
\end{enumerate}
\label{invariant-alg}
\end{lemma}
In relation to the question of finding $h\in \mathscr F = F^*H^\infty(\widetilde{N})$, we need to study the space of all Cayley projections on $\Omega$ onto maximal faces.  For the standard reference pair $(\Sigma,\Sigma')$ of Definition \ref{stdpair}, the hyperbolic flow $H$ fixes both $\Sigma$ and $\Sigma'$. For $w\in \Omega^\flat$, write $b(w) := (1;0;w)$, $b'(w) = (-1;0;w)$ and define $\Lambda(w) \subset V_{b'(w)}$ to be $(-1,1)\times\{0\}\times\{w\}$, which is a real geodesic curve on $V_{b'(w)}$.
Denote by $\mathscr L$ the union of $\Lambda(w)$ as $w$ runs over $\Omega^\flat$.
The hyperbolic flow $H$ fixes each $\Lambda(w)$ and $\theta_{t,\Sigma}(x) \to b = (1;0;w)$ for any $w\in\Omega^\flat$ and for any point $x\in V_{b'(w)}$.  In particular, this holds true for $x\in \Lambda(w)$.
  Let $g \in G_0$, then we have also a hyperbolic flow $H^g$ defined by $H^{g} = gHg^{-1}$, so that for $t \in (-1,1)$, and writing $H^g = \{\eta_t^g: t\in (-1,1)\}$, $\eta^g_t(x) = g(\eta_t(x))$. Thus, $H^g$ is a 1-parameter family of transvections preserving $g\mathscr L$ and flowing from $gb'(w)$ to $gb(w)$.  Hence $H^g$ fixes $\Phi = g\Sigma$ and $\Psi = g\Sigma'$, and it flows from $\Psi$ to $\Phi$. Equivalently, we are looking for $g\in G_0$ such that for the function $g^*h$ defined by $g^*h(x) = h(gx)$, and for the standard Cayley projection $\rho$, $g^*h$ has admissible boundary values almost everywhere both on $\Sigma$ and on $\Sigma'$.

\vskip 0.2cm

Given a single nonconstant bounded holomorphic function $h \in \mathscr F$, we need to look for $g\in G_0$ such that for the Cayley projection $\rho$ associated to the one-parameter family of transvections $H^g$ flowing from $\Psi$ to $\Phi$, admissible boundary values of $h$ exist both on $\Phi$ and on $\Psi$. We will need to require $g\in G_0$ to satisfy stronger properties in order to be able to construct new functions in $\mathscr F$ so as to recover the linear functions $z_1,\cdots,z_n$ as members of $\mathscr F$.  This is what we have proven in \cite{MW2025}. Nonetheless, to render the arguments more intuitive, we have chosen to give more geometric rather than just abstract arguments. This has the advantage that certain intermediate results are interesting in their own right.

\vskip 0.2cm
We have referred to $\Sigma'$ as the maximal face opposite to $\Sigma$ with respect to $0$
as in Proposition \ref{cayproj}.  More properly, we should refer to $\Sigma'$ as the maximal face opposite to $\Sigma$ with respect to the section $\{(0;0)\}\times\Omega^\flat$ of the totally geodesic complex submanifold $\Delta\times\{0\}\times\Omega^\flat \subset \Omega$ fibered over $\Omega^\flat$, the former being the image of $\Delta\times\Omega^{\flat}$ under a standard embedding into $\Omega$ defined by root systems.  By \cite[Lemma 3.11, Lemma 4.4]{MW2025}, for any pair of maximal faces $(\Phi,\Psi)$ on $\Reg(\partial\Omega)$ such that $\overline{\Phi} \cap \overline{\Psi} = \emptyset$, there exists $g\in G_0$ such that $g\Sigma = \Phi$ and $g\Sigma' = \Psi$.  For this reason it is more proper to denote the standard reference Cayley projection as $\rho_{\Sigma,\Sigma'}$ and the general one as $\rho_{\Phi,\Psi}$ as these maps are completely determined by the choice of the two maximal faces on $\Reg(\partial\Omega)$ and the orientation of the flow $H^g$.  In fact, in our application of Moore's ergodicity theorem to be explained in \S\ref{ergodic}, it will be necessary to consider in general both Cayley projections $\rho_{\Phi,\Psi}:\Omega \to \Phi$ and $\rho_{\Psi,\Phi}:\Omega \to\Psi$.

\vskip 0.2cm
\subsubsection{Cayley limits}
\label{cayley-limits}
Let $\Omega$ be a bounded symmetric domain and $\Gamma\subset \Aut_0(\Omega) = G_0$ be a torsion-free irreducible lattice, and write $X_\Gamma:=\Omega/\Gamma$. Let $P\subset \Omega$ be a maximal polydisk of $\Omega$. Using root space decompositions, we have a standard homomorphism $\varpi: {\rm SU}(1,1)\to \Aut_0(\Omega)$. Write $\eta_t(z)=\frac{z+t}{1+tz}$ so that, writing $\theta_t := \varpi(\eta_t,e,\cdots,e)$, where $e\in{\rm SU}(1,1)$ stands for the identity element, $H :=\{\theta_t: -1 < t < 1\}\hra \Aut_0(\Omega) = G_0$ it is a noncompact 1-parameter closed subgroup of transvections.

\vskip 0.2cm
In the case of $\Omega = \Delta$, the 1-parameter group $H$ gives the hyperbolic flow which preserves the interval $(-1,1) = \Delta\cap\mathbb R$ which is a geodesic on $(\Delta,ds^2_\Delta)$. We may reparametrize $H$ as $\{\varphi_s: -\infty < s < +\infty\}$ so that,
parametrizing $(-1,1)$ by $\gamma: (-\infty,+\infty) \to (-1,1)$ by geodesic distances such that $\varphi(0) = 0$, $H|_{(-1,1)}$ gives the geodesic flow along $(-1,1) = \gamma(\mathbb R)$ given by $\varphi_s(\gamma(s_0)) = \gamma(s_0+s)$. For any geodesic $\Lambda \subset \Delta$ with respect to $(\Delta,ds_{\Delta}^2)$, $\overline{\Lambda} \cap \partial\Delta = \{b, b'\}$ the 1-parameter flow $\{\psi_s\}_{s\in\mathbb R}$ preserving $\Lambda$ and restricting to the geodesic flow on $\Lambda$ will be referred to as a hyperbolic flow. Given $\Lambda$, such a flow is unique up to orientation.  In the general case of a bounded symmetric domain $\Omega$, in this article by a 1-parameter group of transvections or a hyperbolic flow we will mean the flow defined by $\{\theta_t\}_{t\in (-1,1)}$ as in the last paragraph, or $\{\psi_s\}_{s\in\mathbb R}$, where $\psi_s := \varpi(\varphi_s,e,\cdots,e)$. The two descriptions of $H$ only differ by reparametrizing $(-1,1)$ in the former by $(-\infty,+\infty)$ in the latter, each with its own advantages.

\vskip 0.2cm
Let $\Omega$ be a bounded symmetric domain of rank $r\geq 2$,
and $(\Phi,\Psi)$ be a pair of maximal faces,
by Proposition \ref{cayproj} and the paragraphs following it, we obtain the Cayley projection
$\lim\limits_{t\rightarrow 1} \theta_{t,\Phi,\Psi}
= \rho_{\Phi,\Psi}:\Omega\rightarrow \Phi$,
where $\{\theta_{t,\Phi,\Psi}\}\subset \Aut(\Omega)$
is a 1-parameter subgroup of transvections.
In analogy to the notion of admissible limits on $\mathbb{B}^n$, we have

\begin{definition} \
Let $h$ be a function defined on a bounded symmetric domain $\Omega$ of rank $r\geq 2$.
Let $\{\theta_{t,\Phi,\Psi}\}_{-1<t<1}\subset \Aut(\Omega)$ be the 1-parameter subgroup of transvections as defined in \S\ref{conevmrt} so that $\lim\limits_{t\rightarrow 1}\theta_{t,\Phi,\Psi}=\rho_{\Phi,\Psi}:\Omega \rightarrow \Phi$ is a Cayley projection. Then,
for each  point $x\in \Omega$, we write $\hat h_{\Phi,\Psi}(x):=\lim\limits_{t\rightarrow 1}\theta_{t,\Phi,\Psi}^*h(x) := \lim\limits_{t\rightarrow 1}h(\theta_{t,\Phi,\Psi}(x))$ if the limit in the sense of uniform convergence on compact subsets of $\Omega$ exists, and call it the Cayley limit of $h$ at $x$ with respect to $\rho_{\Phi,\Psi}$.
\end{definition}

The moduli space $\mathfrak M$ of faces $\Phi\subset \Reg(\partial\Omega)$ is of the form $G_0/N$ for some maximal parabolic subgroup $N\subset G_0$ \cite[p. 298, Corollary 1]{Wo1972}, thus in particular $\mathfrak M$ is a compact homogeneous manifold.
In fact, $\mathfrak M$ is homogeneous under the isotropy subgroup $K\subset \Aut(\Omega)$ at $0$, which is compact (cf. the Boundary Flag Theorem in \cite[p. 299]{Wo1972}).

\vspace{0.2cm}
Let $(\Phi,\Psi)$ be a pair of maximal faces for which $\rho_{\Phi,\Psi}$ is defined, in which case $(\Phi,\Psi) = (g\Sigma, g\Sigma')$ for some $g \in G_0$. As in the description of Definition \ref{standcay} and the paragraphs following it, we write $\xi_t = g\theta_t g^{-1} = \theta_{t,\Phi,\Psi}$ for $t \in (-1,1)$ and for the 1-parameter subgroup $\{\theta_t\}_{-1<t<1} = \Theta \subset G_0$ of transvections.  We have

\begin{lemma}[{\cite[{Lemma 3.16}]{MW2025}}]\ \
Fix $h \in H^\infty(\Omega)$.  Suppose $(\Phi,\Psi)$ is a pair of maximal faces of $\Omega$ for which $\rho_{\Phi,\Psi}$ is defined.
Writing $\partial^\flat V_c \cap \Phi =: \{a(c)\}$, assume that the admissible
boundary values $h^\sharp_{a(c),c}$ and $h^\sharp_{c,a(c)}$ both exist for almost all $c \in \Psi$. Then, for {\bf every} point $c \in \Psi$ the admissible boundary values $h^\sharp_{a(c),c}$ and $h^\sharp_{c,a(c)}$ both exist.
Moreover, $\hat h_{\Phi,\Psi} := \lim\limits_{t\to 1} h\circ\xi_t$ exists in the sense of uniform convergence on compact subsets of $\Omega$, given by $\hat h_{\Phi,\Psi} = \rho_{\Phi,\Psi}^*s_{\Phi,\Psi}$, where $s_{\Phi,\Psi} \in H^{\infty}(\Phi)$ is given by $s_{\Phi,\Psi}(a) := h^\sharp_{a,c}$,
for the unique point $c\in\Psi$ satisfying $\partial^\flat V_c \cap \Phi = \{a\}$.
\label{aebdyvalue}
\end{lemma}

Given a holomorphic function $h$ on $\Omega$, a pair $(\Phi,\Psi)$ of maximal faces for which the Cayley projection $\rho_{\Phi,\Psi}$ is defined, and given $a \in \partial^\flat V_c, c\in \Psi$, in principle the admissible boundary value $h_{a,c}^\sharp$ of $h|_{V_c}$ at $a$ depends on $\Psi$. From estimates on intrinsic metrics we have nonetheless the following result.

\begin{proposition}[{\cite[{Proposition 3.17}]{MW2025}}] \
Let $[\Phi] \in \mathfrak M = G_0/N$, and $[\Psi_1], [\Psi_2] \in \mathfrak M - \{[\Phi]\}$ for which Cayley projections $\rho_i := \rho_{\Phi,\Psi_i}$, $i = 1, 2$, are defined.  Suppose $a \in \Phi$ and $c_i \in \Psi_i$ for $i = 1,2$.  Assume that, for $i = 1, 2$, $\lim h(x_k)$ exists for all sequences $\{x_k\}_{0\le k < \infty}$ on $V_{c_i}$ converging admissibly to $a$,
and denote the common limit $($for $i$ fixed$)$ by $h^\sharp_{a,c_i}$.  Then, $h^\sharp_{a,c_1} = h^\sharp_{a,c_2}$.  As a consequence, in the conclusion of Lemma \ref{aebdyvalue} we have $\lim\limits_{t\to 1} \theta_t^*h= \rho_{\Phi,\Psi}^*s_{\Phi}$ for $s_{\Phi} \in H^\infty(\Phi)$ independent of the choice of $\Psi$.
\label{aebdyfcn}
\end{proposition}

\vskip 0.2cm
\subsection{Finding regular pairs for which Cayley limits of \texorpdfstring{$h\in\mathscr F$}{h in script-F} exist}
\label{regular-pairs-find}
\mbox{}

\vskip 0.5cm
\subsubsection{The universal space}
\label{universal-space}
The algebra $\mathscr F = F^*H^\infty(\widetilde{N})$ of holomorphic functions on $\Omega$, as discussed in \S\ref{algF}, is our main object of study. In order to answer Problem \ref{ext-problem} (the Extension Problem) in the affirmative, our aim is to prove that $\id_\Omega \in \mathscr F^n$, i.e., to show that there exist bounded holomorphic functions $\mu_1,\cdots,\mu_n$ on $\widetilde{N}$ such that $F^*\mu_i = z_i$ for $1\leq i\leq n$ in terms of the Harish-Chandra coordinates on $\Omega \Subset \mathbb C^n$. Towards that end we need to be able to apply
 Lemma \ref{aebdyvalue} and Proposition \ref{aebdyfcn}
 to obtain nonconstant bounded holomorphic functions of the form $\hat{h}_{\Phi,\Psi}= \rho_{\Phi,\Psi}^*s_\Phi$, for some functions $s_\Phi\in H^{\infty}(\Phi)$ which depends on fewer variables, viz., variables on the maximal face $\Phi \cong \Omega^{\flat}$, where $\dim_{\mathbb C}\Omega^\flat = q$ and ${\rm rank}(\Omega^\flat) = r-1 < {\rm rank}(\Omega) = r$.  If we denote by $H \subset G_0$ the underlying hyperbolic flow, then these functions $\hat{h}_{\Phi,\Psi}$ have the remarkable property that they are invariant under the action of the 1-parameter group $H$, and they constitute the starting point of our answer to Problem \ref{ext-problem} (the Extension Problem).

\vskip 0.2cm
Fix a nonconstant bounded holomorphic function $\mu$ on $\widetilde {N}$ and consider $h = F^*\mu \in \mathscr F$. We may apply generalized Fatou's theorem on the complex unit ball to study admissible limits of such a function. In order to do so, we have to look for a regular pair of faces
$(\Phi,\Psi)$ such that admissible limits exist almost everywhere for the hyperbolic flow $\{\theta_{t,\Phi,\Psi}\}$ as $t \to 1$ and as $t\to -1$.  We will parametrize all Cayley projections and identify among them those for which Cayley limits exist for $h \in H^\infty(\Omega)$.  For this purpose we introduce the universal space $\mathscr S \subset \Reg(\partial\Omega)\times \Reg(\partial\Omega)$, as follows. For the ensuing discussion we assume that $\Omega$ is irreducible and of rank $\ge 2$.

\vskip 0.2cm
We consider first of all the set of all varieties $V_c \cong \mathbb B^{p+1}$ swept out by minimal disks with vertices at $c \in \Reg(\partial\Omega)$. Denoting by ${\mathfrak D}(A) := \{(a,a): a \in A\}$ the diagonal of $A\times A$ for any set $A$, we introduce
\[
\mathscr{S}:=\{(a,c) \mid c\in \partial^\flat V_a \}  \subset \Reg(\partial\Omega) \times \Reg(\partial\Omega)-{\mathfrak D}(\Reg(\partial\Omega)).
\]
Recall that the moduli space $\mathfrak M$ of maximal faces $\Phi \subset \Reg(\partial\Omega)$ is a compact $G_0$-homogeneous manifold, hence $\mathfrak M = G_0/N$ for some parabolic subgroup $N\subset G_0$.  We have the fibration
\begin{equation}
\epsilon: \Reg(\partial \Omega)\rightarrow \mathfrak M, \quad \epsilon^{-1}(gN)\cong \Omega^\flat.
\label{pi}
\end{equation}
Using the projection of the second factor in $\mathscr{S}$, we get the following fibration:
\begin{equation}
\pi_2: \mathscr{S}\rightarrow \Reg(\partial \Omega), \quad \pi_2^{-1}(c)=\partial^\flat V_c\cong \partial \mathbb{B}^{p+1}-\{u\}, u \in \partial \mathbb B^{p+1},
\label{pi2}
\end{equation}
Note that $a \in \partial^\flat V_c$ if and only if $a$ and $c$ are distinct and they lie on the boundary of a minimal disk $D_\ell=\Omega\cap \ell$, hence $a\in \partial^\flat V_c$ if and only if $c\in \partial^\flat V_a$.

Fix an ordered pair $(a,c) \in \mathscr{S}$. Let $\Phi$ resp.$\,\Psi$ be the unique maximal face containing $a$ resp.\,$c$.
From Proposition \ref{ballext},
we have $\theta^\sharp: V_c\coprod \partial^\flat V_c \overset{\cong}\longrightarrow \overline{\mathbb B^{p+1}} - \{u\}$ for some point $u \in \partial \mathbb B^{p+1}$. We consider the restriction of $h$ to $V_c$, and we write $h^\sharp_{a,c} = \lambda \in \mathbb C$ to mean that
$\big(\theta^{-1}\big)^*h$
has an admissible boundary value equal to $\lambda$ at the point $\theta^{\sharp}(a) \in \partial \mathbb B^{p+1} - \{u\}$.  Given $(a,c) \in \mathscr{S}$, for any point $x \in V_c$, $\rho_{\Phi,\Psi}(x)$ converges to $a$ admissibly, hence $\lim\limits_{t\to 1}h(\theta_{t,\Phi,\Psi}(x)) = h^\sharp_{a,c}$.

We now consider the locus of points on which admissible boundary values do not exist.
\begin{definition} \
Let $M$ be a smooth manifold and $E \subset M$.  We say that $E$ is a measurable subset of $M$ if and only if it is measurable with respect to some smooth $($hence any$)$ smooth volume form $d\mu$ on $M$.  $E \subset M$ is said to be a null set on $M$ if and only if $E \subset M$ is measurable and ${\rm Volume}(E,dV_h) = 0$ for some $($hence any$)$ Riemannian metric $h$ on $M$. A statement is said to hold true almost everywhere on $M$ if and only if it holds true for all points $x \in M$ lying outside some null set on $M$.
\end{definition}

We have the following lemma concerning null sets on the total space of a double fibration which follows readily from Fubini's theorem.

\vskip 0.2cm

Define now $E \subset \mathscr{S}$ to be the subset consisting of points $(a,c)$ such that $h^\sharp_{a,c}$ does not exist. By generalized Fatou's theorem for the complex unit ball (Theorem \ref{fatou} here),
we know that $E\cap \partial^\flat V_c$ is a null set, i.e., it is of zero measure with respect to any smooth volume form on $\partial^\flat V_c$.
It is elementary to verify that $E$ is a measurable set (cf. \cite[\S 4.3]{MW2025}), hence one can apply Fubini's theorem to
$\pi_2: \mathscr{S} \to \Reg(\partial\Omega)$ to deduce that $E \subset \mathscr{S}$ is a null set.
Consider $\mathscr{S}$ as the total space of a double fibration
$\pi_2: \mathscr{S}\rightarrow \Reg(\partial \Omega)$
and $\sigma: \mathscr{S}\rightarrow \mathfrak{M},\, \sigma(a,c):=[\Phi]$,
where $[\Phi]\in \mathfrak{M}$ is the unique maximal face $\Phi$ containing $a$, i.e. $\sigma=\epsilon\circ \pi_1$ for the natural projection
$\epsilon: \Reg(\partial \Omega)\rightarrow \mathfrak{M}$.
Applying the latter statement to the fibration $\sigma: \mathscr{S} \to \mathfrak M$, we conclude that $E\cap\sigma^{-1}([\Phi])$ is a null set
on $\mathfrak{A}:=\{(a,c(a)): a\in \Phi\}\subset \Phi\times \Psi$ for almost every $[\Phi] \in \mathfrak M$.
Note that $(a,c(a))\in \mathfrak{A}$ if and only if $h^\sharp_{a,c(a)}$ exists. This means that $h^\sharp_{a,c(a)}$ exists except for a null set on $\Phi$.
The general principle of this argument is summarized in the following lemma, which is a consequence of Fubini's theorem.

\begin{lemma} \
Let $\mathscr P$ be a smooth manifold which is the total space of a double locally trivial smooth fibration $\eta_i: \mathscr P \to B_i$, $i = 1,2$.  For $b_i \in B_i$ write $F^i_{b_i} := \eta_i^{-1}(b_i)$.  Suppose $E \subset \mathscr P$ is a measurable subset. Assume that for almost every base point $b_1 \in B_1$, $E\cap F^1_{b_1}$ is a null set. Then, for almost
every base point $b_2 \in B_2$, $E\cap F^2_{b_2} \subset F^2_{b_2}$ is a null set.
\label{nullfib}
\end{lemma}

\vskip 0.2cm
\subsubsection{Regular pairs of maximal faces}
\label{regular-pairs}
To study Cayley projections we consider pairs $(\Phi,\Psi)$ which are linked by a hyperbolic flow. (Here and henceforth to simplify notation we write $(\Phi,\Psi)$ in place of $([\Phi],[\Psi])$ as an element in $\mathfrak M\times\mathfrak M$). Let $(\Sigma,\Sigma')$ be the standard reference pair of maximal faces of Definition \ref{stdpair}, regarded as an element of $\mathfrak M\times\mathfrak M$. We say that a pair $(\Phi,\Psi) \in \mathfrak M \times \mathfrak M$ is a regular pair whenever there exists $c_0 \in \Psi$ such that $\partial^\flat V_{c_0} \cap \Phi \neq \emptyset$, equivalently, there exists $g\in G_0$ such that $g\Sigma = \Phi$ and $g\Sigma' = \Psi$, hence $(\Phi,\Psi)$ is a regular pair if and only if $(\Psi,\Phi)$ is a regular pair (since there exists an automorphism of $\Omega$ switching the two faces $\Sigma$ and $\Sigma'$).  Moreover, $(\Phi,\Psi)$ is a regular pair if and only if for every point $c\in \Psi$, $\partial^\flat V_c \cap \Phi = \{a(c)\}$ for a unique point $a(c) \in \Phi$.  Denote by $\mathscr C \subset \mathfrak M\times \mathfrak M$ the set of regular pairs $(\Phi,\Psi)$. Here $V_c$ decomposes into minimal disks $\Delta_b$, $b \neq c$, such that $c \in \partial\Delta_b$. Writing $\mathscr E := (\mathfrak M\times \mathfrak M) - \mathscr C$, recall the following description of the space $\mathscr C$ of all regular pairs.

\begin{lemma}[{Mok-Wong \cite[{Lemma 4.4}]{MW2025}}] \
$G_0 := \Aut_0(\Omega)$ acts transitively on $\mathscr C$ by \, $\alpha(g,(\Phi,\Psi)) = (g\Phi,g\Psi)$. Moreover, a pair $(\Phi,\Psi)$ of maximal faces on $\partial\Omega$ is a regular pair if and only if $\overline{\Phi} \cap \overline{\Psi} = \emptyset$. Hence, $\mathscr E \subsetneq \mathfrak M\times \mathfrak M$ is a proper $($real$)$ algebraic subset. In particular, $\mathscr C$ is an open $($and dense$)$ subset in $\mathfrak M \times \mathfrak M$ of full measure.
\label{regular-pair}
\end{lemma}

For a given holomorphic function $h\in H^\infty(\Omega)$ we define $\mathscr C_{\rm adm}(h) \subset \mathscr C$ to consist of all regular pairs $(\Phi,\Psi)$ such that the admissible boundary values $h^\sharp_{a,c}$ exist for almost every point $a\in\Phi$, where $a = a(c)\in\Phi$ is the opposite point of $c\in\Psi$ defined by $a(c) = \partial^\flat V_c\cap\Phi$. We define $\mathscr C_{\rm adm}'(h) := \left\{(\Phi,\Psi) \in \mathscr{C}: (\Phi,\Psi) \in \mathscr C_{\rm adm}(h) \ {\rm and} \ (\Psi,\Phi)\in \mathscr C_{\rm adm}(h)\right\}$. In other words, $\mathscr C_{\rm adm}'(h) = \mathscr{C}_{\rm adm}(h) \cap\mathfrak S(\mathscr C_{\rm adm}(h))$, where $\mathfrak S$ is the switching map on
$\mathfrak M\times\mathfrak M$ defined by $\mathfrak S(\Phi,\Psi) = (\Psi,\Phi)$.

From Lemma \ref{regular-pair} and Lemma \ref{nullfib} one obtains the following.

\begin{proposition}
\label{bdry-functions}
For almost all $g\in G_0$ one has $(g\Sigma,g\Sigma') \in \mathscr C_{\rm adm}'(h)$.
Hence, for almost every regular pair $(\Phi,\Psi) \in \mathscr C$ of maximal faces,
the Cayley limits $\hat{h}_{\Phi,\Psi} = \rho_{\Phi,\Psi}^*s_{\Phi}$ and
$\hat{h}_{\Psi,\Phi} = \rho_{\Psi,\Phi}^*s_{\Psi}$ exist,
where $s_{\Phi}\in H^\infty(\Phi)$ and $s_{\Psi} \in H^\infty(\Psi)$.
\end{proposition}
Proposition \ref{bdry-functions} is not yet good enough for us to solve Problem \ref{ext-problem} (the Extension Problem). For one thing we have to make sure that we can choose the regular pair $(\Phi,\Psi)\in \mathscr C_{\rm adm}'(h)$ such that both $s_\Phi$ and $s_\Psi$ are nonconstant.  It is necessary to have both Cayley limits since in our application of Moore's ergodicity theorem, we will only be able to conclude that {\it either\/} $\rho_{\Phi,\Psi}^*s_{\Phi}$ {\it or\/} $\hat{h}_{\Psi,\Phi}$ belongs to $\mathscr{F}$ (cf.\,Proposition \ref{pulled-back} in \S\ref{ergodic}).

\vskip 0.2cm
\section{Moore's ergodicity theorem intermingles with Cayley projections}
\label{ergodic}

\vskip 0.2cm
\subsection{Moore's ergodicity theorem of semisimple real Lie groups and the density lemma}
\label{moore-density}
An important component of the proof of Theorem \ref{isom} is ergodic theory, specifically Moore's ergodicity theorem on semisimple real Lie groups and its consequences. We refer the reader to Zimmer \cite{Zim1984} for a systematic discussion of the topic. First of all, we have
\begin{theorem}[{Moore's ergodicity theorem, cf. \cite[{Theorem 2.2.6}]{Zim1984}}] \
Let $G = \prod G_i$ be a connected semisimple real Lie group, where each $G_i$ is a connected non-compact simple Lie group with finite center. Let $\Gamma\subset G$ be an irreducible lattice and $H\subset G$ be a noncompact closed subgroup.  Then, $H$ acts ergodically on $\Gamma\backslash G$.
\label{moore}
\end{theorem}

For our proof of Theorem \ref{extension}, we consider bounded symmetric domains $\Omega =\Omega_1\times\cdots\times\Omega_s$ of rank $\ge 2$ in their decomposition into Cartesian products of irreducible factors and torsion-free irreducible lattices $\Gamma \subset \Aut(\Omega)$.
We apply Moore's ergodicity theorem to the situation where $G = \Aut_0(\Omega)$, $G_i = \Aut_0(\Omega_i)$ for $1\le i \le s$, and $H$ is a hyperbolic flow, i.e., a one-parameter family of transvections, in one of the irreducible factors.

\vskip 0.2cm
Every semisimple Lie group $G$ is unimodular.  In what follows we equip a semisimple Lie group $G$ with a bi-invariant measure so that for a lattice $\Gamma \subset G$, $\Gamma\backslash G$ is equipped with an induced measure $d\mu_{\Gamma\backslash G}$ (i.e., volume form).  The following is a consequence of Moore's ergodicity theorem on semisimple real Lie groups.

\begin{lemma}[{Density lemma, cf. \cite[{Proposition 2.1.7}]{Zim1984}}]
Let $G$ be a semisimple real Lie group, $\Gamma \subset G$ be a lattice, and $H \subset G$ a noncompact Lie subgroup.  Then, there exists a subset $E \subset \Gamma\backslash G$ of zero measure $($with respect to $d\mu_{\Gamma\backslash G})$ such that the $H$-orbit $xH \subset \Gamma\backslash G$ is dense in $\Gamma\backslash G$ whenever $x \in (\Gamma\backslash G) - E.$
\end{lemma}
The density lemma as stated was already applied in the study of Finsler metric rigidity in Theorem \ref{finsler}, where the existence of a single dense leaf was already sufficient to deduce the desired constancy from the Gauss--Bonnet integral formula on $\mathscr C(X_\Gamma)$ when $\Omega$ is irreducible.  The proof for the case where $\Omega$ is reducible and $\Gamma \subset \Aut(\Omega)$ is irreducible is similar.

\vskip 0.2cm
We will make use of Moore's ergodicity theorem to a greater extent as captured by the discussion in the next subsection \S\ref{produce-invariant}.

\vskip 0.2cm
\subsection{Producing functions \texorpdfstring{$h\in\mathscr F$}{h in script-F} invariant under Cayley projections}
\label{produce-invariant}
\mbox{}

\vskip 0.5cm
\subsubsection{Starting with the density lemma for the hyperbolic flow}
\label{density-lemma-start}
We start with recalling two results from \cite{MW2025}.
\begin{lemma}[{cf. \cite[{Lemma 6.3}]{MW2025}}] \
Suppose $g \in G_0$ and the coset $\Gamma gH$ is dense in $\Gamma\backslash G_0$. Then, there exists a discrete sequence $\{\gamma_k\}\subset \Gamma$ such that $\gamma_k=\left(g\theta_{t_k}g^{-1}\right)\delta_k$ for some $\delta_k \in \Aut_0(\Omega)$ and $t_k \in (-1,1)$ satisfying $\delta_k\to \id_{\Omega}$ and either $t_k\to 1$ or $t_k \to -1$.
\label{embdensity}
\end{lemma}

\begin{proposition}[{cf. \cite[{Proposition 6.8}]{MW2025}}] \
Fix $h \in \mathscr F$.  Suppose both $(\Phi,\Psi)$ and $(\Psi,\Phi)$ belong to $\mathscr C_{\rm adm}(h)$, i.e., admissible limits of $h$ exist on both $\Phi, \Psi\in\mathfrak{M}$ under the 1-parameter hyperbolic flow between them. Then, either $\rho_{\Phi,\Psi}^*s_\Phi$ or $\rho_{\Psi,\Phi}^*s_\Psi$ belongs to $\mathscr F$.
\label{pulled-back}
\end{proposition}

\begin{proof}
Since $G_0$ acts transitively on $\mathscr C$ there exists $g \in G_0$ such that $g\Sigma = \Phi$ and $g\Sigma' = \Psi$.  By Lemma \ref{embdensity}, there exists $t_k \to 1$ or $t_k \to -1$ such that, writing $\xi_t = g\theta_t g^{-1}$, there exists $\gamma_k \in \Gamma$ such that $\gamma_k = \xi_{t_k}\delta_k$ for $k \in \mathbb N$ such that $\delta_k \to \id_{\Omega}$.  We assume $t_k \to 1$ and proceed to prove that  $\rho_{\Phi,\Psi}^*s_\Phi \in \mathscr F$.  The case where $t_k \to -1$ will lead to $\rho_{\Psi,\Phi}^*s_\Psi \in \mathscr F$ by exactly the same argument. We have
\begin{eqnarray*}
\big|\xi_{t_k}^*h(z)-\gamma_k^*h(z)\big|
&=&\big|(h\circ\xi_{t_k})(z)-(h\circ\gamma_k)(z)\big|\\
&=&\big|(h\circ\xi_{t_k})(z)-(h\circ\xi_{t_k})(\delta_k(z))\big|\\
&=&\left|\int_{z}^{\delta_k(z)}(h\circ\xi_{t_k})'(\xi)d\xi\right| \\
&\leq& C \left\|h\circ\xi_{t_k}\right\|_{H^{\infty}(\Omega)}\big\|\delta_k(z)-z\big\| \\
&\leq& C \left\|h\right\|_{H^{\infty}(\Omega)}\big\|\delta_k(z)-z\big\|
\rightarrow 0,
\end{eqnarray*}
for some constant $C > 0$, since $\delta_k\rightarrow \id_{\Omega}$ as $k\rightarrow \infty$. Here, the second last inequality follows from Cauchy's estimate for first derivatives of holomorphic functions.
By Lemma \ref{invariant-alg}(a), $\gamma_k^*h\in\mathscr F$. By the $h$-admissibility of $(\Phi,\Psi)$ and Lemma \ref{aebdyvalue}, $\xi_{t_k}^*h\to\rho_{\Phi,\Psi}^*s_\Phi$ uniformly on compact subsets of $\Omega$. The estimate above therefore yields $\gamma_k^*h\to\rho_{\Phi,\Psi}^*s_\Phi$ uniformly on compact subsets. Writing $\gamma_k^*h=F^*u_k$ with $\{u_k\}$ uniformly bounded, Lemma \ref{invariant-alg}(b) gives $\rho_{\Phi,\Psi}^*s_\Phi\in\mathscr F$. If $t_k\to -1$, the same argument yields $\rho_{\Psi,\Phi}^*s_\Psi\in\mathscr F$.
\end{proof}

We proceed with some oversimplification for the ensuing discussion aiming towards producing elements in $\mathscr F$ invariant under a hyperbolic flow.  Recall the standard reference pair $(\Sigma,\Sigma')$ of maximal faces of Definition \ref{stdpair}. Let $h\in\mathscr F$ be a nonconstant bounded holomorphic function.  By an application of Lemma \ref{nullfib}, for almost every group element $g\in G_0$, $(g\Sigma,g\Sigma') \in \mathscr C_{\rm adm}'(h)$ and $\Gamma gH$ is dense in $G_0$. The hyperbolic flow $H$ is replaced by $H^g = gHg^{-1}$.  Changing coordinates on $\Omega$ by an automorphism the function $g^*h$ has admissible boundary limits on both $\Sigma$ and $\Sigma'$. To proceed we may therefore without loss of generality assume that $(\Sigma,\Sigma')\in \mathscr C_{\rm adm}'(g^*h)$.  In this new choice of privileged Harish-Chandra coordinates, the discrete subgroup $\Gamma \subset \Aut(\Omega)$ is changed to $\Gamma' = g^{-1}\Gamma g$, while $\mathscr F$ is changed to $g^*\mathscr F = g^*F^*H^\infty(\widetilde{N}) = (F\circ g)^*H^\infty(\widetilde{N})$. Thus, we are considering
$F\circ g: \Omega\to\widetilde{N}$ which is $\Gamma'$-equivariant for $\Gamma' = g^{-1}\Gamma g$, where
$(F\circ g)\big((g^{-1}\gamma g)(g^{-1}z)\big) = F(\gamma z) =\Phi(\gamma)F(z)$.
Thus, writing $\Phi'(g^{-1}\gamma g) := \Phi(\gamma)$, we have
$(F\circ g)(\gamma'w) = \Phi'(\gamma')\bigl((F\circ g)(w)\bigr)$ for $\gamma' = g^{-1}\gamma g$ and for $w \in \Omega$.
Note that
the density of $\Gamma gH$ in $G_0$ gives also $\overline{\Gamma'H}= \overline{g^{-1}\Gamma gH} = g^{-1}(\overline{\Gamma gH}) = G_0$, i.e., $\Gamma'H$ is dense in $G_0$.

\vskip 0.2cm
We now think of $h$ just as a generic symbol for an element of $\mathscr F$ and $\Gamma$ as a generic symbol for a torsion-free discrete lattice in $\Aut(\Omega)$, etc. Thus, we relabel $g^*h$ as $h$, $\Gamma' = g^{-1}\Gamma g$ as $\Gamma$, $\Phi'$ as $\Phi$, $F\circ g$ as $F$ and $g^*\mathscr F = (F\circ g)^*H^\infty(\widetilde{N})$ as $\mathscr F$.  Then, we have a $\Phi$-equivariant holomorphic map $F: \Omega \to \widetilde{N}$ and a nonconstant holomorphic function $h\in \mathscr F$ such that $(\Sigma,\Sigma')\in \mathscr C_{\rm adm}'(h)$.
\vskip 0.2cm
By Proposition \ref{cayproj}, we have $\hat{h}_{\Sigma,\Sigma'}:=\rho_{\Sigma,\Sigma'}^*s_{\Sigma}$ and $\hat{h}_{\Sigma',\Sigma} :=\rho_{\Sigma',\Sigma}^*s_{\Sigma'}$.  For $w\in\Omega^\flat$ write $s_\Sigma(w) = h_{b(w),b'(w)}^\sharp$ and $s_{\Sigma'}(w) = h_{b'(w),b(w)}^\sharp$, where $b(w) = (1;0;w)$ and $b'(w) = (-1;0;w)$ are opposite points with respect to the zero section $\{(0;0)\}\times\Omega^\flat$,
to signify that these are admissible boundary values at $b(w) \in \partial^\flat V_{b'(w)}$ resp.\,$b'(w) \in \partial^\flat V_{b(w)}$ of $h|_{V_{b'(w)}}$ resp.\,$h|_{V_{b(w)}}$, recalling that $V_{b'(w)}, V_{b(w)} \cong \mathbb B^{p+1}$ with biholomorphisms extending to diffeomorphisms of $V_{b'(w)}\coprod\partial^\flat V_{b'(w)}$ resp.\,$V_{b(w)}\coprod\partial^\flat V_{b(w)}$ with $\overline{\mathbb B^{p+1}}-\{u\}$ as smooth manifolds with boundary, where $u \in\partial\mathbb B^{p+1}$.

\vskip 0.2cm
In what follows, we write $\hat{h}$ for one of the two functions $\hat{h}_{\Sigma,\Sigma'}=\rho_{\Sigma,\Sigma'}^*s_{\Sigma}$ and $\hat{h}_{\Sigma',\Sigma} =\rho_{\Sigma',\Sigma}^*s_{\Sigma'}$, more precisely the one that satisfies $\hat{h}\in\mathscr{F}$ by Proposition \ref{pulled-back}.

\begin{proposition}[{cf. \cite[{Proposition 6.9}]{MW2025}}]
The bounded holomorphic functions $s_{\Sigma}\in H^\infty(\Sigma)$ and $s_{\Sigma'}\in H^\infty(\Sigma')$ may be chosen to be both nonconstant holomorphic functions.  Furthermore, for almost every group element $g\in G_0$ $($and after replacing $\Gamma$ by $g^{-1}\Gamma g)$ we have $\partial \hat{h}(0)\neq 0$.
\label{nontrivial-pullback}
\end{proposition}

\vskip 0.2cm
Clearly, $\hat{h}_{\Sigma,\Sigma'}$ and $\hat{h}_{\Sigma',\Sigma}$ are $H$-invariant.  In the following result, which follows from the proof of Mok-Wong \cite[Proposition 6.10]{MW2025}, we use the notation $h$ instead for an $H$-invariant function belonging to $\mathscr F$,
such that $h=\rho_{\Sigma,\Sigma'}^*s_{\Sigma}$ for some $s_{\Sigma}\in H^\infty(\Sigma)$.

\begin{proposition}
Let $h \in\mathscr F$ be an $H$-invariant bounded holomorphic function. Suppose $\Gamma H$ is dense in $G_0$. Then, for every group element $g\in G_0$, we have $g^*h\in\mathscr F$.
Moreover, $g^*h=\rho_{g\Sigma, g\Sigma'}^*s_{g\Sigma}$ for some $s_{g\Sigma}\in H^\infty(g\Sigma)$.
\label{group-invariance}
 \end{proposition}

\vskip 0.2cm
Proposition \ref{nontrivial-pullback}, together with Proposition \ref{group-invariance}, will pave the way for an affirmative answer to Problem \ref{ext-problem} (the Extension Problem). We will start with the proof of Proposition \ref{group-invariance}.

\vskip 0.2cm
\subsubsection{From \texorpdfstring{$H$}{H}-invariance to generating a \texorpdfstring{$G_0$}{G0}-invariant subspace of \texorpdfstring{$\mathscr F$}{script-F}}
\label{H-to-G0}
We start with a proof of Proposition \ref{group-invariance}.

\begin{proof}
By assumption $h\in\mathscr F$ is $H$-invariant, i.e., $h(\eta z) = h(z)$ for any group element $\eta\in H$. By Lemma \ref{invariant-alg}, for each element $\gamma\in\Gamma$, we have $\gamma^*h \in \mathscr F$. By assumption, $\Gamma H$ hence $(\Gamma H)^{-1} = H\Gamma$ is dense in $G_0$. Given any $g\in G_0$, there exist sequences $\{\eta_k\}$ in $H$ and $\{\gamma_k\}$ in $\Gamma$ such that $g_k :=\eta_k\gamma_k$ converges to $g$.
In other words, $g = \eta_k\gamma_k\delta_k$ for a sequence of elements $\{\delta_k\}$ in $G_0$
converging to the identity element $\id\in G_0$.
By assumption $h$ is $H$-invariant, hence $\eta_k^*h = h$.
Moreover, by Lemma \ref{invariant-alg}, $h_k := g_k^*h = \gamma_k^*(\eta_k^*h) = \gamma_k^*h \in \mathscr F$. Since $g_k \to g$, we have $h_k = g_k^*h \to g^*h$.  On the other hand, since $\delta_k\to \id$ we have by the Cauchy estimate for first derivatives $|h_k(\delta_kz) - h_k(z)|\to 0$ uniformly on compact sets, hence $g^*h = \lim\limits_{k\to\infty} g_k^*h = \lim\limits_{k\to\infty} h_k$ uniformly on compact sets. Since $h_k\in\mathscr F$, by Lemma \ref{invariant-alg} it follows $g^*h \in \mathscr F$.
The last statement holds true because $(g\Sigma,g\Sigma')\in \mathscr{C}_{\rm adm}(h)$ for almost all group elements $g\in G_0$.
\end{proof}

\vskip 0.2cm
\subsubsection{Existence of nonconstant \texorpdfstring{$H$}{H}-invariant functions in \texorpdfstring{$\mathscr F$}{script-F} satisfying a nondegeneracy assumption}
\label{nondeg-H-invariant}
In the statement of Proposition \ref{nontrivial-pullback} it is asserted that there exists a pair of $H$-invariant bounded holomorphic functions $h_{\Sigma}, h_{\Sigma'} \in \mathscr F$ such that $dh_\Sigma(0), dh_{\Sigma'}(0) \neq 0$.  It is necessary to have information about both functions, since by Proposition \ref{pulled-back} we only know that either $\rho_{\Sigma,\Sigma'}^*s_{\Sigma}$ or $\rho_{\Sigma',\Sigma}^*s_{\Sigma'}$ belongs to $\mathscr F$. By an application of Lemma \ref{nullfib}, and Moore's ergodicity theorem, for almost all group elements $g\in G_0$, either $\rho_{g\Sigma,g\Sigma'}^*s_{g\Sigma}$ or $\rho_{g\Sigma',g\Sigma}^*s_{g\Sigma'}$ exists to give a function invariant under the hyperbolic flow $H^g = gHg^{-1}$ from $g\Sigma'$ to $g\Sigma$ or vice versa.  It is tempting to believe that the function $h$ can be recovered from some integral formula involving all the boundary functions $s_{g\Sigma}$, but the space of these boundary functions does not have a simple structure, and the fact that only one of the two opposite faces gives rise to an $H^g$-invariant function in $\mathscr F$ complicates the matter.  Our solution to the problem in fact goes along this line, but we limit ourselves to a geometric situation where the integral formula is the simplest possible, viz., the Cauchy integral formula for a bounded holomorphic function on the unit disk $\Delta$ in terms of admissible limits on the boundary circle $\mathbb S^1\cong \partial\Delta$ expressing the function as a Cauchy integral of nontangential limits on the boundary circle. More precisely, we consider a family of such integral formulas parametrized by $w\in \Omega^\flat$. The argument is given in \cite[Proposition 6.10]{MW2025}.

\vskip 0.2cm
The geometric object we consider is the image of $\Delta\times\Omega^\flat$ in $\Omega$ by a holomorphic totally geodesic embedding defined by root systems.
We call this a special product subspace. As an example, in the case of a type-I domain $\Omega = D^I_{p,q}$ of rank  ${\rm min}(p,q) \ge 2$,
we have $\Omega^\flat \cong D^{I}_{p-1,q-1}$,
and a prototype of the special product subspace is the linear section of $\Omega$
in Harish-Chandra coordinates defined by $z_{12} = \cdots = z_{1q}$ and $z_{21} = \cdots = z_{2p}$ in matrix notation.
In general in privileged Harish-Chandra coordinates we consider $\Delta\times\{0\}\times\Omega^\flat$, which is a prototype of a special product subspace.  Let $f$ be a holomorphic function on the special product subspace $P_0 = \Delta\times\{0\}\times \Omega^\flat$ in privileged Harish-Chandra coordinates which is
furthermore bounded on $\Delta\times\{0\}\times V$ for any relatively compact open subset $V \Subset\Omega^\flat$.  Then, nontangential limits $f^\sharp(\zeta;0;w)$ exist almost everywhere on $\mathbb S^1\times\{0\}\times V$ and we have the Cauchy integral formula
$$
f(z_1;0;w) = \frac{1}{2\pi i}\int_{\partial\Delta}\frac{f^\sharp(\zeta;0;w)}{\zeta-z_1}d\zeta\, ,
$$
We consider now the special case where we have admissible boundary values $h^\sharp(e^{i\theta};0,w) = h^\sharp_{b_\theta(w),b_\theta'(w)}$
where $b_\theta(w) = (e^{i\theta};0;w)$ and $b_\theta'(w) = (-e^{i\theta};0;w)$. Fix an index $j$, $1\le j\le q = \dim_{\mathbb C}(\Omega^\flat)$. We consider the above parametrized Cauchy integral formula on the unit disk to the holomorphic function $f = u_j$ on $P_0$ defined by
$$
u_j(z_1;0;w) := \partial_jh(z_1;0;w)\cdot \partial_jh(-z_1;0;w)
$$
It follows from Cauchy estimates that $\partial_jh(z_1;0;w)$, and hence $u_j(z_1;0;w)$ is uniformly bounded on $\Delta\times\{0\}\times V$ for any relatively compact open subset $V \Subset \Omega^\flat$.

\vskip 0.2cm
For $z_0\in\Delta$, we write $\Omega^\flat(z_0):=\{(z_0;0)\}\times\Omega^\flat$.
Recall that $P\subset \Omega$ is a special product subspace of $\Omega$ whenever $P$ is the image of a totally geodesic holomorphic embedding of $\Delta\times \Omega^\flat$ into $\Omega$.
We also know that $G_0$ acts transitively on the set of special product subspaces.
The special product subspace $P_0=\Delta\times\{0\}\times \Omega^\flat$ in privileged Harish-Chandra coordinates is a reference special product subspace.
We have a projection $P_0\rightarrow \Omega^\flat$ defined by $(z_0;z';z'')\mapsto z''$.
Moreover, $\Omega^\flat(0)=\{(0;0)\}\times \Omega^\flat$ is the image of a holomorphic section $z''\mapsto (0;0;z'')$ of $P_0\rightarrow \Omega^\flat$.
We will consider the pair $(P_0,\Omega^\flat(0))$ and the parametrized Cauchy integral formula
which expresses every $f\in \mathscr{C}(\overline{P_0})\cap \mathcal{O}(P_0)$.
(Here $\mathscr{C}(\overline{P_0})$ stands for the space of complex-valued continuous functions on $\overline{P_0}$ and $\mathcal{O}(P_0)$ stands for the space of holomorphic functions on $P_0$.)
In general, for $g\in G_0$ and $P:=gP_0$, we consider the pair $(P,\Omega')$ where $\Omega'=g\Omega^\flat(0)$.
Our interest is to consider the restriction of $h\in H^\infty(\Omega)$ to $P$ and the parametrized Cauchy integral formula reproducing $h|_{\Omega'}$ from boundary values on $g(\partial \Delta\times\{0\}\times \Omega^\flat)$ under the assumption of the existence of admissible boundary values in some precise form.

\vskip 0.2cm
We will say that $(P_0,\Omega^\flat(0))$ is $h$-admissible if and only if both $h^\sharp_{b_\theta(w),b'_\theta(w)}$
and $h^\sharp_{b'_\theta(w),b_\theta(w)}$ exist for almost all $w\in\Omega^\flat$ and almost all $\theta\in\mathbb{R}$.

\vskip 0.2cm
For a special product subspace $P:=gP_0$ and $\Omega'=g\Omega^\flat(0)$, we say that $(P, \Omega')$ is $h$-admissible if and only if $(P_0,\Omega^\flat(0))$ is $g^*h$-admissible.

\vskip 0.2cm
By an application of Lemma \ref{nullfib}, given any bounded holomorphic function $h$ on $\Omega$,
$(gP_0, g\Omega^\flat(0))$ is $h$-admissible for almost all group elements $g\in G_0$.
Thus replacing $h$ by $g^*h$ for almost every group element $g\in G_0$ if necessary, we may assume that we are dealing with the situation where $(P_0, \Omega^\flat(0))$ is $h$-admissible.
In this situation we have the following lemma, which is $(\dagger\dagger)$ in the proof of
 \cite[Proposition 6.10]{MW2025}.

\begin{lemma}
Replacing $h$ by $g^*h$ for almost every $g\in G_0$ if necessary, there exists a subset $A\subset [0,2\pi)$ of positive Lebesgue measure such that for any $\theta\in A$, both $h^\sharp_{b_\theta(w),b_\theta'(w)}$ and $h^\sharp_{b_\theta'(w),b_\theta(w)}$ exist as bounded holomorphic functions on $\Omega^\flat$ and both functions are nonconstant on $\Omega^\flat$.
\label{positive-measure}
\end{lemma}

\begin{proof}
Suppose otherwise. Then, for almost all $\theta \in \mathbb R$, either $h^{\sharp}_{b_\theta(w),b'_\theta(w)}$ or $h^{\sharp}_{b'_\theta(w),b_\theta(w)}$ is constant as a function in $w \in \Omega^\flat$.
Then, $\partial_j h^{\sharp}_{b_\theta(w),b'_\theta(w)} \equiv 0$ for $1 \le j \le q$, or $\partial_j h^{\sharp}_{b'_\theta(w),b_\theta(w)} \equiv 0$  for $1 \le j \le q$, as functions in $w \in \Omega^\flat$, so that for almost all $\theta \in \mathbb R$, $u_j^\sharp(e^{i\theta};0;w) = 0$ as a function in $w \in \Omega^\flat$.
Then the above parametrized Cauchy integral formula implies that $u_j(z_1;0;w) = 0$ for $(z_1;w) \in \Delta\times\Omega^\flat$.
Hence, for $1 \le j \le q$ we have $\partial_j h(z_1;0,w) = 0$ for at least a set of positive measure on the disk $\Delta$ and for all $w\in \Omega^\flat$, so that the equality holds for all $(z_1,w) \in \Delta\times \Omega^\flat$.
Thus, $h(z_1;0,w) = v(z)$ is a bounded holomorphic function on $\Delta$.
It follows that $h$ is constant on each $\Omega^\flat(z_0)$ for $z_0\in \Delta$.
In particular, $h$ is constant on $\Omega^\flat(0)$.
Using again Lemma \ref{nullfib} the same applies to $g\Omega^\flat(0)$ for almost all $g\in G_0$ hence for all $g\in G_0$.  We have proven that the failure of the lemma implies constancy of $h$ on $\Omega^\flat(0)$ and in particular on minimal disks on $\Omega^\flat(0)$ since any two points on $\Omega$ can be linked by a finite number of minimal disks.
This forces $h$ to be a constant function, a plain contradiction.
\end{proof}

\vskip 0.2cm
Lemma \ref{positive-measure} is the main technical tool to give a proof of Proposition \ref{nontrivial-pullback} and we refer the reader to \cite{MW2025} for details.

\begin{remark} \
{\rm
In place of assuring that $\partial \hat{h}(0)\neq 0$ as is stated in Proposition \ref{nontrivial-pullback},
it is sufficient to know that both $s_\Sigma$ and $s_{\Sigma'}$ are nonconstant functions. Given this, we have produced some nonconstant function $h\in\mathscr{F}$. By Proposition \ref{group-invariance}, $g^*h\in \mathscr{F}$ for every group element $g\in G_0$. Since $h$ is nonconstant, for almost every $g\in G_0$, $g^*h\in\mathscr{F}$ satisfies $\partial(g^*h)(0)\neq 0$.
We can then replace $h$ by $g^*h$ for any such $g\in G_0$ to continue with the argument.
}
\end{remark}

\vskip 0.2cm
\section{The proof of the Isomorphism Theorem}
\label{isomproof}
\vskip 0.2cm
\subsection{The Extension Theorem}
\label{extension-thm-sec}
\mbox{}

\vskip 0.5cm
\subsubsection{A positive answer to Problem \ref{ext-problem}}
\label{extension-answer}
Let $\Omega \Subset \mathbb C^n$ be an irreducible bounded symmetric domain of rank $\ge 2$ and $\Gamma \subset \Aut(\Omega)$ be an irreducible torsion-free lattice and write $X_\Gamma = \Omega/\Gamma$.  Let $N$ be a complex manifold and denote by $\widetilde{N}$ its universal covering space.  Recall that, the Embedding Theorem (Theorem \ref{cara-embed}) was established as an application of Finsler metric rigidity (Theorem \ref{finsler}) and the study of extremal bounded holomorphic functions. Let $f:X_{\Gamma} \to N$ be a holomorphic map, and $F:\Omega \to \widetilde{N}$ be its lifting to universal covering spaces. Writing $\Omega = \Omega_1\times\cdots\times\Omega_s$ the decomposition of $\Omega$ into irreducible factors, we assumed in Theorem \ref{cara-embed} that for $1\le k\le s$, there exists $h_k\in H^\infty(\widetilde{N})$ such that $F^*h_k$ is nonconstant on some $k$-th factor domain $S_k \subset \Omega$.  We proved in Theorem \ref{cara-embed} that the map $F:\Omega \to \widetilde{N}$ is a holomorphic embedding.  Thus, $F:\Omega\to\widetilde{N}$ is a holomorphic immersion and it separates points.  Note that we do not know {\it a priori\/} that $F(\Omega) \subset \widetilde{N}$ is closed. Theorem \ref{cara-embed} prompted Problem \ref{ext-problem} (the Extension Problem), and we formulate here an affirmative answer to the question, as follows.

\begin{theorem} {\rm (the Extension Theorem)} \
Under the hypothesis as in the above paragraph, there exists a bounded holomorphic map $R: \widetilde{N} \to \mathbb C^n$ such that $R\circ F = \id_{\Omega}$.
\label{extension}
\end{theorem}
The existence of $R: \widetilde{N} \to \mathbb C^n$ satisfying $R\circ F = \id_{\Omega}$ implies readily that $F: \Omega\to\widetilde{N}$ is an embedding.  The proof of Theorem \ref{extension} in Mok-Wong \cite{MW2025} is independent of Finsler metric rigidity, and as such it gives a new proof of the Embedding Theorem. Nonetheless, the first author would not have formulated Problem \ref{ext-problem} (the Extension Problem) without having the first proof of the Embedding Theorem in Mok \cite{Mok2004} using Finsler metric rigidity.

\vskip 0.2cm
\subsubsection{Solution to Problem \ref{ext-problem} via harmonic analysis and ergodic theory}
\label{extension-proof}
In view of its importance, we reproduce here a condensed version of the end of proof of Theorem \ref{extension} in the case where $\Omega$ is irreducible (and of rank $\ge 2$).

\begin{proof}
For the proof of Theorem \ref{extension} in the case where $\Omega$ is irreducible, we may assume that $\Gamma H$ is dense in $G_0$, $(P_0,\Omega^\flat(0))$ is $h$-admissible, $(\Sigma,\Sigma')$ belongs to $\mathscr C'_{\rm adm}(h)$, $\hat{h}(z) = \lim\limits_{t \to 1}h(\theta_{t,\Sigma,\Sigma'}(z)) = s_\Sigma(\rho_{\Sigma,\Sigma'}(z))$ for some nonconstant $s_\Sigma \in H^\infty(\Sigma)$.
Define now $\Lambda_0: \Omega \to \mathbb C^n$ by $\Lambda_0 = [\hat{h}, 0,\cdots, 0]^T$ as a column $n$-vector of bounded holomorphic functions, hence $\Lambda_0 \in \mathscr F^n$. Since $\hat{h}$ and hence $\Lambda_0:\Omega \to \mathbb C^n$ is $H$-invariant, by Proposition \ref{group-invariance} we have $g^*\Lambda_0 \in \mathscr F^n$ for every $g \in G_0$.  By our choice of the function $\hat{h}$, for some $i \in \{1,\cdots,n\}$ we have $\frac{\partial \hat{h}}{\partial z_i}(0) \neq 0$.  There exists an elementary transformation $A \in {\rm End}(\mathbb C^n)$ having the effect of switching rows such that, writing $\Lambda = A\Lambda_0 \in \mathscr F^n$, we have ${\rm Tr}(d\Lambda(0)) \neq 0$.  Since $A$ is constant linear and $\Lambda_0$ is $H$-invariant, $\Lambda$ is likewise $H$-invariant. Now for $z \in \Omega$ define
\[
\widetilde{\Lambda}(z):=\int_{K} k\Lambda(k^{-1}z)d\mu(k)\, .
\]
Since by Proposition \ref{group-invariance} $g^*\Lambda \in \mathscr F^n$ for any element $g\in G_0$, the vector-valued function $\Lambda(k^{-1}z)$ belongs to $\mathscr F^n$ whenever $k \in K$, so that $\widetilde{\Lambda} \in \mathscr F^n$ by Lemma \ref{invariant-alg}.  Now $0\in\Omega$ is fixed by $K$.  For $k \in K$, writing $\Lambda_k(z) = k\Lambda(k^{-1}z)$ we have $d\Lambda_k(0) = kd\Lambda(0)k^{-1}$ so that ${\rm Tr}(d\Lambda_k(0)) = {\rm Tr}(d\Lambda(0)) \neq 0$, hence ${\rm Tr}(d\widetilde\Lambda(0)) \neq 0$. Consider the circle group $\mathbb S^1$ acting on $\mathbb C^n$ by $(e^{i\theta};z_1,\cdots,z_n) \mapsto (e^{i\theta}z_1,\cdots,e^{i\theta}z_n)$.  Then, $\mathbb S^1 \subset K$ is the center of $K$.  Hence $\widetilde{\Lambda}$ is invariant under conjugation by the circle group, and must therefore be a linear function by Cartan's theorem.
Since $K$ acts irreducibly as a group of unitary transformations on $\mathbb C^n$, it follows from ${\rm Tr}(d\widetilde\Lambda(0)) \neq 0$ and the Schur lemma that  $\widetilde{\Lambda} =\frac{1}{c} \id_\Omega$ for some constant $c\neq 0$. Hence, $\frac{1}{c} \id_\Omega=F^*\mu = F^*[\mu_1,\dots,\mu_n]^T$, where $\mu_i\in H^\infty(\widetilde{N})$ for $1 \le i \le n$. Finally, defining $R:=c\mu:\widetilde{N}\rightarrow \mathbb{C}^n$ we have $R\circ F \equiv \id_{\Omega}$, completing the proof of Theorem \ref{extension} in the case where $\Omega$ is irreducible.  The general case is similar.
\end{proof}

\vskip 0.2cm
\subsection{Proof of the Isomorphism Theorem}
\label{isom-proof-sec}
\mbox{}

\vskip 0.5cm
\subsubsection{Comparing canonical metrics and canonical volume forms on a complete K\"ahler-Einstein manifold}
\label{compare-metrics}
In the proof of Theorem \ref{isom} and a variation replacing the target manifold of $F: \Omega \to M$ by a target domain $D \Subset \mathbb C^n$, which was covered in the statement of the Isomorphism Theorem in Mok-Wong \cite{MW2025}, we will need some well-known intrinsic estimates.
Apart from the (infinitesimal) Carath\'eodory pseudometric \eqref{cara},
for the reader's reference, we give here the Kobayashi-Eisenmann pseudo-volume form on a complex manifold $M$, as follows.

\begin{definition}$(${\rm Kobayashi-Eisenmann volume form}$)$
Let $n$ be a positive integer and $M$ be an $n$-dimensional complex manifold. The Kobayashi-Eisenmann pseudo-volume form $d\mu_M$ is defined by
$$
\|\xi\|_{d\mu_M} := {\rm inf}\big\{\|\alpha\|_{{\rm det}(g_{\mathbb B^n})}: f: \mathbb B^n\overset{hol}\longrightarrow M \ {\rm and} \ \det(df)(\alpha) = \xi \big\}
$$
for $\xi\in \Lambda^n T_M$.
\label{Eisenmann}
\end{definition}
We have
\begin{lemma}{\rm (the comparison lemma)} \
Let $n$ be a positive integer and $(M,g_M)$ be an $n$-dimensional complete K\"ahler-Einstein manifold of
constant Ricci curvature $-(n+1)$, and denote by
$\det(g_M)$ its volume form. Then, for the infinitesimal Carath\'eodory
metric  $\kappa_M$ and the Kobayashi-Eisenmann volume form
$\mu_M$ on $M$, we have
$$
g_M\ge \frac{2\kappa_M}{n+1}\ ,\quad {\rm det}(g_M) \le \mu_M\! .
$$
\label{comp}
\end{lemma}

\vskip 0.2cm
\subsubsection{Inducing a holomorphic map \texorpdfstring{$\rho: Y_{\Gamma'}\to X_\Gamma$}{rho: Y to X} by descent}
\label{induce-rho}
In the notation of Theorem \ref{extension}, we have $R: \widetilde{N} \to \mathbb C^n$ such that $R\circ F\equiv \id_{\Omega}$. Now we specialize to the situation of Theorem \ref{isom} where $\widetilde N = M$ is equipped with a complete K\"ahler-Einstein metric of negative Ricci curvature.
Writing $g_M := h_M$ for the metric in the statement of Theorem \ref{isom} and $g_{Y_{\Gamma'}} := h_{Y_{\Gamma'}}$ for its quotient metric on $Y_{\Gamma'} = M/\Gamma'$, we have $N = Y_{\Gamma'}$ for some torsion-free discrete subgroup $\Gamma' \subset \Aut(M)$.  Since the complete K\"ahler-Einstein metric $g_M$ is invariant under $\Aut(M)$, it descends to the complete K\"ahler-Einstein metric $g_{Y_{\Gamma'}}$.  In the hypothesis of Theorem \ref{isom}, it is further assumed that $M$ is Carath\'eodory hyperbolic and ${\rm Volume}\left(Y_{\Gamma'},g_{Y_{\Gamma'}}\right) < \infty$, and that
$f_*: \Gamma = \pi_1(X_\Gamma) \overset{\cong}\longrightarrow \pi_1(Y_{\Gamma'})\cong \Gamma' $
is a group isomorphism on fundamental groups. From now on we identify $\Gamma$ with $\Gamma'$ via $f_*$. To proceed further we will resort to the following lemma of Mok-Wong \cite{MW2025} about bounded plurisubharmonic functions on a complete K\"ahler manifold of finite volume, for which we reproduce a more condensed version of the proof.

\begin{lemma}[{Mok-Wong \cite[{Lemma 8.1}]{MW2025}}]
Let $(Z,g)$ be a complete K\"ahler manifold of finite volume, and $u$ be a nonnegative uniformly Lipschitz bounded plurisubharmonic function on $Z$.  Then, $u$ is a constant function.
\label{lipz}
\end{lemma}

\begin{proof} Fix a base point $z_0 \in Z$.  For $R > 0$
write $B_R$ for the geodesic ball $B(z_0;R)$ on $(Z,g)$.
There exists a smooth nonnegative function $\rho_{_R}$
on $Z$, $0 \leq \rho_{_R} \leq 1$, such that $\rho_{_R} \equiv 1$ on $B_R$,
$\rho_{_R} \equiv 0$ outside $B_{R+1}$, so that ${\rm Supp}(d\rho_{_R}) \subset B_{R+1} - \overline{B}_R$,
and such that $\|d\rho_{_R}\| \leq 2$. Writing $s:=\dim_\mathbb{C}Z$ we have
\begin{equation}
0 = \int_Z \sqrt{-1}d(\rho_{_R} u) \wedge \overline{\partial}u \wedge
\omega^{s-1} + \int_Z \rho_{_R} \sqrt{-1}
u\partial\overline{\partial}u \wedge \omega^{s-1}\ .
\end{equation}
Here $\sqrt{-1} \partial\overline{\partial}u \ge 0$ in the sense of currents, hence it has coefficients which are complex measures when expressed in local coordinates, and $\sqrt{-1}
u\partial\overline{\partial}u$ is well-defined since $u$ is a bounded function. We have
\begin{equation}
\gathered
\int_{B_R} \sqrt{-1}\partial u\wedge \overline{\partial}u \wedge \omega^{s-1}
\leq \int_Z \rho_{_R} \sqrt{-1}\partial u\wedge \overline{\partial}u \wedge \omega^{s-1}\\
\qquad
= \ - \int_Z \sqrt{-1}u\partial\rho_{_R}\wedge \overline{\partial}u \wedge \omega^{s-1}
- \int_Z \rho_{_R} \sqrt{-1} u\partial\overline{\partial}u \wedge \omega^{s-1}\ .
\endgathered
\label{lipest2}
\end{equation}
Here and in what follows, $\|\cdot\|$ will denote norms on $(Z,g)$.  By assumption $\|du\|$ is uniformly bounded.  Furthermore, $\|d\rho_{_R}\| \leq 2$, and its support is contained in $\overline{B}_{R+1} - B_R$,
so that the second last term of \eqref{lipest2} is bounded by a constant multiple of ${\rm Volume}\left(\overline{B}_{R+1} - B_R,g\right)$ $\to 0$ as $R\to\infty$, since Volume$(Z,g)<\infty$. On the other hand the last integral is nonnegative since $u \geq 0$ and $u$ is plurisubharmonic. Hence, for some constants $C_1, C_2 > 0$ we have
\begin{equation}
\gathered
\int_{B_{R}} \|\partial u\|^2
\le C_1 \left |\int_Z u\partial\rho_{_R}\wedge \overline{\partial}u \wedge \omega^{s-1}\right |
\le C_2 {\rm Volume}\Big(B_{R+1}-\overline{B}_R,\omega \Big)
 \to 0
\endgathered
\end{equation}
as $R \to \infty$.  Hence, $\partial u \equiv 0$ and also $du \equiv 0$, since $u$ is real-valued. In other words, $u$ is a constant function, as desired.
\end{proof}

\vskip 0.2cm
Denote by $\tau: M \to Y_{\Gamma'}$ the universal covering map.  For any bounded holomorphic function $\theta$ on $M$, consider the function $\psi_{\theta}: Y_{\Gamma'} \to \mathbb R$ defined by $\psi_{\theta}(q) = {\rm sup}\{|\theta(p)|^2: \tau(p) = q\}$.  Defining $\widetilde{\psi}_\theta: M \to \mathbb R$ by $\widetilde{\psi}_\theta(p) = \psi_\theta(\tau(p))$, we have
\begin{equation}
\widetilde{\psi}_\theta(p) = {\rm sup}\left\{|\theta(\gamma(p))|^2: \gamma \in \Gamma'\right\} =
{\rm sup}\left\{|(\theta\circ\gamma)(p)|^2: \gamma \in \Gamma'\right\}\/ .
\label{tphi}
\end{equation}
The ensuing results can be found in Mok-Wong \cite{MW2025} but the setting has changed as we are considering target spaces $Y_{\Gamma'}$ of $f:X_{\Gamma} \to Y_{\Gamma'}$ which are complete K\"ahler-Einstein manifolds of finite volume (instead of $Y_{\Gamma'}$ being covered by a bounded domain $D$).  For this reason and in view of the importance of these results we adapt here the arguments to the current situation with the necessary changes.  Using Lemma \ref{lipz} we will deduce

\begin{lemma} \
For any bounded holomorphic function $\theta$ on $M$ the plurisubharmonic function $\psi_\theta$ is a constant function on $Y_{\Gamma'}$
\label{psh}
\end{lemma}

\begin{proof}
The plurisubharmonic function $\widetilde{\psi}_\theta$ is the supremum of $|\theta_\gamma|^2$, $\gamma \in \Gamma'$,
on a uniformly bounded family of holomorphic functions $\theta_\gamma: M \to \Delta$, $\theta_\gamma := \theta\circ\gamma$.
We claim that the family of differential 1-forms $\left\{d\theta_\gamma\right\}$ are uniformly bounded on $M$ with respect to the K\"ahler-Einstein metric $h_M$. Without loss of generality we may assume that $\theta: M \to \Delta\big(\frac{1}{\phantom{,}2\phantom{,}}\big)$.
Now for $\gamma \in \Gamma'\cong\Gamma$, $\theta_\gamma \in H^\infty(M)$ and $\theta_\gamma(M) \subset \Delta(\frac{1}{2})$.
For any tangent vector $\eta \in T_M$ we have $(\dagger)$ $\|d\theta_\gamma(\eta)\|_{\rm Poin} \le \|\eta\|_{\kappa_M} \le c_1\|\eta\|_{h_M}$ for some universal positive constant $c_1$,
where the first inequality follows from the definition of the infinitesimal Carath\'eodory metric \eqref{cara}, and the second inequality follows from the comparison lemma (Lemma \ref{comp}) after normalizing the Ricci curvature of $h_M$ to be $-(n+1)$ if necessary.
Let $\mu \in \Aut(\Delta)$ such that $\mu(\theta_\gamma(x)) = 0$.
For any $(1,0)$-tangent vector $\xi:=\alpha\frac{\partial}{\partial z}$ at $0\in\Delta$ we write $|\xi| := |\alpha|$. Note that there is a universal positive constant $c_2$ such
that $\|\xi\|_{\rm Poin}= c_2|\xi|$.
Hence by the $\Aut(\Delta)$-invariance of the Poincar\'e metric, we have $\|d\theta_\gamma(\eta)\|_{\rm Poin} = c_2\bigl|d((\mu\circ\theta_\gamma)(\eta))\bigr|$.
Note that whenever $w\in M$, we have $|\theta_\gamma(w)| < \frac{1}{\phantom{,}2\phantom{,}}$.  Hence, while $\mu \in \Aut(\Delta)$ depends on $y = \theta_\gamma(w)$, the point $y$ stays on a relatively compact subset of $\Delta$, so that the choices of $\mu$ stay within a relatively compact subset of $\Aut(\Delta)$.
The upshot is that, in order to show that the family of differential 1-forms $\left\{d\theta_\gamma\right\}$ are uniformly bounded
with bounds independent of $\gamma\in\Gamma'$ on $M$ with respect to the K\"ahler-Einstein metric $h_M$, it suffices to show that $\|d\theta_\gamma(\eta)\|_{\rm Poin}$ for $\|\eta\|_{h_M} = 1$ are uniformly bounded independent of $\gamma \in \Gamma'$, which has been established in $(\dagger)$.
(The rewriting via $\mu$ above shows equivalently that the Euclidean sizes of the derivatives of the recentered maps $\mu\circ\theta_\gamma$ are likewise uniformly controlled.)
Equivalently, we have proven that $\{\theta_\gamma: \gamma\in \Gamma'\}$ are uniformly bounded and uniformly Lipschitz.

\vskip 0.2cm
Finally, for $\gamma\in\Gamma'$,
$d|\theta_\gamma|^2 = d(\theta_\gamma\overline{\theta_\gamma}) = \theta_\gamma\overline{\partial\theta_\gamma} + \overline{\theta_\gamma}\partial\theta_\gamma$. Since $|\theta_\gamma|\le \frac12$ and the differentials $d\theta_\gamma$ are uniformly bounded with respect to $h_M$ as above, and since the Poincar\'e and Euclidean metrics are comparable on $\Delta(\frac12)$, there exists a constant $C > 0$
such that for any $\gamma\in\Gamma'$, $|\theta_\gamma(w)|^2 \le C$ for all $w\in M$ and
$\|d|\theta_{\gamma}|^2(w)\|_{h_M} \le C$ for all $w\in M$.  As is well known, the supremum function $\widetilde{\psi}_\theta$ of such a family of functions is uniformly bounded and uniformly Lipschitz with the same function bound and Lipschitz constant.  Since by assumption the complete K\"ahler-Einstein $\left(Y_{\Gamma'},g_{Y_{\Gamma'}}\right)$ is of finite volume, by Lemma \ref{lipz}, $\psi_\theta$ is a constant function on $Y_{\Gamma'}$, as desired.
\end{proof}
\begin{proposition} \
In the notation of Theorem \ref{isom} $($the Isomorphism Theorem$)$, the image of the bounded holomorphic mapping $R: M\to \mathbb C^n$ lies on $\Omega$ and it is furthermore invariant under $\Gamma' \equiv \Gamma$, i.e., $R(\gamma w) = \gamma(R(w))$ so that $R$ descends to $\rho: Y_{\Gamma'} \to X_\Gamma$ satisfying $\rho\circ f\equiv \id_{X_\Gamma}$.
\label{descend}
\end{proposition}

\begin{proof} \
In what follows we assume for simplicity that $\Omega$ is irreducible (and of rank $\ge 2$), the general case being similar. We claim first of all that $R(M) \subset \Omega$.

\vskip 0.2cm
Let $\alpha \in T_0(\Omega)$ be a minimal rational tangent vector of unit length at the origin $0$, $\Delta_{\alpha} \subset \Omega$ be the minimal disk passing through $0$ such that $\alpha \in T_0(\Delta_{\alpha})$. Let $L_{\alpha}: \mathbb C^n \to \mathbb C\alpha$ be the Euclidean orthogonal projection, which projects $\Omega$ onto $\Delta_{\alpha}$.  We identify $\mathbb C\alpha$ isometrically with $\mathbb C$ and hence $\Delta_{\alpha}$ with $\Delta$.  Applying the same argument now to the function $\theta_\alpha := L_{\alpha}\circ R$ we conclude that  $\psi_{\theta_\alpha}$ must be identically equal to 1, since $\psi_{\theta_{\alpha}}(p) = 1$ for any $p \in f(X_\Gamma)$.  Taking all possible minimal rational tangent vectors $\alpha$ of unit length at
$0$ one concludes readily that $R(M) \subset \overline{\Omega}$, as can be seen for instance from the
Polydisk theorem.
It remains to show that $R(M) \cap\partial\Omega=\emptyset$. Identifying $\Omega$ as an open subset of $T_0(\Omega)$ by the Harish-Chandra embedding, we have $\overline{\Omega}=\{\eta\in T_0(\Omega):\|\eta\|_{\kappa_\Omega}\le 1\}$ for the Carath\'eodory metric $\kappa_\Omega$ on $\Omega$.  If $R(M)\cap\partial\Omega\ne\emptyset$, then the plurisubharmonic function $\varphi(w) := \|R(w)\|_{\kappa_\Omega}$ on $M$ attains its maximum value $1$,
and must therefore be identically equal to $1$ by the maximum principle,
so that $R(M)\subset \partial\Omega$. But this contradicts the fact that $R(w)\in\Omega$ whenever $w \in F(\Omega)$, proving
$R(M)\subset \Omega$ by contradiction, as desired.

\vskip 0.2cm
It remains to check the $\Gamma'$-invariance, $\Gamma'\cong\Gamma$, of the holomorphic mapping $R:M\to \Omega$, i.e., $R(\gamma w) = \gamma(R(w))$ for every point $w \in M$.  Define $\widetilde{\lambda}(w) := {\rm sup}\left\{\|R(\gamma w) - \gamma(R(w))\|^2: \gamma\in \Gamma'\right \}$.  From the definition $\widetilde{\lambda}$ is $\Gamma'$-invariant and it descends to a function $\lambda$ on $Y_{\Gamma'}$.  Each function $\|R(\gamma w) - \gamma(R(w))\|^2$ is plurisubharmonic, bounded and uniformly Lipschitz with respect to the K\"ahler-Einstein $h_M$ with both function bounds and the Lipschitz constants independent of $\gamma\in\Gamma'$.   By the same application of Lemma \ref{lipz}, we conclude that $\lambda \equiv c$ for some nonnegative constant $c$.  On the other hand, $\|R(\gamma w) - \gamma(R(w))\|^2 = 0$ for any $w \in F(\Omega)$, hence $c = 0$, as desired, proving the proposition.
\end{proof}

\vskip 0.2cm
\subsubsection{Discreteness of the general fiber of the retraction map \texorpdfstring{$\rho: Y_{\Gamma'}\to X_\Gamma$}{rho: Y to X-Gamma} and completion of the proof of the Isomorphism Theorem}
\label{discrete-fiber}
To prove Theorem \ref{isom}, the essential thing is to prove that $\rho:Y_{\Gamma'}\to X_{\Gamma}$ is a local biholomorphism at a general point, i.e., the general fiber of the holomorphic retraction map $\rho$ is discrete. For a detailed proof of the latter fact using coordinates we refer the reader to Mok-Wong \cite{MW2025}. In the following we give a sketch of the completion of the proof of Theorem \ref{isom}.

\begin{proof}
We proceed with
\begin{quote}
\underline{Claim}: The general fiber of $\rho: Y_{\Gamma'} \to X_\Gamma$ is discrete.
\end{quote}
\textit{Proof of Claim}: Suppose to the contrary that a general fiber of $\rho: Y_{\Gamma'} \to X_\Gamma$ is positive-dimensional.
By the Carath\'eodory hyperbolicity of $M$ in the hypothesis of Theorem \ref{isom}, $Y_{\Gamma'}$ admits  nontrivial bounded holomorphic functions descending from $M$.
From the definition of the Carath\'eodory metric it follows that $R^*\kappa_\Omega \le \kappa_M$.  Write $m := \dim_{\mathbb C}M$.
By Lemma \ref{comp} we have $g_{M} \ge \frac{2\kappa_M}{m+1}$, hence $\kappa_M\le \frac{m+1}{2}g_M$, yielding $R^*\kappa_\Omega \le \kappa_M \le \frac{m+1}{2}g_M$. Since $\Omega$ is homogeneous, any two invariant metrics on $\Omega$ are equivalent.  Hence, $g_\Omega \le C_1\kappa_\Omega$ for some positive constant $C_1$.  As a consequence, for $C:=\frac{C_1(m+1)}{2}$ we have $R^*g_\Omega \le C_1 R^*\kappa_\Omega \le C_1\kappa_M \le Cg_M$, hence $R^*\omega_\Omega \le C\omega_M$, which descends to $\rho^*\omega_{X_\Gamma} \le C\omega_{Y_{\Gamma'}}$.
We have
\begin{align*}
\infty &> {\rm Volume}\left(Y_{\Gamma'},g_{Y_{\Gamma'}}\right)
= \frac{1}{m!}\int_{Y_{\Gamma'}} \left(\omega_{Y_{\Gamma'}}\right)^m  \\
&\ge \frac{1}{m!\cdot C^{n}}\int_{Y_{\Gamma'}}
\left(\omega_{Y_{\Gamma'}}\right)^{m-n}\wedge\left(\rho^*\omega_{X_\Gamma}\right)^{n} \\
& = \frac{1}{m!\cdot C^{n}} \int_{X_\Gamma} \left(\int_{\rho^{-1}(x)} \left(\omega_{Y_{\Gamma'}}\right)^{m-n}\right) \left(\omega_{X_\Gamma}\right)^n\, .
\end{align*}

where the last equality holds by Fubini's theorem in the form of
the coarea formula for the holomorphic map $\rho$, after discarding critical values as below.
By Sard's Theorem, outside a set $E\subset X_{\Gamma}$ of Lebesgue measure 0, the fibers $\rho^{-1}(x)$ are nonsingular and of dimension $m-n$ for all $x \in X_\Gamma - E$, and the integral inside brackets on the right-hand side of the last line is understood as being performed for all $x \in X_\Gamma-E$ while those fibers over $E$ are ignored.  It follows that for almost all $x \in X_\Gamma-E$, the fibers $\rho^{-1}(x)$ are nonsingular and of finite volume with respect to $g_{Y_{\Gamma'}}|_{\rho^{-1}(x)}$
(up to the usual factorial normalization identifying $\frac{1}{(m-n)!}(\omega_{Y_{\Gamma'}})^{m-n}\big|_{\rho^{-1}(x)}$ with the volume form of the induced metric).
\vskip 0.2cm
Pick such a fiber $\rho^{-1}(x)\subset Y_{\Gamma'}$ and let $L_x \subset \rho^{-1}(x)$ be a connected component. Obviously ${\rm Volume}\left( L_x,(g_{Y_{\Gamma'}})\vert_{L_{x}}\right) < \infty$. With respect to the embedding $L_x \hra Y_{\Gamma'}$ given by inclusion, the image of $\pi_1(L_x)$ in $\pi_1(Y_{\Gamma'})$ must be trivial.
In fact, from $\rho\circ f = \id_{X_\Gamma}$ and the hypothesis $f_*: \Gamma=\pi_1(X_\Gamma) \overset{\cong}\longrightarrow \pi_1(Y_{\Gamma'})\cong\Gamma'$ is a group isomorphism, it follows that $\rho: Y_{\Gamma'} \to X_\Gamma$ induces an isomorphism $\rho_*: \Gamma'\overset{\cong}\longrightarrow\Gamma$, hence $\rho_*(\pi_1(L_x))$ is trivial since $\rho(L_x)$ is collapsed to a point.
Recall that $\pi: \Omega \to X_\Gamma$
and $\tau: M\rightarrow Y_{\Gamma'}$ are
the universal covering maps.
Let $\widetilde{x} \in \Omega$ be such that $\pi(\widetilde{x}) = x$. Then, there exists an irreducible component $\widetilde{L}_x$ of $R^{-1}(\widetilde x)\subset M$ such that $\tau|_{\widetilde{L}_x}: {\widetilde{L}_x}\overset{\cong}\longrightarrow L_x$ is a biholomorphism.

\vskip 0.3cm
Now by assumption the complex manifold $M$ is Carath\'eodory hyperbolic, hence there exists $\theta\in H^\infty(M)$ such that $\theta|_{\widetilde{L}_x}$ is nonconstant.
We may regard $\theta$ as a holomorphic function on $L_x$ via the identification $\tau|_{\widetilde{L}_x}: {\widetilde{L}_x}\overset{\cong}\longrightarrow L_x$,
which contradicts Lemma \ref{lipz}, proving by contradiction the claim that a general fiber of $\rho: Y_{\Gamma'} \to X_\Gamma$ is discrete.
\vskip 0.3cm
From $\rho\circ f\equiv \id_{X_\Gamma}$, we know that $f: X_\Gamma \to Y_{\Gamma'}$ is an open embedding onto a connected open subset $\mathcal O = f(X_\Gamma) \subset Y_{\Gamma'}$.
For any $y=f(x)\in \mathcal O$,
note that $f\circ \rho(y)=f((\rho\circ f)(x))=f(x)=y $. Therefore $f\circ \rho$ is the identity map on the connected open subset $\mathcal O \subset Y_{\Gamma'}$. It follows from the identity theorem for holomorphic functions that $f\circ \rho$ is the identity map on all of $Y_{\Gamma'}$.
Hence $f: X_\Gamma \overset{\cong}\longrightarrow Y_{\Gamma'}$ is a biholomorphism, completing the proof of Theorem \ref{isom}.

\end{proof}

\vskip 0.2cm
\subsection{The Isomorphism Theorem on \texorpdfstring{$\Gamma$}{Gamma}-equivariant holomorphic maps into arbitrary bounded domains \texorpdfstring{$D\Subset\mathbb C^n$}{D subset complex n-space}}
\label{isom-domain-sec}

We state here the Isomorphism Theorem of Mok-Wong \cite{MW2025} except that
we restrict it to the case of bounded domains $D\Subset \mathbb{C}^n$ instead of the more general case of bounded domains on Stein manifolds.

\vskip 0.2cm
\subsubsection{Formulation of the Isomorphism Theorem for arbitrary covering bounded domains as target manifolds}
\label{isom-domain-form}
\begin{theorem} \
Let $\Omega$ be a bounded symmetric domain of rank $\geq 2$ and $\Gamma$ be a torsion-free irreducible lattice on $\Omega$; $X_\Gamma := \Omega/\Gamma$.  Let $D$ be a bounded domain, $\Gamma' \subset$ $\Aut(D)$ be a torsion-free discrete subgroup such that $Y_{\Gamma'} := D/\Gamma'$ is of finite volume with respect to the Kobayashi-Eisenmann volume form $d\mu_{Y_{\Gamma'}}$.
Let $f: X_\Gamma \rightarrow Y_{\Gamma'}$ be a holomorphic map inducing an isomorphism
$f_*: \Gamma \overset\cong\longrightarrow \Gamma'$. Then, $f: X_\Gamma \overset\cong\longrightarrow Y_{\Gamma'}$ is a biholomorphism.
\label{isom-domain}
\end{theorem}

\begin{remark} \
{\rm
We note that a first instance of Theorem \ref{isom-domain} was stated with a complete proof in the special case where $\Omega=\Delta^n, n\geq 2$, in Mok \cite{Mok2007}.
}
\end{remark}

\vskip 0.2cm
\subsubsection{The roles of Oka's Theorem and Kobayashi-Eisenmann volume forms from the function theory of several complex variables}
\label{oka-roles}
On a bounded plane domain $D \Subset \mathbb C$ consider the Poincar\'e metric $g_D$ inherited from the Poincar\'e metric on the universal covering disk $\Delta$.  Then, by the Schwarz lemma one can get a lower bound of the Poincar\'e metric $g_D$ by comparing it to punctured disks containing $D$.  In dimension $n \ge 2$ the analogue is the canonical K\"ahler-Einstein metric on $D$.  Cheng-Yau proved in \cite{CY1980} the existence and uniqueness of the canonical K\"ahler-Einstein metric on a strictly pseudoconvex domain by studying the complex Monge-Amp\`ere equation.
By exhausting a bounded domain of holomorphy by an increasing sequence of strictly pseudoconvex domains, they established the existence and uniqueness of an almost complete K\"ahler-Einstein metric of constant Ricci curvature $-(n+1)$.

\vskip 0.2cm
In general, a bounded domain in dimension $n \ge 2$ does not necessarily admit a complete K\"ahler-Einstein metric of negative Ricci curvature.  It was in fact proven in Mok-Yau \cite{MY1983} that the existence of such a K\"ahler metric forces $D$ to satisfy the Kontinuit\"atssatz for a sequence of holomorphic disks
continuous up to the boundary,
viz., given a sequence of continuous maps $f_n: \overline{\Delta} \to D$ holomorphic on $\Delta$ converging uniformly to a continuous map $f: \overline{\Delta} \to \mathbb C^n$ such that $f(\partial\Delta) \subset D$,
we must have $f(\overline{\Delta}) \subset D$. Equivalently this means that $D \subset \mathbb C^n$ is a domain of holomorphy, by a theorem of Oka.
Another equivalent characterization of domains of holomorphy is Oka's theorem (cf.\, Oka \cite{Oka1942, Oka1953}) that a domain $D \subset \mathbb C^n$, not necessarily bounded, is a domain of holomorphy if and only if $-\log\delta$ is a plurisubharmonic function on $D$.  Here $\delta(x)$ is the Euclidean distance to $\partial D$, $-\log\delta$ is a continuous function, and it is plurisubharmonic if and only if it satisfies locally the mean value inequality in terms of a local holomorphic coordinate when restricted to any local complex 1-dimensional submanifold.

\vskip 0.2cm
Given any bounded domain of holomorphy, using Oka's theorem that $\sqrt{-1}\partial\overline{\partial}(-\log\delta) > 0$ in the sense of currents and Yau's Schwarz lemma for volume forms, it was established in \cite{MY1983} that the volume form of the almost K\"ahler-Einstein metric $g_D$ on $D$ has a lower bound that resembles that of a punctured disk in the case of bounded
plane domains.  Expressing $g_D(x)$ as a Hermitian matrix $\left(g_{\alpha\overline{\beta}}(x)\right)$ in terms of Euclidean coordinates and assuming $D \subset \mathbb B^n\big(0;\frac{1}{\phantom{.}2\phantom{.}}\big)$ for convenience, we have a lower bound $(\dagger)$ $V_D := {\rm det}\left(g_{\alpha\overline{\beta}}\right) \ge \frac{c_n}{\delta^2(-\log\delta)^2}$ for some universal positive constant $c_n$ depending only on the dimension $n$.  The rest of the arguments follow Cheng-Yau \cite{CY1980}. As in \cite{CY1980} the gradient estimate of Yau \cite{Yau1975} for eigensections of the Laplacian on a complete Riemannian manifold of Ricci curvature bounded from below applies to $\log V_D$, since $\Delta\log V_D = n(n+1)$ for the Laplacian $\Delta$ of the complete K\"ahler-Einstein manifold $(D,g_D)$, to give $\|d\log{\rm det}\left(g_{\alpha\overline{\beta}}\right)\|_{g_D} \le C_n$ for some universal positive constant $C_n$ depending only on $n$, in the case of $D\subset\mathbb B^n\big(0;\frac{1}{\phantom{.}2\phantom{.}}\big)$ strictly pseudoconvex, in which case $g_D$ is known to be complete. Joining a point $x_0\in D$ to a point $x$ by a geodesic $\gamma$ and integrating along $\gamma$ to get $\log V_D(x) - \log V_D(x_0) \le C_n d(x_0,x)$, one gets a universal bound $d(x_0,x) \ge \frac{1}{C_n}\left(\log V_D(x) - \log V_D(x_0)\right)$.  One gets the same estimate for a bounded domain of holomorphy $D \subset \mathbb B^n\big(0;\frac{1}{\phantom{.}2\phantom{.}}\big)$ by exhausting $D$ by an increasing sequence of strictly pseudoconvex domains $\{D_n\}_{n\ge 1}$, $D_n \Subset D_{n+1}$ and $\bigcup_{n\ge 1} D_n = D$, which gives the completeness of $(D,g_D)$.  Thus, we can consider a bounded domain of holomorphy $D$ equipped with the canonical complete K\"ahler-Einstein metric $g_D$ of constant Ricci curvature $-(n+1)$.  Then, $(D,g_D)$ can take the place of $(M,g_M)$ in Theorem \ref{isom}, except that we do not assume that $D$ is simply connected. We remark that Theorem \ref{isom} is equally valid for $\Gamma' \subset \Aut(M)$, $Y_{\Gamma'} := M/\Gamma'$ without assuming that $M$ is simply connected.

\vskip 0.2cm
We will need the following estimate for the Kobayashi-Eisenmann volume form $\mu_U$ on a bounded domain
$U \Subset \mathbb C^n$ in terms of distances to the boundary.
\begin{proposition} \
Let $U \Subset \mathbb C^n$ be a bounded domain, and denote by $\mu_U$ the
Kobayashi-Eisenmann volume form on $U$.  For $z \in U$, denote by $\delta(z)$
the Euclidean distance of $z$ from the boundary $\partial U$.  Write
$dV$ for the Euclidean volume form on $\mathbb C^n$.  Then, there exists
a positive constant $c$ depending only on $n$ and the diameter of
$U$ such that
\[
\mu_U(z) \ge \frac{c}{\phantom{.}\delta(z)\phantom{.}}\,dV\,\! .
\]
\label{Kest}
\end{proposition}
The proof of the proposition as given in \cite{MW2025} is elementary and relies on the Cauchy estimate for first derivatives of holomorphic functions.

\vskip 0.2cm
\subsubsection{Proof of Theorem \ref{isom-domain}}
\label{isom-domain-proof}
In what follows we give a sketch of the extra steps for the proof of Theorem \ref{isom-domain}.

\begin{proof}
Suppose for the time being that $D \Subset \mathbb C^m$ is a bounded domain of holomorphy.  By Mok-Yau \cite{MY1983}, there exists on $D$ a unique canonical complete K\"ahler-Einstein metric $g_D$ of Ricci curvature $-(m+1)$. By hypothesis, for the Kobayashi-Eisenmann volume form $d\mu_{Y_\Gamma'}$ on $Y_{\Gamma'} = D/\Gamma'$, ${\rm Volume}\left(Y_{\Gamma'},d\mu_{Y_\Gamma'}\right) < \infty$. By Lemma \ref{comp}, we have $\det{g_{Y_\Gamma'}} \le d\mu_{Y_\Gamma'}$, so that ${\rm Volume}\left(Y_{\Gamma'},g_{Y_\Gamma'}\right) < \infty$.  Hence the proof of Theorem \ref{isom} applies to $(D,g_D)$ to prove that $f: X_\Gamma \overset{\cong}\longrightarrow Y_{\Gamma'}$ is a biholomorphism. (Here we have $Y_{\Gamma'} = D/\Gamma'$ for $\Gamma' \subset \Aut(D)$ a torsion-free lattice such that  ${\rm Volume}\left(Y_{\Gamma'},d\mu_{Y_\Gamma'}\right) < \infty$,
a setting which fits into Theorem \ref{isom}, except that we do not assume $D$ to be simply connected.)
Nonetheless, the proof works {\it verbatim\/} without assuming that $D$ is simply connected.
\vskip 0.2cm
In general, we consider the hull of holomorphy $\varphi: \widehat D \to \mathbb C^N$, which is {\it a priori\/}
a Riemann domain. The hypothesis in Theorem \ref{isom-domain}
says that ${\rm Volume}\left(Y_{\Gamma'},d\mu_{Y_\Gamma'}\right) < \infty$ for the Kobayashi-Eisenmann volume form $d\mu_{Y_\Gamma'}$.
The action of the torsion-free discrete subgroup $\Gamma' \subset \Aut(D)$ extends to an action of $\Gamma'$ on $\widehat{D}$,
defining a torsion-free discrete subgroup $\widehat{\Gamma}' \subset \Aut(\widehat{D})$, $\widehat{\Gamma}'$ being naturally isomorphic to $\Gamma'$, such that $\widehat{Y}_{\widehat{\Gamma}'}:= \widehat{D}/\widehat{\Gamma}'$ contains $Y_{\Gamma'}$ as a
connected open subset and $\widehat{Y}_{\widehat{\Gamma}'}- Y_{\Gamma'} \subset \widehat{Y}_{\widehat{\Gamma}'}$ is of zero Lebesgue measure, hence ${\rm Volume}\big(\widehat{Y}_{\widehat{\Gamma}'},d\mu_{\widehat{Y}_{\widehat{\Gamma}'}}\big) < \infty$, so that by Lemma \ref{comp} we have ${\rm Volume}\big(\widehat{Y}_{\widehat{\Gamma}'},g_{\widehat{Y}_{\widehat{\Gamma}'}}\big) < \infty$.
The arguments of Theorem \ref{isom} then produce a holomorphic retraction map $\widehat{\rho}: \widehat{Y}_{{\widehat{\Gamma}'}}\to X_\Gamma$
such that $\widehat{\rho}\circ f \equiv \id_{X_{\Gamma}}$, which implies that $f: X_\Gamma \overset{\cong}\longrightarrow\widehat{Y}_{{\widehat{\Gamma}'}}$
is a biholomorphism.  But the latter map is nothing other than the composition of $f:X_\Gamma \to Y_{\Gamma'}$ with the natural inclusion $Y_{\Gamma'} \subset \widehat{Y}_{{\widehat{\Gamma}'}}$, and it follows that in fact
$Y_{\Gamma'} = \widehat{Y}_{{\widehat{\Gamma}'}}$ and $f: X_\Gamma \overset{\cong}\longrightarrow Y_{\Gamma'}$ is a biholomorphism.

\vskip 0.2cm
The existence of $\rho: Y_{\Gamma'} \to X_\Gamma$ satisfying $\rho\circ f \equiv \id_{X_\Gamma}$ implies a holomorphic extension to $\widehat{\rho}: \widehat{Y}_{\widehat{\Gamma}'} \to X_\Gamma$ satisfying $\widehat{\rho}\circ f \equiv \id_{X_\Gamma}$.
Using the canonical K\"ahler-Einstein metric on $\widehat{Y}_{\widehat{\Gamma}'}$ we conclude that $f$ is open.
Hence $f\circ \widehat{\rho}\equiv \id_\mathcal{O}$ for $\mathcal{O}=f(X_\Gamma)\subset Y_{\Gamma'}$.
By the identity theorem for holomorphic functions, we have $f: X_\Gamma\overset{\cong}\rightarrow \widehat{Y}_{\widehat{\Gamma}'}$.
Since however $f(X_\Gamma)=\mathcal{O}\subset Y_{\Gamma'}$, we must have $Y_{\Gamma'}=\widehat{Y}_{\widehat{\Gamma}'}$,
proving Theorem \ref{isom-domain}.
\end{proof}

\vskip 0.2cm
\section*{Acknowledgement}
The research of the first author is partially supported by a General Research Grant 17304321 of the HKRGC.

\end{document}